\documentclass[a4paper,11pt]{amsart}
\usepackage{amsmath}  
\usepackage{amssymb}
\usepackage{amsthm}  
\usepackage{amscd}  

\theoremstyle{plain}
\newtheorem{thm}{Theorem}[section]
\newtheorem{prop}[thm]{Proposition}
\newtheorem{lemma}[thm]{Lemma}
\newtheorem{cor}[thm]{Corollary}
\newtheorem*{main}{Main Theorem 1}

\theoremstyle{definition}
\newtheorem{defn}[thm]{Definition}

\theoremstyle{remark}
\newtheorem*{rem}{Remark}
\newtheorem*{exam}{Example}

\numberwithin{equation}{section}

\title[Harmonic map]
{Harmonic maps into Grassmannian manifolds}
\author{Yasuyuki Nagatomo}
\address{Department of Mathematics, MEIJI UNIVERSITY, 
Higashi-Mita, Tama-ku, Kawasaki-shi, Kanagawa 214-8571, JAPAN} 

\email{yasunaga@meiji.ac.jp}

\date{}

\begin{document}

\begin{abstract}

A harmonic map from a Riemannian manifold into a Grassmannian manifold
is characterized by a vector bundle, a space of sections of this bundle 
and a Laplace operator. 
We apply our main theorem (itself a generalization of a theorem of
Takahashi \cite{TTaka})
to  generalize the theory of do Carmo and Wallach \cite{DoC-Wal} and to describe the
moduli space of harmonic maps satisfying the gauge and the
Einstein--Hermitian
conditions from a compact Riemannian manifold into a Grassmannian.
As an application, several rigidity results are exhibited. 
In particular
we generalize a rigidity theorem due to Calabi \cite{Cal} in the case of
holomorphic isometric immersions of compact K\"ahler manifolds into complex 
projective spaces. 
Finally, we also construct moduli spaces of holomorphic
isometric embeddings of the complex projective line into complex quadrics
of low degree.

\end{abstract}

\maketitle

\section{Introduction}

What this paper attempts to do is to bring ideas from the gauge theory of vector bundles 
into the study of harmonic maps. 

One of our main purposes is to generalize the
Theorem of 
Takahashi \cite{TTaka}. 
Let $f:M\to S^{N-1}$ be a smooth map from a Riemannian manifold $(M,g)$ 
into the standard sphere $S^{N-1}$. 
The standard sphere can be considered as a unit sphere  
of a Euclidean space $\mathbf R^N$. 
If we fix an orthonormal basis 
$e_1,\cdots,e_N$ 
of $\mathbf R^N$ and 
the associated coordinates are denoted by 
$(x_1, \cdots, x_N)$, 
then each coordinate function 
$x_A$ ($A=1,\cdots,N$) can be regarded as a function on $S^{N-1}$ 
by restriction. 
We can pull-back each $x_A$ by $f:M\to S^{N-1}$ 
to obtain a function on $M$, which is also denoted by the same symbol. 
Then (a version of) the Theorem of Takahashi states 
\begin{thm}\label{ThmTaka} \cite{TTaka} 
A map $f:M\to S^{N-1}$ is a harmonic map if and only if 
there exists a function $h:M\to \mathbf R$ such that 
$\Delta x_A=hx_A$ for all $A=1,\cdots, N$, 
where $\Delta$ is the Laplace operator of $(M,g)$. 
Under these conditions, we have $h=|df|^2$. 
\end{thm}
In the proof of Theorem \ref{ThmTaka}, the position vector 
$f(x) \in \mathbf R^N$ $(x \in M)$, 
considered as $f:M\to \mathbf R^N$, 
plays a central role. 

First of all, our concern is with a map from $(M,g)$ into a 
real or complex Grassmannian manifold $Gr_p(W)$ 
with a standard metric of Fubini-Study type, 
where $W$ is a real or complex vector space 
with a scalar product. 
Let $S\to Gr_p(W)$ be the tautological bundle. 
Since $S\to Gr_p(W)$ is a subbundle of a trivial bundle 
$\underline{W}=Gr_p(W)\times W \to Gr_p(W)$, 
we have a quotient bundle $Q\to Gr_p(W)$, 
which is called the universal quotient bundle. 
The scalar product on $W$ gives an identification of $Q\to Gr_p(W)$ 
with the orthogonal complement of $S \to Gr_p(W)$ in $\underline{W} \to
Gr_p(W)$. 
Consequently, vector bundles $S,Q \to Gr_p(W)$ are 
equipped with fibre metrics and connections. 

When the standard sphere $S^{N-1}$ 
is identified with the real Grassmannian of 
oriented $(N-1)$-planes in $\mathbf R^N$, 
the position vector $f(x) \in \mathbf R^N$
can be considered as a section of the 
universal quotient bundle $Q\to Gr_{N-1}(\mathbf R^{N})$, 
which is also the normal bundle. 
Then $\mathbf R^N$ induces sections of the bundle and 
we recover functions $x_1,\cdots,x_N$ by relations 
$x_A=\left(e_A,f(x)\right)$, 
where $(,)$ is the inner product on $\mathbf R^N$. 
The differential of the position vector can be 
regarded as the second fundamental form of 
the subbundle $Q \to \underline{\mathbf R^N}$ 
in the sense of Kobayashi \cite{Kob}. 
Moreover, since the bundle has a preferred connection, 
the Laplace operator on the bundle is well-defined. 
Hence we can reformulate the Theorem of Takahashi from the viewpoint of 
vector bundles 
when these geometric structures are pulled back. 
Our first result is 
\begin{main}{\rm (Theorem \ref{main theorem})} 
Let $(M,g)$ be an $n$-dimensional Riemannian manifold and 
$f:M\to Gr_p(W)$ a smooth map.  
We fix an inner product or a Hermitian inner product 
$(\cdot,\cdot)$ on $W$, 
which gives a Riemannian structure on $Gr_p(W)$. 

Then, the following two conditions are equivalent. 
\begin{enumerate}
\item $f:M\to Gr_p(W)$ is a harmonic map. 
\item $W$ has the zero property for the Laplacian. 
\end{enumerate}
Under these conditions, we have for an arbitrary 
$t \in W$, 
$$
\Delta t=-At, \quad \text{and} \quad \left|df\right|^2=-\text{\rm
trace}\,A, 
$$
where the vector space $W$ is regarded as a space of sections of the
pull-back bundle 
$f^{\ast}Q \to M$. 
\end{main}
See Definition \ref{zero} for the zero property for the Laplacian. 
The bundle homomorphism 
$A$ on $f^{\ast}Q \to M$, called the {\it mean curvature operator}, is defined in \S 3
as the trace of the composition of the second fundamental forms.

A second generalization of the Theorem of Takahashi is provided by
Theorem \ref{harmGrirr}. 
In both cases, 
the mean curvature operator plays a crucial role. 

In \S4, we introduce three natural functionals and 
use the Gauss-Codazzi equations for vector bundles 
to compute their Euler-Lagrange equations 
in terms of the second fundamental forms. 
The first two functionals are well-known: they are the Yang-Mills functional 
for the pull-back of the universal quotient bundle and 
the energy funtional of a map. 
The third functional is obtained as the $L^2$-norm of the mean curvature operator $A$ of mappings. 
It is shown that  $f:M\to Gr_p(W)$ is an extremal of all three functionals 
if $f:M\to Gr_p(W)$ is totally geodesic or 
if $f:M\to Gr_p(W)$ is holomorphic and the pull-back of the quotient bundle 
is an Einstein-Hermitian bundle over a compact K\"ahler 
manifold $M$. 
The functional defined by $A$ turns out to have a lower bound which relates to the total energy 
of the map. 
A map which minimizes the functional defined by $A$ has the property that its mean curvature operator is proportional to 
the identity of the pull-back bundle. 
This property shall be referred to as the {\it Einstein-Hermitian} condition. 

From this point of view,
a vector bundle together with a  finite-dimensional vector space 
of sections on the bundle induces a map into a Grassmannian. 
Such a map is called the {\it induced map} (Definition \ref{induced map}). 
A famous example of an induced map is the Kodaira embedding of 
an algebraic manifold into a complex projective space, 
which is induced by a holomorphic line bundle and its space of
holomorphic sections. 
Though it is well-known that a holomorphic map between K\"ahler
manifolds is harmonic, 
our Main Theorem 1 with Lemma \ref{hollap} yields the fact. 

We apply Main Theorem 1 to obtain a generalization of the theory of do
Carmo and Wallach 
\cite{DoC-Wal} in \S5. 
In the original theory the Theorem of Takahashi is applied 
to classify minimal immersions of spheres into spheres.
A minimal immersion is a special case of a harmonic map with constant
energy density. 
We are now concerned with a harmonic map from
a compact Riemannian manifold into a Grassmannian, 
satisfying the {\it gauge} and the {\it Einstein-Hermitian} conditions 
(see \S5 for the definition). 
After specifying a vector bundle with a connection on a Riemannian manifold, 
a {\it standard map} is defined as the map induced by some eigenspace of
the Laplacian acting on sections of the vector bundle. 
We show that those harmonic map is obtained as a deformation of a
standard map 
and corresponds to a pair $(W,T)$, 
where $W$ is an eigenspace of the Laplacian with $L^2$-scalar product 
and 
$T$ is a semi-positive Hermitian transformation of $W$ (Theorem \ref{GGenDW}). 
To establish this correspondence,  we introduce the notion of {\it gauge
equivalence} of maps 
while 
the equivalence relation used in the original do Carmo-Wallch theory is
called 
{\it image equivalence}. 
From the viewpoint of gauge theory, even in the original do Carmo-Wallch theory, 
we can find a bundle isomorphism between flat bundles preserving metrics and connections, 
which is realized by the position vector. 
Due to this realization, the gauge theoretic side of the theory might have been overlooked in the literature. 
This means that the gauge condition is automatically satisfied in the original theory. 
In addition, the mean curvature operator is regarded as a function in the case that the target 
is the sphere or the complex projective space, 
because the universal quotient bundle is of real or complex rank one in each case. 
Thus, the Einstein-Hermitian condition is also automatically satisfied in the original theory. 
This could give an explanation why the Einstein-Hermitian condition of maps has not been considered previously. 
Nevertheless, to go further we require these two conditions to find a generalization of do Carmo-Wallach theory. 
Moreover, introducing bundle isomorphisms in the theory naturally leads us into the notion of gauge equivalence, 
which enables us to eventually give a geometric meaning 
to the compactification of the moduli spaces by the natural topology induced from $L^2$-scalar product. 
(see Remark after Corollary \ref{traceIT}). 
To construct the moduli space by image equivalence, 
the action of the centralizer of the holonomy group needs to be taken into account. 

When $f$ is a holomorphic map from a compact K\"ahler manifold into a
Grassmannian with a K\"ahler structure, 
we obtain Thereoms \ref{HGenDW} and \ref{HGenDWI}. 
In particular, the equation obtained in Theorem \ref{HGenDW} has an
interpretation from the point of view 
of moment maps. 
Two moment maps are considered: one defined  by the action of the gauge
group, and the other 
specified by the action of some finite-dimensional Lie group.
Then, the interplay of both momemt maps implies a rigidity result 
(Theorem \ref{rigid}), 
which can be regarded as a generalization of Calabi's rigidity theorem 
\cite{Cal} (see also Theorem \ref{Calabi}) 
in the case that the domain is a {\it compact} K\"ahler manifold. 

Next, we take a homogeneous vector bundle with a canonical 
connection over a compact reductive Riemannian homogeneous space $G/K$. 
Then, each invariant subspace of the eigenspace of the Laplace operator on the vector bundle
induces a $G$-equivariant map from $G/K$ into a 
Grassmannian, which is also called a standard map. 
We give a sufficient condition for a standard map being harmonic 
satisfyng gauge and Eistein-Hermitian conditions (Lemma \ref{stharm}). 
Since a standard map is $G$-equivariant, the energy density is constant, 
its value being expressible through its eigenvalue.
In this section a few examples of standard maps are displayed, some of
which 
are related to K\"ahler or quaternion-K\"ahler moment maps.  
We shall use our Main Theorem 1 to obtain a classification of 
harmonic maps satisfying the gauge and the Einstein-Hermitian conditions
(Theorem \ref{GenDW}). 
As a result, the moduli problem is connected with the representation
theory of compact Lie groups. 

In \S6, two inequalities given by Joyce \cite{Joy} are our main concern. 
Joyce obtains some estimates of the dimension of the eigenspaces of the
Laplacian acting on the space of functions 
on compact minimal Lagrangian submanifolds of complex projective space. 
Theorem \ref{harmGrirr} implies the same estimates for minimal
Lagrangian submanifolds of complex projective space. 
For one of the inequalities, such an estimate is generalized to 
minimal Lagrangian submanifolds of compact irreducible Hermitian
symmetric spaces. 

In section $7$, we apply the generalization of do Carmo-Wallach Theory 
to obtain various rigidity theorems and moduli spaces. 
We give an alternative proof of the Theorem of Bando-Ohnita \cite{Ban-Ohn}, 
which states the rigidity of minimal immersion of the complex projective
line into complex projective spaces. 
A similar method yields rigidity of holomorphic isometric embeddings between complex projective spaces, 
which is Calabi's result \cite{Cal}. 
In particular, we obtain a new proof of a result due to Calabi 
\cite{Cal} on the rigidity of holomorphic isometric immersions of 
{\it compact} K\"ahelr manifolds 
into the complex projective spaces (Theorem \ref{Calabi}). 
Though the so-called local rigidity is not proved in the present paper, 
it is also shown that the Fubini-Study metric on the projective space is
induced from the $L^2$-scalar product 
on the space of holomorphic sections of an Einstein-Hermitian line 
bundle in Theorem \ref{Calabi}. 
We use Theorem \ref{Calabi} to obtain a result on representaion theory (Theorem \ref{linecl}), 
which is itself of independent interest. 
Then Theorem \ref{linecl} is applied to get a rigidity theorem of {\it projectively flat immersion} defined in \S7.
Toth defines the notion of polynomial minimal immersion between complex projective spaces \cite{Tot}. 
In this notion, the gauge condition is implicitly supposed to be satisfied. 
Then, we show that every harmonic map between complex projective spaces 
satisfying the gauge condition for the canonical connection on the hyperplane bundle 
is automatically 
a polynomial map in the sense of Toth by Theorem \ref{harmGrirr}. 
Finally, we use Theorem \ref{HGenDWI} to describe the moduli spaces of holomorphic isometric 
embeddings of complex projective lines into 
complex quadrics. 
At this stage, these examples manifest the difference of gauge 
equaivalence and image equivalence. 

In the final section, we compare the generalized  do Carmo-Wallach
construction with 
the well-known ADHM-construction of instantons.  
Though both the harmonic map equation and the anti-self-dual equation are non-linear, 
linear equations naturally emerge in our geometric setting, 
which lead us to a description of moduli spaces in linear algebraic terms.  
 
The author would like to express his gratitude to 
Professors O.Macia and  M.Takahashi for many valuable comments and
discussions. 

This research is supported by JSPS KAKENHI Grant Number 26400074.

\section{Preliminaries}

We review some standard material, mostly in order to fix our notation 
in the present paper.
\subsection{A harmonic map}
Let $M$ and $N$ be Riemannian manifolds and $f:M\to N$ 
be a (smooth) map. 
We define {\it the energy density} 
$e(f):M\to \mathbf R$ of $f$ as 
$$
e(f)(x):=|df|^2=\sum_{i=1}^{\text{\rm dim}\, M}|df(e_i)|^2, 
$$
where we use both Riemannian metrics on $M$ and $N$ 
and $e_1,\cdots,e_{\text{\rm dim}\, M}$ denotes 
an orthonormal basis of 
the tangent space $T_xM$. 
Then, {\it the tension field} $\tau(f)$ of $f$ is defined to be 
$$
\tau(f):=\text{\rm trace}\,\nabla df
=\sum_{i=1}^{\text{\rm dim}\, M}(\nabla_{e_i}df)(e_i), 
$$
which is a section of the pull-back bundle 
$f^{\ast}TN\to M$ 
of the tangent bundle $TN \to N$. 

\begin{defn}\cite{Ee-Sam}
A map $f:M\to N$ is called a {\it harmonic map} if the tension field 
vanishes ($\tau(f)\equiv 0$).
\end{defn}

The symmetric form $\nabla df$ with values in 
$f^{\ast}TN\to M$ 
is called {\it the second fundamental form}.  
We say that a map $f:M\to N$ is a {\it totally geodesic map} 
if $\nabla df\equiv 0$. 
By definition, a totally geodesic map is a harmonic map. 

If we suppose that $f:M\to N$ is an isometric immersion, 
then the tension field is a mean curvature vector, 
the second fundamental form is just the one in submanifold geometry 
and a harmonic map is nothing but a minimal immersion. 

\subsection{geometry of Grassmannian manifolds}
First of all, we focus our attention on 
a {\it real} Grassmannian manifold.

Let $W$ be a real $N$-dimensional vector space 
with an inner product $(\cdot, \cdot)$ 
and an orientation. 

Let $Gr_p(W)$ be a Grassmannian manifold of oriented $p$-planes 
in $W$ 
with the homogeneous Riemannian metric 
$g_{Gr}$ 
induced by the inner product on $W$. 
To define $g_{Gr}$ more precisely, 
let $S \to Gr_p(W)$ be the tautological vector bundle. 
It is a homogeneous vector bundle with a metric $g_S$ 
induced by $(\cdot,\cdot)$ 
on $W$ and a canonical connection. 
We have an exact sequence of vector bundles:
$$
0 \to S \to \underline{W} \to Q \to 0,
$$
where $\underline{W}\to Gr_p(W)$ is a trivial vector bundle 
of fibre $W$, and $Q\to Gr_p(W)$ is the quotient bundle. 
The vector bundle $Q\to Gr_p(W)$ is also regarded as a homogeneous 
vector bundle with the induced metric $g_Q$ and a canonical connection. 
The tangent bundle is identified with $S^{\ast}\otimes Q \cong 
S\otimes Q$ and the Riemannian metric $g_{Gr}$ is identified with the 
tensor product of $g_S$ and $g_Q$ : $g_{Gr}=g_S\otimes g_Q$. 
We call $g_{Gr}$ a Riemannian metric of Fubini-Study type. 
To emphasize the role of the inner product $(\cdot,\cdot)$, 
we denote by $\left(Gr_p(W),(\cdot,\cdot)\right)$ a Grassmannian with a metric of Fubini-Study type induced from 
$(\cdot,\cdot)$. 

We fix an orthonormal basis $w_1, \cdots, w_N$ 
of $W$ which is compatible with its orientation. 
We denote by $\mathbf R^p$ 
the subspace spanned by $w_1, \cdots, w_p$ 
and by $\mathbf R^q$ the orthogonal complementary subspace. 
The orthogonal projection to $\mathbf R^p$ is denoted by 
$\pi_p$ and the orthogonal projection to $\mathbf R^q$ by 
$\pi_q$. 
Using the orthogonal projection $\pi_q$, 
$W$ can be considered as a subspace of sections of $Q\to Gr_p(W)$. 
Explicitly, we have 
$$
\pi_Q(w)_{[g]}:=\left[g, \pi_q(g^{-1}w_A) \right]\, \in \, 
\Gamma\left(Q=G\times_{K_0} \mathbf R^q\right), \,w\in W,\,g\in G
$$
where 
$G=\text{SO}(N)$ and $K_0=\text{SO}(p)\times \text{SO}(q)$. 
The inner product $(,)_W$ gives a bundle injection 
$i_Q:Q \to \underline{W}$:
$$
i_Q\left([g, v]\right)=\left([g], gv\right), \quad v\in \mathbf R^q.
$$
In short, $Q\to Gr_p(W)$ is also regarded as the orthogonal complementary 
bundle 
$S^{\bot} \to Gr_p(W)$ to $S\to Gr_p(W)$. 
We can define a connection $\nabla^Q$ on  $Q\to Gr_p(W)$. 
If $t$ is a section of $Q\to Gr_p(W)$, 
then we have 
$$
\nabla^Q t = \pi_Q d\left(i_Q(t)\right).
$$
The connection $\nabla^Q$ is nothing but the canonical connection. 

In a similar way, 
we can use 
$i_S:S \to \underline{W}$:
$$
i_S\left([g, u]\right)=\left([g], gu\right), \quad u\in \mathbf R^p.
$$
and $\pi_S:\underline{W} \to S$:
$$
\pi_S(w_A):=\left[g, \pi_p(g^{-1}w_A) \right]\, \in \, 
G\times_K \mathbf R^p. 
$$
to express the canonical connection $\nabla^S$:
$$
\nabla^S s = \pi_S d\left(i_S(s)\right), \quad s \in \Gamma(S).
$$

In this context, 
since $S\to Gr_p(W)$ is a subbundle of $\underline{W} \to Gr_p(W)$, 
it is natural to introduce the second fundamental form 
$H$ in the sense of Kobayashi \cite{Kob}, 
which is a $1$-form with values in $\text{Hom}(S, Q) \cong 
S^{\ast}\otimes Q$:
$$
di_S s= \nabla^S s + H(s), \quad 
H(s)=\pi_Q d\left(i_S(s)\right).
$$
If $s=\pi_S(w)$, we can compute
$$H(s)=\left[g, \pi_{\mathfrak{m}}(g^{-1}dg)\pi_p (g^{-1}w)\right],
$$
here we use a standard decomposition of Lie algebra $\mathfrak g$ of $G$:
$$
\mathfrak g=\mathfrak k \oplus \mathfrak m, \quad 
\mathfrak k=\mathfrak{so}(p) \oplus \mathfrak{so}(q), 
$$
and the obvious orthogonal projection $\pi_{\mathfrak{m}}$. 

The tangent bundle of $Gr_p(W)$ can be expressed in two ways:
$$
T:=TGr_p(W)=G\times_{K_0} \mathfrak{m}, \quad T=S^{\ast}\otimes Q.
$$
Using the latter expression, the cotangent bundle 
$T^{\ast}$ is considered as
$$
T^{\ast}=S\otimes Q^{\ast}. 
$$
Hence the second fundamental form 
$H \in \Omega^1(S^{\ast}\otimes Q)$ 
can also be regarded as a section of 
$T^{\ast} \otimes T =
S\otimes S^{\ast} \otimes Q\otimes Q^{\ast}$. 

\begin{lemma}
The second fundamental form $H$ can be regarded as the identity 
transformation of the tangent bundle $T$. 
\end{lemma}
\begin{proof}
Let $R_g$ be the right translation of $G$ 
and $\pi:G\to Gr_p(W)$ a natural projection. 
We may evaluate on a Killing vector field 
$X^M=d\pi dR_g X=\left[g, \pi_{\mathfrak m}\left(g^{-1}Xg\right)\right]$, 
$X\in \mathfrak m$.   
Indeed, 
\begin{align*}
H_{X^M}=&\left[g,\pi_{\mathfrak m}\left(g^{-1}dg\right)\right]d\pi dR_g X
=\left[g, \pi_{\mathfrak m}\left(g^{-1}dg\right)
\left(g^{-1} X g\right)_{\mathfrak m}\right]\\
=&\left[g, \pi_{\mathfrak m}\left(g^{-1}X g\right)\right]
=X_M.
\end{align*}
\end{proof}
\begin{cor}
The second fundamental form $H$ is parallel.
\end{cor}
We can also define the second fundamental form 
$K \in \Omega^1(Q^{\ast}\otimes S)$ of a subbundle 
$i_Q:Q \to \underline{W}$: 
$$
di_Q t=\nabla^Q t + K(t)
$$
For a vector $w \in W$, we have two sections 
$s=\pi_S(w)$ and $t=\pi_Q(w)$, 
each of which is sometimes called {\it the section corresponding to }$w$. 
From our expression, we have 
\begin{prop}
If $s$ and $t$ are the sections corresponding to $w \in W$, then 
$$
\nabla^S s = -K(t), \quad 
\nabla^Q t = -H(s).
$$
\end{prop}
\begin{lemma}\label{sksymHK}
The second fundamental forms $H$ and $K$ satisfy 
$$
g_Q(H s, t)=-g_S(s,K t).
$$
\end{lemma}
\begin{proof}If we regard $s$ and $t$ as sections of $\underline{W}\to Gr_p(W)$, 
then
$$
g_Q(H s, t)=(ds ,t)_W=-(s,dt)_W=-g_S(s, K t).
$$
\end{proof}

The orthonormal basis $w_1,\cdots,w_N$ of $W$ provides us with the 
corresponding sections 
$s_A=\pi_S(w_A)$ and 
$t_A=\pi_Q(w_A)$. 


It follows from the definition that 
$$
w^{\alpha} \otimes w_r \in \mathbf R^{p^{\ast}}\otimes \mathbf R^q \cong \mathfrak m, 
\quad \alpha=1,\cdots,p, \, r=1,\cdots,q
$$
is an orthonormal basis of $\mathfrak m$. 
However, if $\mathfrak m$ is regarded as a subspace of $\mathfrak g$, 
we should adopt the identification between 
$\mathbf R^{p^{\ast}}\otimes \mathbf R^q$ 
and $\mathfrak m$ in a following way:
\begin{align*}
w^{\alpha} \otimes w_r &\in \mathbf R^{p^{\ast}}\otimes \mathbf R^q 
\leftrightarrow 
w^{\alpha} \otimes w_r-w_{\alpha} \otimes w^r \in \mathfrak m, \\
w^r \otimes w_{\alpha} &\in \mathbf R^{q^{\ast}}\otimes \mathbf R^p 
\leftrightarrow 
w^{\alpha} \otimes w_r-w_{\alpha} \otimes w^r \in \mathfrak m
\end{align*}
For simplicity, we use the invariant inner products to identify 
$\mathbf R^{p^{\ast}}$ and  $\mathbf R^{q^{\ast}}$ with 
$\mathbf R^{p}$ and  $\mathbf R^{q}$, respectivley. 
As a result, we have 
$$
w_{\alpha} \otimes w_r \in \mathbf R^{p}\otimes \mathbf R^q 
\leftrightarrow 
w^{\alpha} \otimes w_r-w_{\alpha} \otimes w^r \in \mathfrak m. 
$$
Then the second fundamental form 
$K \in \Omega^1(Q^{\ast} \otimes S)$ can also be regarded as 
a section of $Q\otimes S \otimes Q \otimes S$. 
Under the irreducible decomposition, 
$K$ corresponds to a constant section of 
$\mathbf R \subset Q\otimes S \otimes Q \otimes S$. 
More explicitly, our identification between 
$\mathfrak{m}$ and $\mathbf R^{p^{\ast}}\otimes \mathbf R^q 
\cong \mathbf R^p\otimes \mathbf R^q$ 
yields 
$$
H=Id_S \otimes Id_Q, \quad\text{and}\quad K=-Id_S\otimes Id_Q. 
$$
Obviously, we have
\begin{lemma}
The second fundamental tensor $K$ is also parallel.
\end{lemma}

\begin{prop}\label{Riemmet}
For arbitrary tangent vectors $X$ and $Y$, 
we have
$$
g_{Gr}(X,Y)=\sum_A g_S(K_X t_A, K_Y t_A)=\sum_A g_Q(H_X s_A, H_Y s_A).
$$
\end{prop}
\begin{proof}
The key fact is $g_{Gr}=g_S\otimes g_Q$, when identifying 
$TGr \to Gr_p(W)$ with $S\otimes Q \to Gr_p(W)$. 
At $[e] \in Gr_p(W)$, where $e$ is the unit elememt of $G$, 
assume that 
$$
X=\xi_{\alpha}^r w^{\alpha}\otimes w_r, \quad 
Y=\eta_{\beta}^s w^{\beta}\otimes w_s.
$$
From the definition of $g_{Gr}$, it follows that
$$
g_{Gr}(X,Y)=\sum_{r,\alpha} \xi_{\alpha}^r \eta_{\alpha}^r. 
$$
On the other hand, 
\begin{align*}
\sum_A g_S(K_X t_A, K_Y t_A)
=&\sum_{r} g_S(K_X t_r, K_Y t_r)
=\sum_{r,\alpha,\beta} g_S(\xi_{\alpha}^rw^{\alpha}, 
\eta_{\beta}^rw^{\beta}) \\
=&\sum_{r,\alpha} \xi_{\alpha}^r \eta_{\alpha}^r.
\end{align*}
\end{proof}
\begin{rem}
Lemma \ref{sksymHK} and Proposition \ref{Riemmet} give us 
$$
g_{Gr}=-\text{\rm trace}_Q\,HK=-\text{\rm trace}_S\,KH.
$$
\end{rem}

We can easily compute $(\nabla^S)^2$ and $(\nabla^Q)^2$:
$$
\nabla^S_X(\nabla^S s_A)(Y)=K_YH_X s_A, \quad 
\nabla^Q_X(\nabla^Q t_A)(Y)=H_YK_X t_A.
$$
In particular, we know that sections $s_A$ and $t_A$ 
are eigensections of the Laplacian 
($\Delta s_A = q s_A, \quad \Delta q_A = p q_A$).

Next, we consider a {\it complex} Grassmannian manifold. 
The main difference of a complex Grassmannian from a real Grassmannian 
is that we can use the Hodge decomposition, because a complex Grassmannian 
is a K\"ahler manifold. 
More precisely, let $W$ be a complex vector space with a Hermitian inner 
product $(\cdot, \cdot)_W$ and 
$Gr_p(W)$ a complex Grassmannian of $p$-planes in $W$. 
We can define homogeneous vector bundles 
$S\to Gr_p(W)$ and $Q \to Gr_p(W)$ 
with induced Hermitian metrics $g_S$ and $g_Q$ by $W$, respectively. 
Canonical connections give holomorphic structures to 
$S\to Gr_p(W)$ and $Q \to Gr_p(W)$. 
In particular, $W$ can be regarded as the space of holomorphic sections of 
$Q \to Gr_p(W)$. 
The holomorphic tangent bundle $T$ is identified with 
$S^{\ast}\otimes Q$ and 
the holomorphic cotangent bundle is 
$S\otimes Q^{\ast}$. 
The identification includes Hermitian metrics and connections. 
The second fundamental form $H \in \Omega^1(\text{Hom}(S,Q))$ is of type 
$(1,0)$. 
This means that $H$ can be considered as a section of 
$S\otimes Q^{\ast}\otimes S^{\ast}\otimes Q$ and we obtain 
$H=id_S \otimes id_Q$ in a similar way. 
The second fundamental form 
$K \in \Omega^1(\text{Hom}(Q,S))$ is of type $(0,1)$ and recognized as 
$-id_S \otimes id_Q \,\in\, \Gamma(S^{\ast}\otimes Q \otimes S\otimes Q^{\ast})$.

\section{Harmonic maps into Grassmannians} 
In this section, we shall prove the main theorems. 
We denote by $\underline{W} \to M$ a trivial vector bundle $M\times W \to M$. 

\begin{defn}
Let 
$V \to M$ be a vector bundle and 
$W$ a space of sections of 
$V \to M$. 
We define an evaluation map 
$ev:\underline{W} \to V$ in such a way that 
$ev(t)(x):=t(x) \in V_x$ for $t\in W$. 
Hence an evaluation map is a bundle homomorphism. 
The vector bundle $V\to M$ 
is called to be {\it globally generated by} 
$W$ 
if the evaluation homomorphism 
$ev: \underline{W} \to V$ 
is surjective. 
\end{defn}

We have an exact sequence of vector bundles on 
$Gr_p(W)$:
$$
0 \to S \to \underline{W} \to Q \to 0, 
$$
where $S\to Gr_p(W)$ is a tautological vector bundle, 
$\underline{W}\to Gr_p(W)$ is a trivial vector bundle of rank $N$, 
and $Q \to Gr_p(W)$ is the quotient bundle. 
Then $W$ can be regarded as a space of sections of 
$Q\to Gr_p(W)$ which is globally generated by $W$.   

We fix an inner product or a Hermitian inner product $(\cdot, \cdot)$ on 
a linear space $W$ 
according to a ground field. 
We call $(\cdot, \cdot)$ a {\it scalar product}. 
Then, as explained in \S 2, 
the $\text{SO}(N)$ or $\text{SU}(N)$ structure of $W$ provides us with 
a Riemannian structure on the Grassmannian 
$Gr_p(W)$ 
and the vector bundles 
$S\to Gr_p(W)$ and $Q\to Gr_p(W)$ can be regarded as homogeneous vector bundles 
on $Gr_p(W)$ with fibre metrics and canonical connections. 

Let $f:M\to Gr_p(W)$ be a smooth map. 
Pulling back $Q \to Gr_p(W)$ to $M$, 
we obtain a vector bundle 
$f^{\ast}Q \to M$, 
which is denoted by $V\to M$. 
Though $W$ also gives sections of $V \to M$, 
the linear map $W\to \Gamma(V)$ might not be an injection. 
Even in such a case, $W$ is still called a space of sections. 
Then the pull-back bundle $V \to M$ 
is also globally generated by $W$. 

If $f:M \to Gr_p(W)$ is a smooth map, 
then we also pull back a fiber metric and a connection on $Q\to Gr_p(W)$ 
to obtain a fibre metric $g_V$ and a connection $\nabla^V$
on $V\to M$. 

In a similar way, the pull-back bundle 
$f^{\ast}S \to M$ is denoted by 
$U \to M$ which has a pull-back fibre metric $g_U$ 
and a pull-back connection $\nabla^U$. 

The second fundamental forms are also pulled back 
and denoted by the same symbols 
$H \in \Gamma(f^{\ast}T^{\ast}\otimes U^{\ast}\otimes V)$ 
and 
$K \in \Gamma(f^{\ast}T^{\ast}\otimes V^{\ast}\otimes U)$. 
If we restrict bundle-valued linear forms $H$ and $K$ on 
the pull-back bundle 
$f^{\ast}T^{\ast} \to M$ 
to linear forms on $M$, 
$H$ and $K$ are nothing but the second fundamental forms 
of subbundles 
$U \to \underline{W}$ and $V \to \underline{W}$, respectively, 
where $\underline{W}$ is a trivial vector bundle $M\times W \to M$. 

From now on, we assume that $M$ is a Riemannian manifold 
with a metric $g$. 
Then, we use the Riemannian structure on $M$ 
and the pull-back connection on $V\to M$ 
to define 
the Laplace operator 
$\Delta^V=\Delta=\nabla^{V^{\ast}}\nabla^V
=-\sum_{i=1}^n \nabla^{V}_{e_i}\left(\nabla^V\right)(e_i)$ 
acting on sections of $V\to M$ 
and 
a bundle homomorphism 
$A \in \Gamma\left(\text{Hom}\,V \right)$ 
is defined as the trace of the composition of the 
second fundamental forms $H$ and $K$:
$$
A:=\sum_{i=1}^n H_{e_i}K_{e_i}, 
$$
where $n$ is the dimension of $M$ 
and $\{e_i\}_{i=1,2,\cdots n}$ is an orthonormal basis of 
the tangent space of $M$. 
The bundle homomorphism $A \in \Gamma\left(\text{Hom}\,V \right)$ 
is called the {\it mean curvature operator} 
of $f$. 
\begin{rem}
Since we use the same symbols for the second fundamental forms and their 
pull-backs, the definition of $A$ may cause confusion. 
We sometimes express $A$ as $\sum H_{df(e_i)}K_{df(e_i)}$. 
\end{rem}

We now describe properties of 
$A \in \Gamma\left(\text{Hom}\,V\right)$. 

\begin{lemma}\label{Nonpos}
The mean curvature operator $A$ is a non-positive symmetric 
(or Hermitian) operator. 
\end{lemma}
\begin{proof}
It follows from Lemma \ref{sksymHK} that 
\begin{align*}
g_V\left(At,t\right)
=&\sum_{i=1}^n g_V\left( H_{e_i}K_{e_i}t,t\right)
=-\sum_{i=1}^n g_U\left(K_{e_i}t,K_{e_i}t\right) \\
=&\sum_{i=1}^n g_V\left(t, H_{e_i}K_{e_i}t\right)
=g_V(t,At)
\end{align*}
for an arbitrary $t\in W$. 
It immediately yields the desired result. 
\end{proof}

\begin{lemma}\label{ed}
The energy density $e(f)$ is equal to $-\text{\rm trace}\,A$. 
\end{lemma}
\begin{proof}
We use Proposition \ref{Riemmet} to obtain 
$$
e(f)=\sum_{i=1}^n g_{Gr}(df(e_i),df(e_i))
=-\sum_{i=1}^n \text{\rm trace}\,H_{df(e_i)}K_{df(e_i)}
=-\text{\rm trace}\,A 
$$
\end{proof}

Let $t$ be a section of $V \to M$. 
We denote by $Z_t$ the zero set of $t$:
$$
Z_t:=\left\{ x \in M \,|\,t(x)=0\right\}.
$$
\begin{defn}\label{zero}
A space of sections $W$ of a vector bundle $V\to M$ has the 
{\it zero property} for 
the Laplacian if 
$ Z_t \subset Z_{\Delta t}$ for an arbitrary $t \in W$.
\end{defn}

\begin{exam}
If $W$ is an eigenspace for the Laplacian, then $W$ has the zero property.
\end{exam}

\begin{thm}\label{main theorem}
Let $(M,g)$ be an $n$-dimensional Riemannian manifold and 
$f:M\to Gr_p(W)$ a smooth map.  
We fix a scalar product 
$(\cdot,\cdot)$ on $W$, 
which gives a Riemannian structure on $Gr_p(W)$. 

Then, the following two conditions are equivalent. 
\begin{enumerate}
\item $f:(M,g)\to \left(Gr_p(W), (\cdot,\cdot)\right)$ is a harmonic map. 
\item $W$ has the zero property for the Laplacian. 
\end{enumerate}
Under these conditions, we have for an arbitrary 
$t \in W$, 
$$
\Delta t=-At, \quad \text{and} \quad e(f)=-\text{\rm trace}\,A.
$$
\end{thm}
\begin{proof}
The pull-back bundle of the tautological vector bundle and the universal quotient bundle 
are denoted by $U \to M$ and $V \to M$, respectively. 

Let $X$ and $Y$ be tangent vectors of $M$ and $t \in \Gamma(V)$. 
We consider the second fundamental form $K\in \Omega^1(V^{\ast}\otimes U)$. 
Since $\nabla K=0$ on $Gr_p(W)$, we have 
\begin{align*}
(\nabla_X K)(Y;t)&=\nabla^U_X(K_Y t)-K_{\nabla_X Y}t-K_Y(\nabla^V_X t) \\
&=\nabla^U_X(K_Y t)-K_{\tilde{\nabla}_X Y-(\nabla_X df)(Y)}t-K_Y(\nabla^V_X t) \\
&=K_{(\nabla_X df)(Y)}t, 
\end{align*}
where $\nabla$ is the Levi-Civita connection on $M$ and 
$\tilde \nabla$ is the Levi-Civita connection on $Gr_p(W)$. 
In particular, we obtain
$$
-\delta^{\nabla} K = K_{\tau(f)},
$$
where $\tau(f)$ is the 
tension field of $f:M\to Gr_p(W)$. 

Next we fix a vector $w \in W$ and take the corresponding sections 
$s \in W \subset \Gamma(U)$ and $t \in W \subset \Gamma(V)$ to $w$. 
Then we have 
\begin{align*}
\nabla^V_X \left(\nabla^V t\right) (Y)=&
\nabla^V_X \left(\nabla^V_Y t\right) -\nabla^V_{\nabla_X Y} t \\
=&\nabla^Q_{df(X)} \left(\nabla^Q_{df(Y)} t\right) 
-\nabla^Q_{\tilde{\nabla}_X Y-(\nabla_X df)(Y)} t \\
=&\nabla^Q_{df(X)} \left(\nabla^Q t\right) (df(Y))+\nabla^Q_{(\nabla_X df)(Y)}t \\
=&H_YK_X t-H_{(\nabla_X df)(Y)}s.
\end{align*}
In particular, we obtain 
\begin{eqnarray}\label{Baseq} 
\Delta t -H_{\tau(f)}s +\sum_{i=1}^n H_{e_i}K_{e_i} t=
\Delta t -H_{\tau(f)}s +At=0.
\end{eqnarray}

First, we assume condition (1). 
The assumption that $f:M \to Gr_p(W)$ is harmonic yields that 
the equation (\ref{Baseq}) reduces to 
\begin{eqnarray}\label{Baseqmin}
\Delta t +At=0. 
\end{eqnarray}
We immediately conclude that $W$ has the zero property. 

Conversely, suppose condition (2). 
For an arbitrary vector $u \in U_x$, $x\in M$, we can find 
an element $w \in W$ such that the corresponding sections 
$s\in \Gamma(U)$ and $t\in \Gamma(V)$ satisfy 
$$
s(x)=u, \quad \text{and}\quad t(x)=0. 
$$
The equation (\ref{Baseq}) gives us 
$$
H_{\tau(f)}s=\Delta t+At. 
$$
Since $W$ has the zero property for the Laplacian and  
$t(x)=0$,  
it follows that $\Delta t(x)=0$. 
Hence we have 
$$
H_{\tau(f)}u=0, 
$$
and so $\tau(f)=0$, which means that $f$ is a harmonic map. 

\end{proof}

\begin{thm}\label{harmGrirr}
Let $(M,g)$ be an $n$-dimensional Rimannian manifold 
and $f:M\to Gr_p(W)$ a map. 
We fix a scalar product $(\cdot,\cdot)$ on $W$. 

Then, the following two conditions are equivalent. 
\begin{enumerate}
\item $f:(M,g)\to \left(Gr_p(W), (\cdot,\cdot)\right)$ is a harmonic map and there exists a function 
$h(x)$ such that 
$A_{x}=-h(x)Id_V$ for 
an arbitrary $x \in M$.  
\item There exists a function $h$ on $M$ such that 
$$
\Delta t = ht \quad \text{for an arbitrary}\,\,t\in W. 
$$
\end{enumerate}
Moreover, under the above conditions, we have 
$$
e(f)=qh, 
$$
where $e(f)$ is the energy density of $f$. 
\end{thm}
\begin{proof}
First, suppose condition (1). 
We have from equation (\ref{Baseq}) that 
$$
\Delta t + At=\Delta t-h(x)t=0, 
$$
for an arbitrary $t\in W$. 

Conversely, suppose condition (2). 
Theorem \ref{main theorem} yields that $W$ has the zero property, and so 
$f$ is harmonic. 

It follows from (\ref{Baseq}) that 
$$
\Delta t+At=0, 
$$
for an arbitrary $t\in W$. 
Then condition (2) yields that $A=-h Id_{V}$ on $V \to M$. 

Lemma \ref{ed} gives 
$\text{\rm trace}\,A=-qh$. 
\end{proof}

Next, suppose that $f:M \to Gr_p(W)$ is an isometric immersion. 
Instead of the tension field, we can use the mean curvature vector 
to obtain similar results. 
In this case, note that $e(f)=n$, because $f$ is an isometric immersion. 
By replacing harmonicity by minimality, 
we obtain a straightforward generalization of Theorem of Takahashi \cite{TTaka}.  
\begin{thm}\label{genTak}
Let $(M,g)$ be an $n$-dimensional Riemannian manifold 
and $f:(M,g)\to \left(Gr_p(W), (\cdot,\cdot) \right)$ an isometric immersion. 

Then, the following two conditions are equivalent. 
\begin{enumerate}
\item $f:M\to Gr_p(W)$ is a minimal immersion and there exists a function 
$h(x)$ such that 
$A_{x}=-h(x)Id_V$ for 
an arbitrary $x \in M$.  
\item There exists a function $h$ on $M$ such that 
$$
\Delta t = ht 
\quad \text{for an arbitrary}\,\,\,t\in W. 
$$
\end{enumerate}
Moreover, under the above conditions, we have 
$$
n=qh.
$$
Hence the function $h$ is really a constant function. 
As a consequence, $t$ is an eigensection with an eigenvalue $\frac{n}{q}$. 
\end{thm}
\begin{rem}
This gives us the original form of Theorem of Takahashi \cite{TTaka}. 
Indeed, in the case that the target is a sphere, 
$V\to M$ is of rank $1$, and so the mean curvature operator 
$A$ can be always considered as a function. 
\end{rem}

\section{Functionals}

Let $f:(M,g)\to \left(Gr_p(W), (\cdot,\cdot)\right)$ be a smooth map. 
We denote by $V \to M$ the pull-back bundle of the universal quotient bundle over $Gr_p(W)$ by $f$. 

Since the connection on $\underline{W}\to M$ is flat, 
it follows from the Gauss-Codazzi equations for vector bundles 
that 
$$
R^V(X,Y)=H_YK_X-H_XK_Y, \quad 
(\nabla_X K)(Y)=(\nabla_Y K)(X).
$$
We have
\begin{align*}
(\nabla_X R^V)(Y,Z)
=&H_{(\nabla_X df)(Z)}K_Y+H_ZK_{(\nabla_X df)(Y)}\\
&-H_{(\nabla_X df)(Y)}K_Z-H_YK_{(\nabla_X df)(Z)}.
\end{align*}
In particular, 
\begin{align}\label{YMV}
(\delta^{\nabla} R^V)(X)=&-(\nabla_{e_i} R^V)(e_i,X) \\
=&-H_{(\nabla_{e_i} df)(X)}K_{e_i}-H_XK_{\tau(f)} \notag \\
&+H_{\tau(f)}K_X+H_{e_i}K_{(\nabla_{e_i} df)(X)}.\notag 
\end{align}
On the other hand, we obtain
\begin{equation}\label{A}
\nabla_X A=H_{(\nabla_{e_i} df)(X)}K_{e_i}+
H_{e_i}K_{(\nabla_{e_i} df)(X)}.
\end{equation}

Immediately, we have
\begin{lemma}
Let $f:(M,g)\to \left(Gr_p(W), (\cdot,\cdot)\right)$ be a smooth map. 
If $f$ is a totally geodesic map, 
then the pull-back connection on $V \to M$ is a 
Yang-Mills connection and $A$ is parallel. 
\end{lemma}
\begin{proof}
By definition, $\nabla df=0$. 
The result follows from \eqref{YMV} and \eqref{A}. 
\end{proof}

We shall present another occurrence in which the pull-back connection is 
a Yang-Mills connection and $A$ is parallel. 
To do so, 
we need 

\begin{lemma}\label{hollap}
Let $M$ be a K\"ahler manifold and $V\to M$ a holomorphic 
vector bundle with a Hermitian metric. 
We take a compatible connection $\nabla$ on $V\to M$. 
Then, for an arbitrary holomorphic section $t \in \Gamma(V)$, we 
have
$$
\Delta t=K_{EH}t, 
$$
where $K_{EH}$ is the mean curvature in the sense of Kobayashi \cite{Kob}:
$$
K_{EH}=\sqrt{-1}\sum_{i=1}^n R(e_{i},Je_{i})
$$
where $J$ is the complex structure of $M$, 
$e_1,Je_1,\cdots, e_n,Je_n$ 
is an orthonormal basis and $R$ is the curvature of the 
compatible connection. 
\end{lemma}
\begin{proof}
We put $Z_i=\frac{1}{2}\left(e_i-\sqrt{-1}Je_i \right)$. 
We can extend the vectors $Z_i$ locally to get 
a local holomorphic frame denoted by the same symbols. 
On the one hand, we have 
$$
\nabla_{Z_i}\left(\nabla t \right)(\overline{Z_i})=
\nabla_{Z_i}\left(\nabla_{\overline{Z_i}} t \right)
-\nabla_{D_{Z_i}{\overline{Z_i}}}t=0,
$$
because $t$ is a holomorphic section. 

On the other hand, we get 
$$
\sum_{i=1}^n\nabla_{Z_i}\left(\nabla t \right)(\overline{Z_i})
=\frac{1}{4}\left\{-\Delta t
+ \sum_{i=1}^n \sqrt{-1}R(e_i,Je_i)t\right\}.
$$
\end{proof}
\begin{cor}
Under the hypothesis of Lemma \ref{hollap}, 
$V \to M$ is an Einstein-Hermitian vector bundle 
if and only if the space of holomorphic sections $H^0(M,V)$ 
of $V\to M$ is an eigenspace of the Laplacian.
\end{cor}
\begin{prop}\label{EH}
Let $M$ be a compact K\"ahler manifold and $Gr_p(W)$ a complex Grassmannian 
or a complex quadric. 
Suppose that $f:M\to Gr_p(W)$ is a holomorphic map. 
Then the mean curvature $K_{EH}$ of the pull-back bundle $f^{\ast}Q\to M$ of the quotient bundle 
equals the mean curvature operator $A$ of $f$ up to sign. 
\end{prop}
\begin{proof}
Since $f$ is holomorphic, 
it follows that the map $f$ is a harmonic map 
and $W$ can be regarded as a space of holomorphic sections of the pull-back 
of the universal quotient bundle. 
Then Theorem \ref{main theorem} implies that $\Delta t+At=0$ for $t\in W$. 
On the other hand, Lemma \ref{hollap} yields that $\Delta t=K_{EH}t$ for all holomorphic sections of $f^{\ast}Q \to M$. 
Since $W$ globally generates $f^{\ast}Q \to M$, it follows that $A=-K_{EH}$. 
\end{proof}

\begin{cor}
Let $M$ be a compact K\"ahler manifold and $Gr_p(W)$ a complex Grassmannian 
or a complex quadric. 
Suppose that $f:M\to Gr_p(W)$ is a holomorphic map such that 
the pull-back bundle $V\to M$ 
is an Einstein-Hermitian vector bundle 
with respect to the pull-back metric. 
Then  the pull-back connection is a 
Yang-Mills connection and the mean curvature operator $A$ is parallel. 
\end{cor}
\begin{proof}
Since any Einstein-Hermitian connection minimizes the Yang-Mills functional, 
the pull-back connection is a Yang-Mills connection. 
Proposition \ref{EH} yields that $A$ is parallel. 
\end{proof}


More explicitly, under the condition that $M$ is compact, we naturally have 
three functionals on the space of mappings $f:M\to Gr_p(W)$:
$$
\int_M |R^V|^2 dv_M, \,\, 
\int_M |H|^2 dv_M, \,\,
\int_M |A|^2 dv_M,
$$
where $dv_M$ is the Riemannian volume form on $M$. 
The first is the Yang-Mills functional, and Lemma \ref{ed} yields that 
$$
\int_M |H|^2 dv_M = \int_M e(f) dv_M. 
$$
If $M$ is a K\"ahler manifold and $f:M\to Gr_p(W)$ is a holomorphic map 
into a complex Grassmannian or a complex quadric, then Proposition 
\ref{EH} yields that 
$$
\int_M |A|^2 dv_M=\int_M |K_{EH}|^2 dv_M.
$$
In this case, it is the same as the Yang-Mills functional up to a topological constant 
\cite[p.111]{Kob}. 

Hence if $f:M\to Gr_p(W)$ is totally geodesic or 
if $f:M\to Gr_p(W)$ is holomorphic and the pull-back bundle 
$V \to M$ is an Einstein-Hermitian vector bundle, 
then $f$ is an extremal of all three functionals.

We consider a set of harmonic maps $f$ from a compact Riemannian manifold $(M,g)$ into a Grassmannian $Gr_p(W)$ 
with the fixed energy $E(f)$. 
Let $\mu$ be a constant determined by 
\begin{equation}\label{constant}
q\mu \text{Vol}\,(M)=E(f), 
\end{equation}
where $q=\text{dim}\,W-p$ and $\text{Vol}\,(M)=\int_M dv_M$. 
Then
$$
0\leqq |A+\mu Id_W|^2=|A|^2+2\mu \text{trace}A+\mu^2q=|A|^2-2\mu e(f)+\mu^2q.
$$
Integration and the definition of $\mu$ yield that 
$$
q\mu^2\text{Vol}\,(M)=\mu E(f) \leqq \int_M |A|^2dv_M, 
$$
where the equality holds if and only if $A=-\mu Id_W$. 

In the case that $f$ is a holomorphic map, since $A=-K_{EH}$, 
we have that $|df|^2=\sigma$, where $\sigma$ is the scalar curvature of the pull-back metric on the bundle, 
which is defined as $\text{trace}\,K_{EH}$ (see \cite[p.108]{Kob}). 
The scalar curvature $\sigma$ satisfies 
$$
\int_M \sigma dv_M=2\pi \int_M c_1(f^{\ast}Q)\wedge \frac{\omega_M^{n-1}}{(n-1)!},
$$
where $n$ denotes the complex dimension of $M$ and $\omega_M$ denotes the K\"ahler form on $M$ \cite[p.108]{Kob}. 
Therefore, we obtain
$$
E(f)=2\pi \int_M c_1(f^{\ast}Q)\wedge \frac{\omega_M^{n-1}}{(n-1)!}.
$$
Thus, 
$$
2\pi \mu \int_M c_1(f^{\ast}Q)\wedge \frac{\omega_M^{n-1}}{(n-1)!}\leqq \int_M |A|^2dv_M, 
$$
where the equality holds if and only if $A=-\mu Id_W$. 
Notice that, by definition, in the holomorphic case the constant $\mu$ depends only on the homotopy class of $f$ and the 
cohomology class of $\omega_M$. 

\section{A generalization of Theory of do Carmo and Wallach}
We give a generalization of do Carmo-Wallach theory \cite{DoC-Wal}. 
\begin{defn}\label{induced map}
Let 
$V \to M$ be a real or complex vector bundle of rank $q$ 
which is globally generated by 
$W$ of dimension $N$. 
If the real vector bundle $V\to M$ has an orientation, 
we also fix an orientation on $W$. 
Then we have a map $f:M \to Gr_p(W)$, 
where $Gr_p(W)$ is a real (oriented) or complex Grassmannian 
according to the coefficient field of $V\to M$ 
and $p=N-q$. 
The map $f$ is defined by 
$$
f(x):=\text{Ker}\,ev_x=
\left\{t\in W \,\vert \, t(x)=0 \right\}.
$$
We call $f:M \to Gr_p(W)$ the {\it induced map by} 
$(V\to M,W)$, or the {\it induced map by} $W$, if the vector bundle 
$V\to M$ is specified. 
\end{defn}
From the definition of the induced map $f:M\to Gr_p(W)$, 
the vector bundle $V\to M$ can be naturally identified with 
$f^{\ast}Q \to M$. 
To be more precise, let $\text{Ker}\, ev\to M$ be a vector bundle obtained 
as the kernel of $ev:\underline{W}\to V$. 
Since 
$S_{f(x)}=\text{Ker}\,ev_x$, 
we get a natural identification $i:\text{Ker}\,ev \to f^{\ast}S$. 
Then the following diagram gives the bundle isomorphism $\phi:V \to f^{\ast}Q$,  
which is called the {\it natural identification} of $V\to M$ with $f^{\ast}Q\to M$. 
$$
\begin{CD}
0@>>> {Ker}\,ev @>>> \underline{W} @>>> V @>>> 0 \\
@. @V{i}VV  @|  @VV{\phi}V \\
0@>>> f^{\ast}S @>>> \underline{W} @>>> f^{\ast}Q @>>> 0.
\end{CD}
$$


Conversely, 
if $f:M\to Gr_p(W)$ is a smooth map, 
then we obtain a vector bundle 
$f^{\ast}Q \to M$ which is globally generated by $W$, 
where $W$ is regarded as a space of sections of the pull-back bundle.  
It is easily observed that the induced map by $W$ is the same as 
the original map $f:M\to Gr_p(W)$. 
In this way, every map 
$f:M\to Gr_p(W)$ 
can be recognized as an induced map by 
$(f^{\ast}Q\to M,W)$. 

Let $(M,g)$ be a Riemannian manifold and $V\to M$ a vector bundle with 
a fibre metric $h_V(\cdot, \cdot)$ and a connection $\nabla$.  
Then the space of sections $\Gamma(V)$ of $V \to M$ has the 
$L^2$-scalar product induced by $g$ and $h_V$. 
Moreover, using the Riemannian structure and $\nabla$,  
we can define the Laplace operator $\Delta$ acting on $\Gamma(V)$. 
Since $\Delta$ is an elliptic operator, 
we can decompose $\Gamma(V)$ into 
the eigenspaces of the Laplacian in the $L^2$-sense:
$$
\Gamma(V)=\oplus_{\mu} W_{\mu}, \quad 
W_{\mu}:=\left\{t \in \Gamma(V)\,|\,\Delta t=\mu t \right\}.
$$
It is well-known that $W_{\mu}$ is a finite dimensional space 
equipped with the scalar product induced by 
$L^2$-scalar product. 

\subsection{Standard maps}
Suppose that an eigenspace $W_{\mu}$ globally generates $V\to M$. 
Then we define the induced map 
$f_0:M\to Gr_p(W_{\mu})$ 
by $W_{\mu}$, where $p=N-q$, $N=\text{\rm dim}\,W_{\mu}$, 
$$
f_0(x)=\text{\rm Ker}\,ev_x=\left\{t \in W_{\mu} \,|\,t(x)=0\right\},
$$
where $ev:\underline{W_{\mu}} \to V$ is the evaluation map. 
We call $f_0$ {\it the standard map} by $W_{\mu}$. 
Notice that the natural identification can be considered as the adjoint 
homomorphism $ev^{\ast}$ of the evaluation $ev:\underline{W_{\mu}}\to V$, 
when $f^{\ast}Q \to M$ is regarded as the orthogonal complement of $f^{\ast}S \to M$ 
and $ev^{\ast}$ is considered as a bundle isomorphism onto the image.

\subsection{A Generalization of do Carmo-Wallach Theory}
Let $W$ be a real or complex vector space with a scalar product $(,)_W$. 
%
We denote by $\text{H}(W)$ the set of symmetric or Hermitian endomorphisms of $W$ 
depending on $W$ being a real or complex vector space. 
We equip $\text{H}(W)$ with an inner product $(,)_H$; 
$(A,B)_H:=\text{trace}\,AB$, for$A,B \in \text{H}(W)$. 

In this section, $\mathbf K$ denotes $\mathbf R$ or $\mathbf C$. 
Symmetric operators are also called Hermitian operators, 
for simplicity. 
\begin{defn}
Let $f:M\to Gr_p(\mathbf K^m)$ 
be a map and regard 
$\mathbf K^m$ as a space of sections of $f^{\ast}Q \to M$. 
Then 
the map $f:M\to Gr_p(\mathbf K^m)$ 
is called a {\it full map} 
if the linear map $\mathbf K^m \to \Gamma(f^{\ast}Q)$ is injective. 
\end{defn}
Notice that the notion of full map is the same as the one 
in \cite{DoC-Wal} and \cite{Tot} if 
the target space is the sphere or the complex projective space. 

We give two equivalence relations of maps under the condition that 
$\mathbf K^m$ has a scalar product. 
\begin{defn}
Let $f_1$ and $f_2:M\to Gr_p(\mathbf K^m)$
be maps. 
Then $f_1$ is called {\it image equivalent} to $f_2$, 
if there exists an isometry $\psi$ of $Gr_p(\mathbf K^m)$ 
such that $f_2=\psi\circ f_1$. 
\end{defn}
An isometry $\psi$ of $Gr_p(\mathbf K^m)$ gives a bundle isomorphism of 
$Q\to Gr_p(\mathbf K^m)$ 
denoted by 
$\tilde \psi$ which covers $\psi$.  
If we have a map $f:M\to Gr_p(\mathbf K^m)$, then $\tilde \psi$ induces a bundle isomorphism 
denoted by the same symbol 
from $f^{\ast}Q \to M$ to $f^{\ast}\tilde \psi Q\to M$, 
which is the pull-back bundle of the quotient bundle by the composition $\psi \circ f$.  
\begin{defn}
Let $V\to M$ be a vector bundle 
and $f$ a map from $M$ into $Gr_p(\mathbf K^m)$ such that 
$f^{\ast}Q\to M$ is isomorphic to $V\to M$. 
We consider a pair $(f, \phi)$, where $\phi:V \to f^{\ast}Q$ is a bundle isomorphism. 
Then such pairs $(f_1, \phi_1)$ and $(f_2, \phi_2)$ are called {\it gauge equivalent}, 
if there exists an isometry $\psi$ of $Gr_p(\mathbf K^m)$ such that 
$f_2=\psi\circ f_1$ and $\phi_2=\tilde \psi\circ \phi_1$.
\end{defn}

By definition, gauge equivalence yields 
image equivalence of maps. 

We are now in position to state the main theorem in this section. 
\begin{thm}\label{GGenDW}
Let $(M,g)$ be a compact Riemannian manifold. 
We fix a vector bundle 
$V \to M$ 
of rank $q$ with a fibre metric $h_V$ and a connection $\nabla$ preserving the metric. 

Let 
$f:M \to Gr_p(\mathbf K^m)$ be a full harmonic map 
satisfying the following two conditions. 

\noindent{\rm (i)} 
The pull-back bundle 
$f^{\ast}Q \to M$ 
with the pull-back metric and connection 
is gauge equivalent to $V\to M$ with $h_V$ and $\nabla$. 
{\rm (}Hence, $q=m-p$.{\rm )}

\noindent{\rm (ii)} 
The mean curvature operator $A \in \Gamma(\text{\rm End}\,V)$ of $f$ 
is expressed as 
$-\mu Id_V$ for some positive real number $\mu$, 
and so $e(f)=\mu q$. 

Then there exist an eigenspace $W \subset \Gamma(V)$ with eigenvalue 
$\mu$ of the Laplacian equipped with the $L^2$-scalar product $(\cdot,\cdot)_W$ 
and a semi-positive Hermitian endomorphism 
$T\in \text{\rm End}\,(W)$. 
The pair $(W,T)$ satisfies the following four conditions.

\noindent {\rm (I)} The vector space $\mathbf K^m$ can be regarded as a subspace of $W$ 
with the inclusion $\iota:\mathbf K^m \to W$ 
and $V\to M$ is globally generated by $\mathbf K^m$ (and so, by $W$). 
In particular, $\text{\rm dim}\,\mathbf K^m \leqq \text{\rm dim}\,W$. 

\noindent {\rm (II)} 
As a subspace, $\iota\left(\mathbf K^m\right)=\text{\rm Ker}\,T^{\bot}$, 
and the restriction of $T$ is a positive Hermitian transformation of $\mathbf K^m$. 

\noindent {\rm (III)} 
The endomorphism $T$ satisfies 
\begin{equation}\label{GDW 2}
ev \circ T^2 \circ ev^{\ast}=Id_V, 
\,\, 
ev \circ T^2 \circ \left(\nabla ev^{\ast}\right)=0 \in \Omega^1(\text{\rm End}\,V), 
\end{equation}
where $ev^{\ast}:V \to \underline{W}$ is the adjoint homomorphism of $ev:\underline{W} \to V$ 
with respect to 
$(\cdot,\cdot)_W$ and $h_V$, 
and the connection on $\text{\rm Hom}\,\left(V, \underline{W}\right)$ 
is induced by the product connection on $\underline{W}\to M$ and $\nabla$ on $V \to M$, 
which is denoted by the symbol $\nabla$. 

\noindent {\rm (IV)} 
The endomorphism $T$ provides an embedding of $Gr_p(\mathbf K^m)$ into $Gr_{p^{\prime}}(W)$, 
where $p^{\prime}=p+\text{\rm dim}\,\text{\rm Ker}\,T$
and also provides a bundle isomorphism $\phi:V\to f^{\ast}Q$ 
(the natural identification).

Then, $f:M \to Gr_{p}(\mathbf K^m)$ 
can be expressed as 
\begin{equation}\label{GDW3} 
f\left(x \right)=\left(\iota^{\ast} T \iota\right)^{-1} \left(f_0(x) \cap \iota\left(\mathbf K^{m}\right)\right), 
\end{equation}
where $\iota^{\ast}:W \to \mathbf K^m$ denotes the adjoint linear map of $\iota$ 
and $f_0$ is the standard map by $W$. 

The pairs $(f_1, \phi_{1})$ and $(f_2, \phi_{2})$ are gauge equivalent if and only if 
$$
\iota_1^{\ast}T_1\iota_1=\iota_2^{\ast}T_2\iota_2, 
$$
where $(T_i,\iota_i)$ corresponds to $f_i$ $(i=1,2)$ in expression \eqref{GDW3}. 

Conversely, 
suppose that a vector space $\mathbf K^m$, an eigenspace $W \subset \Gamma(V)$ 
with eigenvalue $\mu$ and 
a semi-positive Hermitian endomorphism 
$T\in \text{\rm End}\,(W)$ 
satisfying 
conditions {\rm (I)}, {\rm (II)} 
and {\rm (III)} are given.  
Then there is a unique embedding of $Gr_p(\mathbf K^m)$ into $Gr_{p^{\prime}}(W)$ 
and 
the map $f:M \to Gr_{p}(\mathbf K^m)$ defined by \eqref{GDW3}
is a full harmonic map into $Gr_p(\mathbf K^m)$ 
satisfying conditions {\rm (i)} and {\rm (ii)} with bundle isomorphsm $V\cong f^{\ast}Q$. 
\end{thm}

We begin with a lemma needed for the proof of Theorem \ref{GGenDW}. 
\begin{lemma}\label{linalg}
Let $W$ be a vector space with a scalar product 
and $\mathbf K^m$ be a subspace of $W$. 
The orthogonal projection is denoted by $\pi:W \to \mathbf K^m$. 
Let $V$ be a vector space with a scalar product and suppose that we have a surjective linear map 
$ev:W \to V$. 

If the restriction of $ev$ to $\mathbf K^m$ denoted by $ev_K$ is also surjective, 
then, 
$$
\pi\left( \text{\rm Ker}\,ev^{\bot}\right)=\left( \text{\rm Ker}\,ev_K\right)^{\bot}, 
$$
where $\text{\rm Ker}\,ev^{\bot}$ {\rm (}resp. $\left( \text{\rm Ker}\,ev_K\right)^{\bot}${\rm )} 
denotes the orthogonal complement of $\text{\rm Ker}\,ev$ of $W$ 
{\rm (} resp. $\text{\rm Ker}\,ev_K$ of $\mathbf K^m$  
endowed with the induced scalar product{\rm )}. 
\end{lemma}

\begin{proof}
Using the scalar product, 
we have adjoint homomorphisms $ev^{\ast}$ and $ev_K^{\ast}$ of $ev$ and $ev_K$, respectively. 
From the hypothesis, we have $ev\circ i=ev_K$ on $\mathbf K^m$, where 
$i:\mathbf K^m \to W$ is the inclusion map. 
The adjoint of $i$ is nothing but the projection $\pi$. 
It follows that $ev_K^{\ast}=i^{\ast}\circ ev^{\ast}$ and so, $ev_K^{\ast}=\pi \circ ev^{\ast}$.
\end{proof}

\newtheorem{mainpr}{Proof of Theorem 5.5}
\renewcommand{\themainpr}{}
\begin{mainpr}
Suppose that 
$f:M\to Gr_p(\mathbf K^m)$ is a full harmonic map 
satisfying conditions (i) and (ii). 
This implies that $\mathbf K^m$ globally generates $V\to M$. 

Then 
Theorem \ref{harmGrirr} together with condition (ii) immediately yields that 
$
\Delta t= \mu t, 
$
for an arbitrary $t \in \mathbf K^m$. 
It follows from the definition of full map that 
$$
\mathbf K^m \subset W=W_{\mu},  
$$
with inclusion $\iota:\mathbf K^m \to W$. 
Hence we obtain condition (I). 
In particular, $m \leqq N$, where $N={\rm dim}\,W$. 

Though the scalar product $(\cdot,\cdot)_W$ can be 
restricted to $\mathbf K^m$ to obtain a scalar product 
denoted by the same symbol, 
we also have the original scalar product 
$(\cdot,\cdot)_{m}$ 
on $\mathbf K^m$
which induces the Riemannian metric on 
$Gr_p(\mathbf K^m)$. 
Consequently, we have a 
positive Hermitian transformation 
$\underline{T}:\mathbf K^m \to \mathbf K^m$ 
such that 
\begin{equation}\label{posT}
(\underline{T}\cdot,\underline{T}\cdot)_{m}=(\cdot,\cdot)_W. 
\end{equation}
An isometry $\underline{T}^{-1}:\left(Gr_p(\mathbf K^m),(\cdot,\cdot)_{m}\right)  \to 
\left(Gr_p(\mathbf K^m), (\cdot,\cdot)_{W}\right)$ is given by 
$U \mapsto \underline{T}^{-1}U$, where $U$ is a $p$-dimensional subspace of $\mathbf K^m$. 

Since $\mathbf K^m$ globally generates $V\to M$, 
$W$ also does. 
Hence we have a surjective evaluation homomorphism $ev:\underline{W} \to V$ 
and the standard map $f_0$ by $W$. 
The restriction of $ev$ to $\mathbf K^m$ gives the surjective 
evaluation homomorphism $ev_K:\underline{\mathbf K^m} \to V$. 
Then the map $f$ is expressed as 
$$
f(x)=\text{Ker}\,ev_{K_{x}}. 
$$
Using the composition,
$\underline{T}^{-1}\circ f:M \to \left(Gr_p(\mathbf K^m), (\cdot,\cdot)_{W}\right)$ 
is also a full harmonic map satisfying conditions (i) and (ii). 
Thus, from now on, we consider $\underline{T}^{-1}\circ f$ which is refered to simply as 
$f:M \to \left(Gr_p(\mathbf K^m), (\cdot,\cdot)_{W}\right)$. 
Then we have 
\begin{equation}\label{Gexpf}
f(x)=\underline{T}^{-1}\text{Ker}\,ev_{K_{x}}. 
\end{equation}
Since 
$$
\left(\underline{T}^{-1}\text{\rm Ker}\,ev_K\right)^{\bot}
=\underline{T}\left(\text{\rm Ker}\,ev_K\right)^{\bot}, 
$$
we apply lemma \ref{linalg} to obtain 
\begin{equation}\label{GTT}
\left(\underline{T}^{-1}\text{\rm Ker}\,ev_K\right)^{\bot}
=\underline{T}\iota^{\ast}\left(\text{\rm Ker}\,ev^{\bot}\right). 
\end{equation}
Let $\mathbf K^{m^{\bot}}$ be the orthogonal complement of $\mathbf K^m$ in $W$. 
We define a semi-positive Hermitian endomorphism $T:W\to W$ in such a way that 
$K^{m^{\bot}}$ is the eigenspace with eigenvalue $0$ and 
$T|_{\mathbf K^{m}}=\underline{T}$. 
More precisely, the latter condition means that $\iota^{\ast}T \iota =\underline{T}$. 
Consequently we have condition (II). 

Since $\text{Ker}\,ev_K=\text{Ker}\,ev \cap \mathbf K^m$ and $f_0(x)=\text{Ker}\,ev_{x}$, 
the definiton of $T$ and \eqref{Gexpf} yields that $f$ is expressed as 
$$
f\left(x\right)=\left(\iota^{\ast}T\iota \right)^{-1}\left(f_0\left(x\right)\cap \text{\rm Ker}\,T^{\bot}\right).
$$

We will describe the common zero set $Z$ of sections of $\text{Ker}\,T$ in $Gr_{p^{\prime}}(W)$. 
By definition, 
\begin{align*}
Z=&\left\{U^{\prime} \subset W\,|\, \text{dim}\,U^{\prime}=p^{\prime} \,\,\text{and}\,\, \text{Ker}\,T \subset U^{\prime} \right\} \\
 =&\left\{U\oplus \text{Ker}\,T \subset W\,|\,U \subset \text{Ker}\, T^{\bot}\, \text{and}\,\,\text{dim}\,U=p \right\}, 
\end{align*}
which gives the embedding of $Gr_p(\mathbf K^m)$ into 
$Gr_{p^{\prime}}(W)$ with image $Z$ . 

When $f^{\ast}Q \to M$ is identified with the orthogonal complememt of $f^{\ast}S \to M$, 
it follows from \eqref{GTT} that the natural identification 
$\phi:V \to f^{\ast}Q$ is expressed as:  
\begin{equation}\label{GnatidT}
\phi_x\left(v\right)=\left(x, T\circ ev^{\ast} (v)\right)
\left(=\left(x, \left(\iota^{\ast} T \iota\right) \iota^{\ast}  ev^{\ast} (v)\right) \right), 
\end{equation}
where $v\in V_x$ and $T\circ ev^{\ast}$ is considered as a map onto the image. 
Notice that $Tw=T\iota \iota^{\ast} w$ for any $w \in W$ by definition of $T$. 
Therefore we have condition (IV). 

Since the metric $h_{f^{\ast}Q}$ 
on $f^{\ast}Q\to M$ is induced from the scalar product on $\mathbf K^m \subset W$, 
it follows from condition (i) and \eqref{GnatidT} that 
\begin{equation}\label{GisomVV}
h_V(v,v^{\prime})=h_{f^{\ast}Q}\left(\phi(v), \phi(v^{\prime})\right)
=(T\circ ev^{\ast}(v), T\circ ev^{\ast}(v^{\prime}))_W, 
\end{equation}
for arbitrary $v, v^{\prime} \in V$. 
Hence we obtain 
\begin{equation}\label{GBV}
ev \circ T^2 \circ ev^{\ast}=Id_V. 
\end{equation}


Next, we compare the given connection $\nabla$ with the induced connection $\nabla^Q$ 
on $V \to M$ 
by $f:M\to Gr_p(\mathbf K^m)$. 
Since the orthogonal projection $\underline{\mathbf K^m}\to V$ is $ev \circ T$, which is 
the adjoint homomorphism of $\phi$, 
the induced connection is calculated as follows:
\begin{align*}
\nabla^Q t=&ev \circ T \circ d \phi(t)
=ev \circ T \left(x, {T}d\left(ev^{\ast}(t) \right) \right)\\
=&ev \circ T \left(x, {T}\left(\nabla ev^{\ast}\right)(t)\right)
+ev \circ T \left(x, {T}\circ ev^{\ast}(\nabla t)\right), \\
=&ev \circ T^2\circ \left(\nabla ev^{\ast}\right)(t)
+\phi^{\ast}\phi(\nabla t), 
\end{align*}
for an arbitrary section $t \in \Gamma(V)$. 
It follows from \eqref{GBV} that $\phi^{\ast}\phi=Id_V$, and so, 
\begin{equation}\label{GIII2}
\nabla^Q -\nabla 
=ev \circ T^2 \circ \left( \nabla ev^{\ast}\right).
\end{equation}
It follows from \eqref{GBV} and \eqref{GIII2} that condition (III) holds. 

Suppose that those two pairs $(f_1,\phi_1)$ and $(f_2,\phi_2)$ 
are gauge equivalent as maps into $Gr_p(\mathbf K^m)$. 
The corresponding endomorphisms on $W$ are denoted by $T_1$ and $T_2$, respectively. 
To destinguish two subspaces of $W$ which are isomorphic to $\mathbf K^m$, 
we emphasize the role of the inclusions. 
From \eqref{GnatidT}, the natural identification $\phi_i$ ($i=1,2$) is expressed as 
$$
\phi_i(v)=\left(x, \left(\iota_i^{\ast} T_i \iota_i\right) \iota_i^{\ast}  ev^{\ast} (v)\right).
$$
By definition of gauge equivalence, 
there exisits an isometry $\psi$ of $Gr_p(\mathbf K^m)$ such that 
\begin{equation}\label{GT1T2}
\tilde \psi \phi_1(v) = \phi_2 (v) \Longleftrightarrow 
\tilde \psi \left(\iota_1^{\ast} T_1 \iota_1\right) \iota_1^{\ast} ev^{\ast}(v) = 
\left( \iota_2^{\ast} T_2 \iota_2\right) \iota_2^{\ast}ev^{\ast}(v), 
\end{equation}
for an arbitrary $v\in V$. 
Since $f_1$ and $f_2$ are full maps, \eqref{GT1T2} gives 
$$
\tilde \psi \iota^{\ast}T_1\iota_1=\iota_2^{\ast}T_2\iota_2. 
$$
Then uniqueness of the polar decomposition yields that $\tilde \psi=Id_{\mathbf K^m}$ and 
$\iota_1^{\ast}T_1\iota_1=\iota_2^{\ast}T_2\iota_2$. 

Suppose that a subspace $\mathbf K^m$ of an eigenspace $W$ 
with an eigenvalue $\mu$ and 
a semi-positive Hermitian endomorphism 
$T\in \text{\rm End}\,(W)$ 
satisfy
conditions {\rm (I)}, {\rm (II)} 
and {\rm (III)}. 
Then we can define a full map $f:M \to Gr_p(\mathbf K^m)$ as in \eqref{GDW3} 
and an embedding of $Gr_p(\mathbf K^m)$ into $Gr_{p^{\prime}}(W)$ as the zero set
of sections of $\text{Ker}\,T$. 
Moreover we also define a bundle isomorphism $\phi:V\to f^{\ast}Q$ as in \eqref{GnatidT}. 
Then we obtain condition (i), because $\phi$ preserves the metrics and the connections from 
the above computations. 
Hence we can apply Theorem \ref{harmGrirr} to conclude that $f$ is a full harmonic map and $A=-\mu Id_V$. 
\end{mainpr}

\begin{rem}
Since conditions (i) and (ii) in Theorem \ref{GGenDW} are repeated several times in the paper, 
condition (i) is called {\it gauge condition} and condition (ii) is called 
{\it Einstein-Hermitian condition} or {\it EH condition} for short. 
The gauge condition may also be called the {\it balanced condition} (cf. \cite{Do-3}). 
\end{rem}

\begin{rem}
When the sphere $S^{N-1}$ is identified with an oriented Grassmannian of hyperplanes $Gr_{N-1}(\mathbf R^N)$, 
the position vector gives a trivialization of $Q\to Gr_{N-1}(\mathbf R^N)$ and so, a product connection. 
Thus, this connection gives the usual differential of functions. 
When we have a map $f:M\to Gr_{N-1}(\mathbf R^N)$, $f$ can also be recognized as a trivialization of $f^{\ast}Q \to M$. 
Then every section of the pull-back bundle can be recognized as a function on $M$, 
and the pull-back connection gives the usual differentiation. 
Consequently when the target is the sphere, we can drop gauge condition {\rm (i)} in Theorem \ref{GGenDW}.  
\end{rem}

\begin{rem}
When the target is a symmetric space of rank $1$, 
the quotient bundle is also of rank $1$. 
Consequently, the mean curvature operator can be considered as a function. 
Hence the EH condition in Thereom \ref{GGenDW} is equivalent to 
$f$ having constant energy density. 
\end{rem}

\begin{rem}
The role of condition (IV) in Theorem \ref{GGenDW} should be emphasized. 
Whenever we consider a full harmonic map from $(M,g)$ into $Gr_p(\mathbf K^m)$ satisfying gauge and EH conditions, 
we also have an embedding of $Gr_p(\mathbf K^m)$ into $Gr_{p^{\prime}}(W)$ and a bundle isomorphism $V \to f^{\ast}Q$, 
which is the natural identification. 
This enables us to take a quotient of the set of those maps into $Gr_{p^{\prime}}(W)$ 
by {\it gauge equivalence}. 
\end{rem}
Moreover, Theorem \ref{GGenDW} provides a geometric meaning to the compactification of the moduli of those {\it full} maps into $Gr_{p^{\prime}}(W)$ 
by the induced topology. 
To see this,  we need the following

\begin{lemma}\label{traceT}
Let $(M,g)$ be a compact Riemannian manifold. 
We fix a vector bundle 
$V \to M$ 
of rank $q$ with fibre metric $h_V$ and connection $\nabla$ preserving the metric. 

Suppose that $f:M \to Gr_p(\mathbf K^m)$ is a harmonic map  
satisfying the gauge condition for $(V\to M, h_V, \nabla)$ and the Einstein-Hermitain condition. 
Let $T$ be the semi-positive Hermitian endomorphism on $W$ associated to $f$ in Theorem \ref{GGenDW}, 
where $W \subset \Gamma(V)$ is an eigenspace of the Laplacian equipped with the $L^2$-scalar product $(\cdot,\cdot)_W$. 

Then 
\begin{equation}\label{trT}
\text{\rm trace}\,T^2=q\text{\rm Vol}(M), 
\end{equation}
where, $\text{\rm Vol}(M)$ denotes the volume of $M$. 
\end{lemma}
\begin{proof}
Let $w_1,\cdots, w_N$ be a unitary basis of $W$, where $N:=\text{\rm dim}\,W$. 
Under the given assumptions, we claim that 
$$
\sum_{A=1}^N\left(Tw_A, Tw_A \right)_W=q\text{\rm Vol}(M).
$$
To see this, we use the definition of $L^2$-scalar product to get 
\begin{align}\label{trace}
\sum_{A=1}^N\left(Tw_A,Tw_A \right)_W=&\sum_{A=1}^N\int_M h_V\left(ev(Tw_A),ev(Tw_A) \right)dv \\
=&\int_M \sum_{A=1}^N h_V\left(ev(Tw_A),ev(Tw_A) \right)dv, \notag
\end{align}
where $dv$ denotes the volume form on $M$. 
Fix a point $x \in M$ and 
let $v_1,\cdots, v_q$ be a unitary basis of $V_x$. 
Since $f$ satisfies the gauge condition by assumption, 
we have 
$\left(w_i, w_j \right)_W=h_V\left(ev(Tw_i), ev(Tw_j) \right)$, 
for $w_i, w_j \in f^{\ast}Q_x$. 
If $w \in W$ is perpendicular to $f^{\ast}Q_x$, then 
$$
0=(Tev^{\ast}(v_i), w)_W=(ev^{\ast}(v_i), Tw)_W=h_V(v_i, ev(Tw)). 
$$
Hence, if necessary, we can change to another unitary basis of $W$ such that 
$ev(Tw_i)=v_i$ for $i=1,\cdots,q$ and $ev(Tw_{q+j})=0$ for $j=1,\cdots N-q$. 
However, the sum $\sum_{A=1}^N h_{V_x}\left(ev(Tw_A),ev(Tw_A) \right)$ does not change, 
because $h_{V_x}\left(evT, evT \right)$ can be considered as an Hermitian form on $W$. 
Therefore, 
$$
\sum_{A=1}^N h_V\left(ev(Tw_A),Tev(w_A) \right)=\sum_{i=1}^q h_V\left(v_i, v_i \right)=q. 
$$

Combining this with \eqref{trace}, we obtain the desired formula.
\end{proof}
\begin{cor}\label{traceIT}
Under the hypothesis of Lemma \ref{traceT}, 
suppose that the standard map $f_0$ by $(V\to M,W)$ satisfies the gauge and Einstein-Hermitian conditions. 
Then 
\begin{equation}
\text{\rm trace}\,T^2=N.
\end{equation}
\end{cor}
\begin{proof}
Since the identity transformation $Id$ on $W$ corresponds to the standard map, we get 
$$
N=\text{\rm trace}\,Id^2=q\text{\rm Vol}(M). 
$$
\end{proof}

\begin{rem}
Let us discuss the moduli space by gauge equivalence. 

In addition to the hypothesis in Lemma \ref{traceT}, 
suppose that the standard map $f_0$ by $(V\to M,W)$ is also a harmonic map which 
satisfies the gauge and the Einstein-Hermitian conditions.
If we put $C=T^2-Id_W$, then it follows from Theorem \ref{GGenDW} and Lemma \ref{traceT} that 
$C$ is trace-free and satisfies 
$$
ev \circ C \circ ev^{\ast}=0, 
\,\, \text{and} \,\,
ev \circ C \circ \left(\nabla ev^{\ast}\right)=0 \in \Omega^1(\text{\rm End}\,V). 
$$
Since the both equations are linear with respect to $C$, 
$C$ belongs to a subspace $\text{M}(W)$ of the set of trace-free Hermitian endomorphisms on $W$. 

If $\text{M}(W)$ is non-trivial, 
notice that $I+C$ is positive for $C$ small enough. 
Hence, we can apply do Carmo-Wallach's argument \cite[\S 5.1]{DoC-Wal} and Theorem \ref{GGenDW} to deduce that 
the moduli space of {\it full} harmonic maps into $Gr_{p^{\prime}}(W)$ satisfying gauge and EH conditions 
is a bounded connected {\it open} convex body in $\text{M}(W)$ with topology induced by $L^2$-scalar product. 

Under the natural compactification of the moduli in the topology,  
Theorem \ref{GGenDW} implies that each boundary point $C \in \text{M}(W)$ (where, $I+C$ is not positive, but semi-positive,) 
determines 
an {\it embedding} of the zero set into $Gr_{p^{\prime}}(W)$ determined by $\text{Ker}\,T \subset \Gamma(V)$, 
a {\it full harmonic map} into the zero set satisfying the gauge and the EH conditions, 
and a {\it bundle isomorphism}. 

If $\text{M}(W)=\left\{0 \right\}$, then the standard map is the unique full harmonic map satisfying 
gauge and EH conditions up to gauge equivalence. 

We give some examples later in which the above argument can be applied. 
\end{rem}

If $V\to M$ is holonomy irreducible in Theorem \ref{GGenDW}, then condition (III) reduces to a single equation, 
which will be useful to describe moduli spaces in later chapters 
(in particular, see Theorem \ref{mod} and its proof).  

\begin{prop}\label{Holsuf}
Let $(M,g)$ be a compact Riemannian manifold. 
We fix a vector bundle 
$V \to M$ 
of rank $q$ with fibre metric $h_V$ and connection $\nabla$ preserving the metric. 
Assume that $V\to M$ is holonomy irreducible with respect to $\nabla$. 

Let $W \subset \Gamma(V)$ be an eigenspace with eigenvalue 
$\mu$ of the Laplacian equipped with the $L^2$-scalar product $(\cdot,\cdot)_W$ 
and $\mathbf K^m$ a subspace of $W$. 

Suppose that $T\in \text{\rm End}\,(W)$ is a semi-positive Hermitian endomorphism on $W$ 
satisfying the following three conditions.

\noindent {\rm (I)} The vector space $\mathbf K^m$ is regarded as a subspace of $W$ 
with the inclusion $\iota:\mathbf K^m \to W$, 
and $V\to M$ is globally generated by $\mathbf K^m$. 

\noindent {\rm (II)} 
As a subspace, $\iota\left(\mathbf K^m\right)=\text{\rm Ker}\,T^{\bot}$, 
and the restriction of $T$ is a positive Hermitian transformation of $\mathbf K^m$. 

\noindent {\rm (III)} 
The endomorphism $T$ satisfies 
\begin{equation}\label{GDW 3}
ev \circ T^2 \circ \left(\nabla ev^{\ast}\right)=0 \in \Omega^1(\text{\rm End}\,V). 
\end{equation}

Then we can find a semi-positive Hermitian endomorphism $\tilde T$ on $W$ such that 
\begin{equation}\label{GDW 4}
ev \circ \tilde T^2 \circ ev^{\ast}=Id_V, \,\text{and}\quad 
ev \circ \tilde T^2 \circ \left(\nabla ev^{\ast}\right)=0.
\end{equation}
Moreover, $T$ and $\tilde T$ give the same harmonic map from $(M,g)$ into 
$Gr_p(\mathbf K^m)$, where $p=m-q$  
and the same embedding of $Gr_p(\mathbf K^m)$ into $Gr_{p^{\prime}}(W)$, 
where $p^{\prime}=\text{dim}\,W-q$. 
\end{prop}
\begin{proof}
Since $T$ is a Hermitian endomorphism on $W$, it follows from \eqref{GDW 3}
\begin{align*}
0=&h_V\left( ev \circ T^2 \circ \left(\nabla ev^{\ast}\right)(v_1),v_2\right)
=\left(T^2 \circ \left(\nabla ev^{\ast} \right)(v_1), ev^{\ast}(v_2)\right) \\
=&\left(\left(\nabla ev^{\ast} \right)(v_1), T^2 \circ ev^{\ast}(v_2)\right)
=h_V\left(v_1, \left(\nabla ev\right) \circ T^2 \circ ev^{\ast}(v_2)\right),
\end{align*}
for arbitrary $v_1, v_2 \in V$. 
Hence a bundle endomorphism on $V\to M$ defined by $ev\circ T^2 \circ ev^{\ast}$ is parallel 
with respect to the induced connection on $\text{End}\,V \to M$ from $\nabla$. 
From the assumption that $V\to M$ is holonomy irreducible, there exists a real number $c$ such that 
$$
ev\circ T^2 \circ ev^{\ast}=cId_V. 
$$
Since $\mathbf K^m$ globally generates $V\to M$ and $T$ is positive on $\mathbf K^m$, 
$c$ is a positive number. 
Thus $\tilde T:=\frac{1}{\sqrt{c}}T$ is the desired endomorphism on $W$. 
\end{proof}
Let $\text{H}_0(W)$ be the set of trace-free Hermitian endomorphisms on $W$. 
From the definition of the inner product on $\text{H}(W)$, every element of 
$\text{H}_0(W)$ is perpendicular to the identity. 

\begin{cor}\label{Holsuftf}
Let $(M,g)$ be a compact Riemannian manifold. 
Fix a vector bundle 
$V \to M$ 
of rank $q$ with fibre metric $h_V$ and connection $\nabla$ preserving the metric. 
Assume that $V\to M$ is holonomy irreducible with respect to $\nabla$. 

Let $W \subset \Gamma(V)$ be an eigenspace with eigenvalue 
$\mu$ of the Laplacian equipped with the $L^2$-scalar product $(\cdot,\cdot)_W$ 
and $\mathbf K^m$ a subspace of $W$. 
Assume that the standard map $f_0$ by $(V\to M, W)$ is a harmonic map 
satisfying the gauge and the Einstein-Hermitian conditions. 

Suppose that $W$ and $C\in \text{\rm H}_0\,(W)$ 
satisfy the following three conditions.

\noindent {\rm (I)} The vector space $\mathbf K^m$ is regarded as a subspace of $W$ 
with the inclusion $\iota:\mathbf K^m \to W$ 
and $V\to M$ is globally generated by $\mathbf K^m$. 

\noindent {\rm (II)} 
The Hermitian endomorphism $Id+C$ is semi-positive, 
$\iota\left(\mathbf K^m\right)=\text{\rm Ker}\,(Id+C)^{\bot}$, 
and 
the restriction of $Id+C$ is a positive Hermitian transformation of $\mathbf K^m$. 

\noindent {\rm (III)} 
The endomorphism $C$ satisfies 
\begin{equation}\label{GDW 5}
ev \circ C \circ \left(\nabla ev^{\ast}\right)=0 \in \Omega^1(\text{\rm End}\,V). 
\end{equation}

Then we have a unique embedding of $Gr_p(\mathbf K^m)$ into $Gr_{p^{\prime}}(W)$ 
and 
the map $f:M \to Gr_{p}(\mathbf K^m)$ defined by $\left(\iota^{\ast}\sqrt{I+C}\iota\right)^{-1}\iota^{\ast}f_0$  
is a full harmonic map into $Gr_p(\mathbf K^m)$ 
satisfying the gauge and the Einstein-Hermitian conditions with bundle isomorphism $V\cong f^{\ast}Q$. 
\end{cor}
\begin{proof}
Define a semi-positive Hermitian endomorphism $T$ in such a way that $T^2:=Id+C$. 
Since the standard map is harmonic and satisfies the gauge and EH conditions by hypothesis, 
Theorem \ref{GGenDW} implies that 
$$
ev \circ Id \circ ev^{\ast}=Id_V, \,\,\text{and}\,\,ev \circ Id \circ \nabla \left(ev^{\ast}\right)=0. 
$$
Then the assumptions on $C$ and Proposition \ref{Holsuf} yield that 
there exists a semi-positive Hermitian endomorphism $\tilde T$ defined by 
$\tilde T:=\frac{1}{\sqrt{c}}T$, where $c$ is a positive number,  which satisfies 
$
ev\circ \tilde T^2 \circ ev^{\ast}=Id_V
$ and 
$
ev\circ \tilde T^2 \circ \nabla ev^{\ast}=0.
$
By definition of $T$, we get 
$$
\tilde T^2= \frac{1}{c}Id + \frac{1}{c}C.
$$
It follows from Corollary \ref{traceIT} that 
$$
N=\text{trace}\,\tilde T^2=\frac{1}{c}N + \frac{1}{c}0=\frac{1}{c}N,
$$
and so, $c=1$. 
Hence Theorem \ref{GGenDW} yields the result. 
\end{proof}

Next we consider the image equivalence relation. 
Suppose that two full harmonic maps $f_1$, $f_2$ satisfying gauge and EH conditions are 
image equivalent as maps into $Gr_p(\mathbf K^m)$. 
The corresponding endomorphisms on $W$ are denoted by $T_1$ and $T_2$, respectively:
$$
f_i\left(x\right)=\left(\iota_i^{\ast}T_i\iota_i \right)^{-1}\left(f_0\left(x\right)\cap \text{\rm Ker}\,T_i^{\bot}\right). 
$$
where $i=1,2$. 
By definiton of image equivalence, we have an isometry $\psi$ of $Gr_p(\mathbf K^m)$ such that $f_2=\psi\circ f_1$ and so, 
$f_2^{{\ast}}Q=f_1^{{\ast}}\tilde \psi Q$ as a set. 
Using the bundle isomorphisms $\phi_1$ and $\phi_2$ obtained in Theorem \ref{GGenDW}, 
we get two bundle isomorphisms $\tilde \psi\circ \phi_1$ and $\phi_2: V \to f_2^{\ast}Q$. 
Hence we have a gauge transformation $\phi_2^{-1}\tilde \psi \phi_1$ of $V\to M$ 
preserving the metric and the connection. 
Such a gauge transformation belongs to the centralizer of the holonomy group of 
the connection in the structure group $\text{Aut}\,V$ 
of $V$ (if the base manifold is connected). 
Hence, 
$$
\phi_2^{-1}\tilde \psi \phi_1(v)=cv, \quad v \in V, \quad c \in \text{Aut}\,V, 
$$
where $c$ is an element of the centralizer regarded as a subgroup of $\text{Aut}\,V$. 
Then \eqref{GnatidT} yields that 
\begin{equation}\label{oimeq}
\tilde \psi \iota_1^{\ast}T_1 \iota_1 \iota_1^{\ast} ev^{\ast}(v)  = \iota_2^{\ast}T_2 \iota_2 \iota_2^{\ast}ev^{\ast}(cv). 
\end{equation}
Therefore the action of the centralizer of the holonomy group 
of the connection is also needed to be taken into account.

\begin{lemma}
Let $f_1$ and $f_2$ be full harmonic maps from $(M,g)$ into the sphere $Gr_n(\mathbf R^{n+1})$ with constant energy density. 
If $f_1$ and $f_2$ are image equivalent, then both are gauge equivalent as maps. 
\end{lemma}
\begin{proof}
We use the notation in Theorem \ref{GGenDW}. 
If the target is the sphere, 
the universal quotient bundle is a trivial bundle of real rank one 
and so, the structure group is trivial. 
The centralizer of the holonomy group is also trivial. 
It follows from \eqref{oimeq} that 
$$
\tilde \psi \iota_1^{\ast}T_1 \iota_1 \iota_1^{\ast} ev^{\ast}(v)  = \iota_2^{\ast}T_2 \iota_2 \iota_2^{\ast}ev^{\ast}(cv)
=\iota_2^{\ast}T_2 \iota_2 \iota_2^{\ast}ev^{\ast}(v).
$$
The fullness of $f_i$ implies that $\tilde \psi \iota_1^{\ast} T_1 \iota_1 =\iota_2^{\ast} T_2 \iota_2$. 
Uniqueness of the polar decomposition yields the result. 
\end{proof}

\begin{prop}\label{gaimC}
Suppse that 
$f_1$ and $f_2$ are full harmonic maps from $(M,g)$ 
into the complex Grassmannian $Gr_p(\mathbf C^m)$ 
satisfying the gauge and EH conditions 
and,  that the holonomy group of the connection on $V\to M$ is irreducible. 
If $f_1$ and $f_2$ are image equivalent, then both are gauge equivalent as maps.
\end{prop}
\begin{proof}
Since the holonomy group is irreducible, 
Schur's lemma yields that $c$ in \eqref{oimeq} is regarded as a scalar multiplication. 
Since $c$ gives a unitary transformation, $c$ is regarded as complex number $a$ with $|a|=1$.  
Then \eqref{oimeq} yields that 
$$
\tilde \psi \iota_1^{\ast}T_1 \iota_1 \iota_1^{\ast} ev^{\ast}(v)  = \iota_2^{\ast}T_2 \iota_2 \iota_2^{\ast}ev^{\ast}(cv)
= a\iota_2^{\ast}T_2 \iota_2 \iota_2^{\ast}ev^{\ast}(v).
$$
and so, $\tilde \psi \iota_1^{\ast} T_1 \iota_1 =a \iota_2^{\ast} T_2 \iota_2$ by fullness of maps.  
Uuniqueness of the polar decomposition yields that $\tilde \psi=aId_{\mathbf C^m}$ and we obtain the desired result. 
\end{proof}

\subsection{Holomorphic case}

When $f:(M,g,J)\to Gr_p(W)$ is a holomorphic map, 
Proposition \ref{EH} makes the generalization of do Carmo-Wallch theorem simpler. 
However, we must divide it in two theorems, according to the target. 

\begin{thm}\label{HGenDW}
Let $(M,g,J)$ be a compact K\"ahler manifold. 
Fix a holomorphic vector bundle $V\to M$ 
of rank $q$ with an Einstein-Hermitian metric $h_V$ and the Einstein-Hermitian connection $\nabla$ 
with $K_{EH}=\mu Id_V$. 

Let 
$f:M \to Gr_p(\mathbf C^m)$ be a full holomorphic map 
satisfying the gauge condition: 

\noindent{\rm (i)} 
There exists a holomorphic bundle isomorphism between 
the pull-back bundle of the universal quotient bundle and 
$V\to M$, which preserves the metrics. 
{\rm (}Hence, $q=m-p$.{\rm )}

Then we have the space of holomorphic sections $W$ of $V \to M$ which is 
also an eigenspace of the Laplacian with eigenvalue 
$\mu$ equipped with $L^2$-scalar product $(\cdot,\cdot)_W$ 
and a semi-positive Hermitian endomorphism 
$T\in \text{\rm End}\,(W)$. 
The pair $(W,T)$ satisfies the following four conditions.

\noindent {\rm (I)} The vector space $\mathbf K^m$ is a subspace of $W$ with the inclusion 
$\iota:\mathbf K^m \to W$ and 
$V\to M$ is globally generated by $\mathbf K^m$.

\noindent {\rm (II)} 
As a subspace, $\mathbf K^m=\text{\rm Ker}\,T^{\bot}$ and 
the restriction of $T$ is a positive Hermitian transformation of $\mathbf K^m$. 

\noindent {\rm (III)} 
The endomorphism $T$ satisfies 
\begin{equation}\label{HDW 2}
ev \circ T^2 \circ ev^{\ast}=Id_V, 
\end{equation}

\noindent {\rm (IV)} 
The endomorphism $T$ gives a holomorphic embedding of $Gr_p(\mathbf K^m)$ into $Gr_{p^{\prime}}(W)$, 
where $p^{\prime}=p+\text{\rm dim}\,\text{\rm Ker}\,T$
and also gives a bundle isomorphism $\phi:V\to f^{\ast}Q$. 


Then, $f:M \to Gr_{p}(\mathbf K^m)$ 
can be expressed as 
\begin{equation}\label{HDW3} 
f\left([g]\right)=\left(\iota^{\ast}T\iota \right)^{-1}\left(f_0\left([g]\right)\cap \text{\rm Ker}\,T^{\bot}\right), 
\end{equation}
where $\iota^{\ast}$ denotes the adjoint operator of $\iota$ under the induced scalar product on 
$\mathbf K^m$ from $(\cdot,\cdot)_W$ on $W$ and $f_0$ is the standard map by $W$. 
Such two pairs $(f_i, \phi_i)$, $(i=1,2)$ are gauge equivalent if and only if 
$
\iota_1^{\ast}T_1\iota_1=\iota_2^{\ast}T_2\iota_2, 
$
where $(T_i,\iota_i)$ correspond to $f_i$ $(i=1,2)$ under the expression in \eqref{HDW3}, respectively. 

Conversely, 
suppose that a vector space $\mathbf K^m$, the space of holomorphic sections 
$W \subset \Gamma(V)$ and 
a semi-positive Hermitian endomorphism 
$T\in \text{\rm End}\,(W)$ 
satisfying 
conditions {\rm (I)}, {\rm (II)} 
and {\rm (III)} are given.  
Then there is a unique holomorphic embedding of $Gr_p(\mathbf K^m)$ into $Gr_{p^{\prime}}(W)$ 
and 
the map $f:M \to Gr_{p}(\mathbf K^m)$ defined by \eqref{HDW3}
is a full holomorphic map into $Gr_p(\mathbf K^m)$ 
satisfying the gauge condition with bundle isomorphism $V\cong f^{\ast}Q$. 
\end{thm}
\begin{proof}
Since the compatible connection with holomorphic bundle structure and metric is unique, 
condition (i) yields that the pull-back connection is gauge equivalent to $\nabla$. 
Under the gauge condition, Proposition \ref{EH} implies that the mean curvature operator satisfies the EH condition. 

On the other hand, if we have a semi-positive Hermitian transformation $T:W \to W$, 
$ev \circ T$ gives a holomorphic bundle isomorphism between $V \to M$ and $f^{\ast}Q \to M$. 
Hence, for the same reason, condition (III) yields that $ev\circ T$ preserves the connections. 

The remaining assertion is proved in the same way as in the proof of Theorem \ref{GGenDW}. 
\end{proof}

From now on, we discuss the equation \eqref{HDW 2} for a holomorphic map. 
Since our presentation here follows closely the discussion in \cite{D-K} and \cite{Do-3}, 
readers may consult \cite{D-K} and \cite{Do-3} for more details.

Let $V\to M$ be an Einstein-Hermitian vector bundle  
of rank $q$ with Einstein-Hermitian metric $h_V$  over a compact $n$-dimensional K\"ahler manifold $M$ with K\"ahler form $\omega$. 
The group of bundle automorphisms of $V \to M$ preserving $h_V$ is refered to as the gauge group of $V\to M$ and denoted by 
$\mathcal G_V$. 
The vector bundle of skew-Hermitian endomorphisms of $V \to M$ is denoted by $\mathfrak{g}_V \to M$. 
The Lie algebra of $\mathcal G_V$ is regarded as the space of sections on $\mathfrak{g}_V \to M$ and denoted by 
$\Gamma(\mathfrak{g}_V)$. 

Let $W$ be the space of holomorphic sections of $V\to M$ which is equipped with the $L^2$-Hermitian inner product $(\cdot,\cdot)_W$. 
Then, $\text{Hom}(\underline{W}, V)$, the set of bundle homomorphisms from $\underline{W}\to M$ to $V\to M$, 
can be regarded as an infinite-dimensional K\"ahler manifold with metric: 
$$
\int_M \langle \Phi, \Psi \rangle \frac{\omega^n}{n!}, \quad \Phi, \Psi \in \text{Hom}(\underline{W}, V), 
$$
where $\langle \Phi, \Psi \rangle$ is the induced metric on $\underline{W}^{\ast}\otimes V \to M$ 
from $h_V$ and $(\cdot,\cdot)_W$. 

Then the gauge group acts on $\text{Hom}(\underline{W}, V)$ preserving the K\"ahler form. 

\begin{lemma}{\rm (}cf.\cite{D-K}{\rm )}\label{mom1}
Define a map $\mu_1:\text{\rm Hom}(\underline{W}, V) \to \Gamma(\mathfrak{g}_V)$ as 
$$
\mu_1(\Phi)=\sqrt{-1}\Phi \Phi^{\ast},
$$
where $\Phi \in \text{\rm Hom}(\underline{W}, V)$. 
Then $\mu_1$ is regarded as an equivariant moment map for the action of $\mathcal G_V$ on $\text{\rm Hom}(\underline{W}, V)$. 
\end{lemma}

Suppose that $\mu_1(\Phi)= \sqrt{-1}Id_V$. 
The condition $\Phi\Phi^{\ast}=Id_V$ yields that the metric $h_V$ is the same as the induced metric from $(\cdot,\cdot)_W$ by 
$\Phi$. 
To see this, let $v_1,\cdots,v_q$ be an unitary basis of $V_x$, where $x \in M$. 
Then, it follows from $\Phi\Phi^{\ast}=Id_V$ that
$$
h_V(v_i,v_j)=h_V\left(\Phi\Phi^{\ast}v_i,v_j \right)=\left(\Phi^{\ast}v_i, \Phi^{\ast}v_j\right)_W. 
$$
In particular, $\mu_1(\Phi)= \sqrt{-1}Id_V$ implies that $\Phi$ is a surjective bundle homomorphism. 
If $\Phi$ is a holomorphic bundle surjection, then $\Phi$ induces a holomorphic map $f$ from $M$ into $Gr_p(W)$, where $p=\text{dim}\,W-q$. 
In this case, the pull-back bundle of the universal quotient bundle is holomorphically isomorphic to $V \to M$. 
Since the compatible connection is unique, the pull-back connection is also the Einstein-Hermitian connection. 
Thus the equation \eqref{HDW 2} in condition (III) is recognized as the equation $\mu_1(\Phi)= \sqrt{-1}Id_V$. 

Next, the action of $\text{U}(W)$ on $\text{\rm Hom}(\underline{W}, V)$ is defined as 
$g\cdot \Phi=\Phi \circ g$, which also preserves the K\"ahler metric. 
The equivariant moment map $\mu_2:\text{\rm Hom}(\underline{W}, V) \to \mathfrak{u}(W)$ is given by 
$$
\mu_2(\Phi)=\sqrt{-1}\int_M \Phi^{\ast}\Phi \frac{\omega^n}{n!}. 
$$

Notice that the $\text{U}(W)$ action preserves $\mu_1^{-1}(\sqrt{-1}Id_V)$. 
From Theorem \ref{HGenDW}, the moduli space of full holomorphic maps into $Gr_p(W)$ satisfying the gauge condition 
by gauge equivalence 
is identified with the quotient of $\mu_1^{-1}(\sqrt{-1}Id_V) \cap \text{H}(W)$ by $\text{U}(W)$. 

However, we have a rigidity theorem in the case that $V\to M$ is irreducible. 

\begin{thm}\label{rigid}
Let $(M,g,J)$ be a compact K\"ahler manifold and 
$V\to M$ a holonomy irreducible Einstein-Hermitain vector bundle of rank $q$. 

Suppose that  
$f:M \to Gr_{p}(\mathbf C^{m})$ is a full holomorphic map satisfying the gauge condition:  

\noindent{\rm (i)} 
There exists a holomorphic bundle isomorphism between 
the pull-back bundle of the universal quotient bundle and 
$V\to M$, which preserves the metrics. 
{\rm (}Hence, $q=m-p$.{\rm )}



Then 
$m$ is uniquely determined 
and $\mathbf C^{m}$ can be regarded as a subspace of 
the space of holomorphic sections $H^0(M;V)$. 
The map $f$ is expressed as $\pi_m f_0$ up to image equivalence, 
where $\pi_m:W \to \mathbf C^{m}$ is the orthgonal projection and $f_0$ is the standard map by 
$(V \to M, H^0(M;V))$. 
\end{thm}
\begin{proof}
We denote by $W$ the space of holomorphic sections $H^0(M;V)$. 
Then $\mathbf C^{m}$ is a subspace of $W$ by Theorem \ref{main theorem} and fullness of the map. 

From Theorem \ref{HGenDW}, we have a surjective bundle homomorphism 
$ev \circ T :\underline{W}\to V$, where $ev:\underline{W} \to V$ 
is the evaluation map (for the standard map) and 
$T$ is a semi-positive Hermitian endomorphism of $W$. 
Moreover $\text{Ker}\,T^{\bot}=\mathbf C^{m}$ and the restriction of $T$ is a positive Hermitian transformation of $\mathbf C^{m}$, 
which is denoted by the same symbol. 
The map $f$ is obtained by the composition $T^{-1}\pi_m f_0$. 

We denote by  $f^{\ast}h_Q$ 
the pull-back metric  on $f^{\ast}Q \to M$. 
Since $T$ satisfies \eqref{HDW 2} or $\mu_1(ev\circ T)=\sqrt{-1}Id_W$, it follows that 
$f^{\ast}h_Q=h_V(ev\circ T, ev \circ T)$. 
By gauge condition, $f^{\ast}h_Q$ gives an Einstein-Hermitian strucutre on $V \to M$. 
We use the uniqueness of the Einstein-Hermitian structure again and irreducibility to deduce that 
$ch_V=h_V(ev\circ T, ev \circ T)$, where $c$ is a positive constant. 
It follows from Lemma \ref{linalg} or \eqref{GTT} that 
$ch_V(w_1, w_2)=h_V\left(ev\circ T(w_1), ev \circ T(w_2)\right)$ for arbitrary $w_1,w_2 \in \mathbf C^{m}$. 
Consequently, 
\begin{align*}
c(w_1, w_2)_W=&\int_M h_V\left(ev \circ T(w_1), ev \circ T (w_2)\right) \frac{\omega^n}{n!} \\
=&\int_M \left(T^{\ast}ev^{\ast}ev T(w_1), w_2\right)_W \frac{\omega^n}{n!}, 
\end{align*}
and so, 
$\mu_2(ev \circ T)=\sqrt{-1}cId_{\mathbf C^{m}}$. 
(This condition is closely related to the balanced condition \cite{Do-3}. Indeed, it is the balanced condition if $\mathbf C^m=W$.) 
In other words, we get 
\begin{equation}\label{Tmu2}
cId_{\mathbf C^{m+1}}=\int_M T^{\ast} ev^{\ast} ev T \frac{\omega^n}{n!}=T^{\ast} \left(\int_M ev^{\ast} ev  \frac{\omega^n}{n!}\right) T. 
\end{equation}
By definition of $L^2$-scalar product, we have that for arbitrary $w_1,w_2 \in W$, 
\begin{align*}
\left( \int_M ev^{\ast} ev  \frac{\omega^n}{n!}(w_1), w_2\right)_W=&\int_M\left(ev^{\ast}ev(w_1),w_2 \right)_W\frac{\omega^n}{n!} \\
=&\int_M h_V\left(ev(w_1), ev(w_2) \right)\frac{\omega^n}{n!}=(w_1,w_2)_W. 
\end{align*}
Therefore, 
$$
\int_M ev^{\ast}ev \frac{\omega^n}{n!}=Id_W. 
$$
Combining this with \eqref{Tmu2}, we obtain $T^{\ast}T=cId_{\mathbf C^{m}}$. 
Since $f$ is also expressed as $\sqrt{c}^{-1}T^{-1}\pi_m f_0$ and $\sqrt{c}T \in \text{\rm U}(\mathbf C^{m})$, 
we obtain $f=\pi_m f_0$ up to image equivalence. 

Finally, suppose that there exist two those maps $f_1:M \to Gr_{m-q}(\mathbf C^{m})$ and $f_2:M \to Gr_{n-q}(\mathbf C^{n})$. 
From the above argument, we can conclude that $f_1=\pi_m f_0$ and $f_2=\pi_n f_0$. 
We can assume that $\mathbf C^{m+1} \subset \mathbf C^{n+1}$ without loss of generality. 
Then $s\pi_m + (1-s)\pi_n$ ($0\leqq s < 1$) satisfies \eqref{HDW 2}. 
Theorem \ref{HGenDW} implies that  
we obtain a full holomorphic map $f:M \to Gr_{n-q}(\mathbf C^{n})$ satisfying the gauge condition, 
which are not gauge equivalent to $f_2$. 
Proposition \ref{gaimC} yields that $f$ and $f_2$ are not image equivalent, 
which is a contradiciton. 
Thus $m=n$. 
\end{proof}

\begin{rem}
We can develop a similar argument in the case of harmonic maps from a compact Riemannian manifold into complex Grassmannian 
using the Riemannian volume form. 
Let $M$ be a compact Riemannian manifold. 
We assume that $V\to M$ is a complex vector bundle with Hermitian metric and connection $\nabla$ preserving the metric. 
An eigenspace of the Laplace operator is denoted by $W$. 
Then we have two moment maps $\mu_1:\text{\rm Hom}(\underline{W}, V) \to \Gamma(\mathfrak{g}_V)$ and 
$\mu_2:\text{\rm Hom}(\underline{W}, V) \to \mathfrak{u}(W)$. 
Let $\mathcal A_V$ be the set of connections on $V\to M$. 
A $\mathcal G_V$-equivariant map 
$F:\text{\rm Hom}(\underline{W}, V) \to \mathcal A_V$ is defined as 
$$
F(\Phi)=\Phi d\Phi^{\ast}. 
$$
More precisely, this means that we define a covariant derivative $\nabla^{\Phi}$ for a section $t \in \Gamma(V)$ as 
$
\nabla^{\Phi} t =\Phi d\left(\Phi(t) \right). 
$
Then the equations \eqref{GDW 2} in condition (III) in Theorem \ref{GGenDW} is $\mu_1(\Phi)= \sqrt{-1}Id_V$ and $F(\Phi)=\nabla$ 
up to gauge transformation. 

Notice that the $\text{U}(W)$ action preserves $\mu_1^{-1}(\sqrt{-1}Id_V)$ and the equation $F(\Phi)=\nabla$. 
From Theorem \ref{GGenDW}, the moduli space of full harmonic maps into $Gr_p(W)$ satisfying gauge and EH conditions by 
gauge equivalence is identified with the quotient of $\mu_1^{-1}(\sqrt{-1}Id_V) \cap F^{-1}([\nabla]) \cap \text{H}(W)$ by $\text{U}(W)$. 
Here $F$ is regarded as a map from $\text{\rm Hom}(\underline{W}, V)$ to $\mathcal A_V/\mathcal G_V$. 

The complexification of $\text{SU}(W)$, which is denoted by $\text{SL}(W)$, also acts on $\text{\rm Hom}(\underline{W}, V)$ in a similar manner. 
Using the standard argumet, it can be shown that every $\text{SL}(W)$-orbit has at most one $\text{SU}(W)$-orbit satisfying $\tilde \mu_2(\Phi)=0$  
(cf. \cite{Do-3}), where 
$$
\tilde \mu_2(\Phi)=\sqrt{-1}\int_M \left(\Phi^{\ast}\Phi-\frac{\text{\rm trace}\,\Phi^{\ast}\Phi}{\text{\rm dim}\,W} \right)\frac{\omega^n}{n!}. 
$$
\end{rem}

If the target is a complex quadric, then the space of holomorphic sections must be regarded as 
a {\it real} vector space with an orientation. 
Though the universal quotient bundle $Q \to Gr_n(\mathbf R^{n+2})$ has a holomorphic bundle structure 
induced from the canonical connection, we regard $Q \to Gr_n(\mathbf R^{n+2})$ as 
a real vector bundle with the complex structure. 
This is the main difference from the case in which the target is a complex Grassmannian. 

\begin{thm}\label{HGenDWI}
Let $(M,g,J)$ be a compact K\"ahler manifold. 
Fix a holomorphic line bundle $L\to M$ 
with Einstein-Hermitian metric $h_L$ and Einstein-Hermitian connection $\nabla$ 
with $K_{EH}=\mu Id_L$. 
We regard $L\to M$ as a real vector bundle with complex structure $J_L$. 

Let 
$f:M \to Gr_n(\mathbf R^{n+2})$ be a full holomorphic map 
satisfying the gauge condition: 

\noindent{\rm (i)} 
The pull-back bundle 
$f^{\ast}Q \to M$ 
with the pull-back metric, connection and complex structure 
is gauge equivalent to $L\to M$ with $h_L$, $\nabla$ and $J_L$. 

Then we have the space of holomorphic sections $W$ of $L \to M$ which is 
also an eigenspace of the Laplacian with eigenvalue 
$\mu$ equipped with $L^2$-inner product $(\cdot,\cdot)_W$ 
induced from $L^2$-Hermitian inner product. 
Regard $W$ as a real vector space with $(\cdot,\cdot)_W$. 
Then, there exists a semi-positive symmetric endomorphism 
$T\in \text{\rm End}\,(W)$ 
such that the pair $(W,T)$ satisfies the following four conditions:

\noindent {\rm (I)} The vector space $\mathbf R^{n+2}$ is a subspace of $W$ with the inclusion 
$\iota:\mathbf R^{n+2} \to W$ preserving the orientation and 
$L\to M$ is globally generated by $\mathbf R^{n+2}$.

\noindent {\rm (II)} 
As a subspace, $\mathbf R^{n+2}=\text{\rm Ker}\,T^{\bot}$ and 
the restriction of $T$  is a positive symmetric transformation of $\mathbf R^{n+2}$. 

\noindent {\rm (III)} 
The endomorphism $T$ satisfies 
\begin{equation}\label{HDW 4}
ev \circ T^2 \circ ev^{\ast}=Id_L, 
\,\, 
ev \circ T^2 \circ \left(\nabla ev^{\ast}\right)=0 \in \Omega^1(\text{\rm End}\,L), 
\end{equation}

\noindent {\rm (IV)} 
The endomorphism $T$ provides a holomorphic embedding of $Gr_n(\mathbf R^{n+2})$ into $Gr_{n^{\prime}}(W)$, 
where $n^{\prime}=n+\text{\rm dim}\,\text{\rm Ker}\,T$
and also provides a bundle isomorphism $\phi:L\to f^{\ast}Q$. 


Then, $f:M \to Gr_{p}(\mathbf R^{n+2})$ 
can be expressed as 
\begin{equation}\label{HDW5} 
f\left([g]\right)=\left(\iota^{\ast}T\iota \right)^{-1}\left(f_0\left([g]\right)\cap \text{\rm Ker}\,T^{\bot}\right), 
\end{equation}
where $\iota^{\ast}$ denotes the adjoint operator of $\iota$ under the induced inner product on 
$\mathbf R^{n+2}$ from $(\cdot,\cdot)_W$ on $W$ and $f_0$ is the standard map by $W$. 
Such two pairs $(f_i, \phi_i)$, $(i=1,2)$ are gauge equivalent if and only if 
$
\iota_1^{\ast}T_1\iota_1=\iota_2^{\ast}T_2\iota_2, 
$
where $(T_i,\iota_i)$ correspond to $f_i$ $(i=1,2)$ under the expression in \eqref{HDW5}, respectively. 

Conversely, 
suppose that a vector space $\mathbf R^{n+2}$, the space of holomorphic sections 
$W \subset \Gamma(V)$ regarded as real vector space and 
a semi-positive symmetric endomorphism 
$T\in \text{\rm End}\,(W)$ 
satisfying 
conditions {\rm (I)}, {\rm (II)} 
and {\rm (III)} are given.  
Then we have a unique holomorphic embedding of $Gr_n(\mathbf R^{n+2})$ into $Gr_{n^{\prime}}(W)$ 
and 
the map $f:M \to Gr_{p}(\mathbf K^m)$ defined by \eqref{HDW5}
is a full holomorphic map into $Gr_n(\mathbf R^{n+2})$ 
satisfying the gauge condition with bundle isomorphism $L\cong f^{\ast}Q$. 
\end{thm}
\begin{proof}
Under the gauge condition, it follows from Proposition \ref{EH} that the mean curvature operator satisfies the EH condition. 
The remaining assertion is proved in the same way as in the proof of Theorem \ref{GGenDW}. 
\end{proof}

\subsection{Homogeneous cases}
Let $M=G/K_0$ be a compact reductive Riemannian homogeneous space 
with Lie algebra decomposition $\mathfrak{g}=\mathfrak{k}\oplus \mathfrak{m}$, 
where $G$ is a compact Lie group 
and $K_0$ is a closed subgroup 
of $G$. 

Let $V_0$ be a $q$-dimensional real or complex 
$K_0$-representation space 
with a $K_0$-invariant scalar product. 
We can construct a homogeneous vector bundle $V \to M$, $V:=G\times_{K_0} V_0$ 
with an invariant fibre metric $g_V$ 
induced by the scalar product on $V_0$. 
Moreover $V\to M$ has a canonical connection $\nabla$ 
with respect to the decomposition 
$\mathfrak{g}=\mathfrak{k}\oplus \mathfrak{m}$. 
(This means that the horizontal distribution is defined as 
$\{L_g\mathfrak{m}\subset TG_g \,|\,g\in G\}$ on the principal fibre bundle 
$\pi:G \to M$, where $L_g$ denotes the left translation on $G$.) 

The Lie group $G$ naturally acts on the space of sections 
$\Gamma(V)$ of $V\to M$, which 
has a $G$-invariant $L^2$-scalar product. 

Using the Levi-Civita connection and $\nabla$, 
we can decompose the space of sections of $V \to M$ into 
the eigenspaces of the Laplacian:
$$
\Gamma(V)=\oplus_{\mu} W_{\mu}, \quad 
W_{\mu}:=\left\{t \in \Gamma(V)\,|\,\Delta t=\mu t \right\}.
$$
It is well-known that $W_{\mu}$ is a finite-dimensional 
$G$-representation space equipped with a $G$-invariant scalar product induced by 
the $L^2$-scalar product. 

\begin{lemma}\label{subsp}
Let $W$ be a $G$-subspace of the eigenspace $W_{\mu}$. 
If $W$ globally generates $V\to G/K_0$, 
then $V_0$ can be regarded as a subspace of $W$. 
\end{lemma}
\begin{proof}
We identify $V_0$ with the fiber $V_{[e]}$ of $V\to M$ at $[e]\in M$, 
where $e$ is a unit element of $G$. 
Since the evaluation homomorphism $ev:\underline{W} \to V$ is $G$-equivariant and 
the scalar product and the fiber metric are $G$-invariant, 
the adjoint homomorphism $ev^{\ast}:V \to \underline{W}$ is also $G$-equivariant. 
Thus the image of $ev^{\ast}_{[e]}$ is a subspace of $W$ equivalent to $V_0$ as 
$K_0$-representation, because $W$ globally generates $V \to M$. 
\end{proof}

\subsubsection{Standard maps}
Suppose that an eigenspace $W_{\mu}$ globally generates $V\to M$. 
Then we have the standard map 
$f_0:M\to Gr_p(W_{\mu})$ 
by $W_{\mu}$, where $p=N-q$, $N=\text{\rm dim}\,W_{\mu}$, 
$$
f_0([g])=\left\{t \in W_{\mu} \,|\,t([g])=0\right\}.
$$

In general, 
$W_{\mu}$ is not irreducible as 
$G$-representation. 
Let $W$ be a $G$-subspace of $W_{\mu}$ 
and suppose that $W$ globally generates $V \to M$. 
Then the induced map by $W$ is also called 
the standard map by $W$. 

Since $V_0 \subset W$ by Lemma \ref{subsp}, we have the orthogonal complement of $V_0$ 
denoted by $U_0$. 
Then the standard map $f_0:M\to Gr_p(W)$ is expressed as 
$$
f_0([g])=gU_0 \subset W, 
$$
which is $G$-equivariant. 
When we regard $f^{\ast}Q\to M$ as a subbundle of $\underline{W}\to M$, 
the natural identification $\phi:V\to f^{\ast}Q$ is given by 
\begin{equation}\label{stnatid}
\phi\left([g, v]\right)=\left([g], gv\right), 
\end{equation}
where $g\in G$ and $v\in V_0\subset W$.

Suppose that we have a standard map by $W$. 
Next, consider the pull-back connection $\nabla^V$ and the gauge condition. 
Then we have 
\begin{lemma}\label{pullcan}
Let $f_0:M\to Gr(W)$ be the standard map by $W$ wihch is a $G$-subspace of 
the eigenspace $W_{\mu}$. 
Then the pull-back connection $\nabla^V$ 
is gauge equivalent to the canonical connection if and only if 
$\mathfrak{m}V_0 \subset U_0$.
\end{lemma}
\begin{proof}
We use sections $t[g]=\left[g,\pi_0(g^{-1}w)\right]\in \Gamma(V)$ corresponding to $w\in W$, 
where $\pi_0:W \to V_0$ denotes the orthogonal projection. 
The canonical connection $\nabla^0$ is computed as follows:
$$
\nabla^0_{d\pi L_g{\xi}}t=\left[g, -\pi_0\left(\xi g^{-1}w\right)\right].
$$
Next, we use the natural identification to compute the pull-back connection:
$$
\nabla^V_{d\pi L_g{\xi}}t=\pi_V d_{d\pi L_g{\xi}}\phi(t)=
\left[g, \pi_0\left(\xi \pi_0(g^{-1}w)\right)-\pi_0\left(\xi g^{-1} w\right)\right].
$$
Since $W$ globally generates $V\to M$, the result follows. 
\end{proof}
\begin{lemma}\label{stharm}
Let $f_0:M\to Gr(W)$ be the standard map by $W$. 
We use Lemma \ref{subsp} to regard $V_0$ as a subspace of $W$. 
If 
$\mathfrak{m}V_0$ is orthogonal to $V_0$, 
then 
the standard map $f_0:M\to Gr_p(W)$ is harmonic and we have 
$$
e(f_0)=q\mu, \quad 
A=-\mu Id_V.
$$
\end{lemma}
\begin{proof}
Since the pull-back connection is the canonical connection, $W$ is also a subspace of the eigenspace of the Laplacian 
induced by the pull-back connection. 
Then we apply Theorem \ref{harmGrirr} to get the result. 
\end{proof}
\begin{exam}
Let $\mathbf CP^1=\text{SU}(2)/\text{U}(1)$ 
be a complex projective line and 
$\mathcal O(1)\to \mathbf CP^1$ a holomorphic line bundle 
of degree $1$ with the canonical connection. 
Frobenius reciprocity 
yields that the symmetric power $S^{2n+1}\mathbf C^2$ 
of the standard representation $\mathbf C^2$ 
($n\in \mathbf Z_{\geqq 0}$) 
is an $\text{SU}(2)$-invariant space of sections 
of $\mathcal O(1)\to \mathbf CP^1$, 
where $(\varrho_{2n+1},S^{2n+1}\mathbf C^2)$ 
is an irreducible representation of $\text{SU}(2)$. 
Moreover, $S^{2n+1}\mathbf C^2$ 
is an eigenspace of the Laplacian (see \cite{Wall}). 
We denote by $\mathbf C_k$ ($k \in \mathbf Z$) 
an irreducible $\text{U}(1)$-module with 
weight $k$. 
As homogeneous vector bundle, 
$\mathcal O(1)\to \mathbf CP^1$ 
is regarded as $\text{SU}(2)\times_{\text{U}(1)} \mathbf C_{-1}$. 
We can regard $\mathbf C_{-1}$ as a weight subspace of $S^{2n+1}\mathbf C^2$ 
and this provides us with the standard evaluation 
$\underline{S^{2n+1}\mathbf C^2} \to \mathcal O(1)$. 
It follows that 
$$
\varrho_{2n+1}(\mathfrak{m})\mathbf C_{-1} \subset \mathbf C_{-3} \oplus \mathbf C_{1},
$$
because the complexification of 
$\mathfrak{m}$ is identified with 
$\mathbf C_2 \oplus \mathbf C_{-2}$. 
Consequently, the standard map $f_0:\mathbf CP^1 \to 
\mathbf CP^{2n}=\mathbf P(S^{2n+1}\mathbf C^2)$ 
is a harmonic map from Lemma \ref{stharm}. 

See also \cite{Ohn} about an equivariant harmonic map into a complex projective space. 
\end{exam}
\begin{exam}
Let $M=\mathbf HP^1=\text{Sp}(2)/\text{Sp}(1)\times \text{Sp}(1)$ 
be a quaternion projective line. 
To destinguish two copies of $\text{Sp}(1)$ in the isotoropy subgroup, 
we write the isotropy subgroup as $\text{Sp}_{+}(1)\times \text{Sp}_{-}(1)$. 
Let $\mathbf H$ be the standard representation of 
$\text{Sp}_{+}(1)$ and 
$\mathbf E$ be the standard representation of 
$\text{Sp}_{-}(1)$. 
Then the associated homogeneous vector bundles are denoted by the same 
symbols $\mathbf H \to M$ and $\mathbf E\to M$, respectively. 
We suppose that 
$\mathbf H \to M$ is the tautological vector bundle and 
$\mathbf E\to M$ is the orthogonal complement in 
a trivial bundle 
$\underline{\mathbf H^2}\cong \underline{\mathbf C^4}\to M$. 

We take the symmetric power $\mathbf S^k\mathbf H\to M$ of 
$\mathbf H \to M$ and 
$\mathbf S^l\mathbf E\to M$ of 
$\mathbf E \to M$. 
When $k$ (resp.$l$) is even, $S^k\mathbf H$ (resp.$S^l\mathbf E$) 
has a real structure. 
If the both of $k$ and $l$ are odd, then 
$S^k\mathbf H \otimes S^l\mathbf E$ has a real structure. 
In those cases, for example, $S^k\mathbf H$ is supposed to 
represent a real representation or the associated real vector bundle. 

Since the Lie algebra $\mathfrak{sp}(2)$ has 
the standard decomposition as the symmetric pair $\left(\text{Sp}(2),\text{Sp}(1)\times \text{Sp}(1)\right)$:
$$
\mathfrak{sp}(2)=S^2\mathbf H \oplus S^2\mathbf E \oplus 
(\mathbf H \otimes \mathbf E),  
$$
$\mathfrak{sp}(2)$ can be regarded as an eigenspace 
of the Laplacian acting on sections of 
$S^2\mathbf H \to M$. 
Then we have 
$$
[\mathbf H \otimes \mathbf E, S^2\mathbf H] \subset \mathbf H \otimes \mathbf E,
$$
because $\left(\text{Sp}(2),\text{Sp}(1)\times \text{Sp}(1)\right)$ 
is a symmetric pair. 
Lemma \ref{stharm} implies that 
the standard map $f_0:\mathbf HP^1 \to Gr_7(\mathfrak{sp}(2))
=Gr_7(\mathbf R^{10})$ is a harmonic map. 

Now the standard map has another interpretation 
(see also Swann \cite{Swan} and Gambioli \cite{Gam}). 
Let $\mu:\mathbf HP^1 \to \mathfrak{sp}(2)^{\ast}\otimes S^2\mathbf H$ 
be a quaternion moment map \cite{G-L}. 
By definitin of a moment map, for an arbitrary $X \in \mathfrak{sp}(2)$, 
we have 
$$
\mu_X([g])=\left[g, \pi_{S^2\mathbf H}(g^{-1}Xg) \right], \quad g\in \text{Sp}(2),
$$
where $\pi_{S^2\mathbf H}:\mathfrak{sp}(2) \to S^2\mathbf H$ 
is the orthgonal projection. 
It follows that $\mathfrak{sp}(2)$ is a subspace of sections of 
$S^2\mathbf H\to M$ by the moment map $\mu$. 
It is clear that $\mathfrak{sp}(2)$ globally generates $S^2\mathbf H\to M$. 
We can define the induced map $f_{\mu}:\mathbf HP^1 \to Gr_7(\mathbf R^{10})$. 
By definition of the induced map, we have
\begin{align*}
f_{\mu}([g])=&\left\{X\in 
\mathfrak{sp}(2) \,|\,
\text{Ad}(g^{-1})X \in 
S^2\mathbf E \oplus (\mathbf H \otimes \mathbf E)
\right\} \\
=&\text{Ad}(g)\left(S^2\mathbf E \oplus (\mathbf H \otimes \mathbf E) \right)
\subset \mathfrak{sp}(2),
\end{align*}
which is the same as the standard map $f_0$. 

The standard map induced by $S^2\mathbf H \to \mathbf HP^1$ and $\mathfrak{sp}(2)$ 
can be generalized on any compact quaternion symmetric space. 
It is induced by a quaternion moment map for an isometry group 
in the same way. 
\end{exam}
\begin{exam}
Let $G/K$ be a compact irreducible Hermitian symmetric space 
and consider a moment map $\mu:G/K \to \mathfrak{g}^{\ast}$. 
In this situation, $\mu_X:G/K \to \mathbf R$ for an arbitrary 
$X \in \mathfrak{g}$ is an eigenfunction of the Laplacian. 
Then the theorem of Takahashi \cite{TTaka} yields that the induced map 
$f:G/K \to S\subset \mathfrak{g}$ is a harmonic map, 
where $S$ is a hypersphere of $\mathfrak{g}$ 
(Takeuchi-Kobayashi \cite{Kob-Tak}). 
\end{exam}

\subsubsection{A Generalization of dC-W Theory in homogeneous cases}
Let $G$ be a compact Lie group and $W$ be a real or complex representation 
of $G$ with an invariant scalar product $(,)_W$. 
%
Then $G$ naturally acts on $\text{H}(W)$. 
If we equip $\text{H}(W)$ with an inner product $(,)_H$; 
$(B,C)_H:=\text{trace}\,BC$, for $B,C \in \text{H}(W)$, 
then it is easily seen that $(,)_H$ is $G$-invariant. 
We define a symmetric or Hermitian operator $H(u,v)$ for $u$, $v \in W$ as 
$$
H(u,v):=\frac{1}{2}\left\{u\otimes (\cdot, v)_W + v\otimes (\cdot, u)_W \right\}.
$$
Then it follows that for an arbitrary $B \in \text{\rm H}(W)$ 
\begin{equation}\label{BHuv}
\left(B, H(u,v)\right)_H=\frac{1}{2}\left\{(Bu, v)_W+(Bv,u)_W\right\}. 
\end{equation}
If $U$ and $V$ are subspaces of $W$, we define a real subspace  
$H(U,V) \subset \text{H}(W)$ spanned by $H(u,v)$ where $u\in U$ and 
$v\in V$. 
In a similar fashion, $GH(U,V)$ denotes the subspace of $\text{H}(W)$ spanned by 
$gH(u,v)$, where $g\in G$, and so $GH(U,V)$ is a $G$-submodule of $\text{H}(W)$.  


We are now in a position to state a generalization of do-Carmo Wallach thoery in 
homogeneous cases. 
Though the difference from Theorem \ref{GGenDW} is only condition (III), 
we state the theorem in its complete form for readers' convenience. 
\begin{thm}\label{GenDW}
Let $G/K_0$ be a compact reductive Riemannian homogeneous space with decomposition 
$\mathfrak{g}=\mathfrak{k}\oplus \mathfrak{m}$. 
Fix a homogeneous vector bundle 
$V=G\times_{K_0} V_0 \to G/K_0$ 
of rank $q$ with an invariant metric and the canonical connection. 

Let 
$f:G/K_0 \to Gr_p(\mathbf K^m)$ be a full harmonic map 
satisfying the following two conditions. 

\noindent{\rm (i)} 
The pull-back bundle 
$f^{\ast}Q \to M$ 
with the pull-back metric and connection 
is gauge equivalent to $V\to G/K_0$ with the invariant metric and the canonical connection. 
{\rm (}Hence, $q=m-p$.{\rm )}

\noindent{\rm (ii)} 
The mean curvature operator $A \in \Gamma(\text{\rm End}\,V)$ of a map $f$ 
is expressed as 
$-\mu Id_V$ for some positive real number $\mu$, 
and so $e(f)=\mu q$. 

Then there exist an eigenspace $W \subset \Gamma(V)$ of the Laplacian with eigenvalue 
$\mu$  equipped with $L^2$-scalar product $(\cdot,\cdot)_W$ 
and a semi-positive Hermitian endomorphism 
$T\in \text{\rm End}\,(W)$. 
Regard $W$ as $\mathfrak{g}$-representation $(\varrho, W)$. 
The pair $(W,T)$ satisfies the following four conditions.

\noindent {\rm (I)} The vector space $\mathbf K^m$ is a subspace of $W$ with the inclusion 
$\iota:\mathbf K^m \to W$ and 
$V\to G/K_0$ is globally generated by $\mathbf K^m$.

\noindent {\rm (II)} 
As a subspace, $\mathbf K^m=\text{\rm Ker}\,T^{\bot}$, 
and the restriction of $T$ is a positive Hermitian transformation of $\mathbf K^m$. 

\noindent {\rm (III)} 
The endomorphism $T$ satisfies 
\begin{equation}\label{DW 2}
\left(T^2-Id_W, GH(V_0, V_0)\right)_H=0, 
\,\, 
\left(T^2, GH(\varrho(\mathfrak m)V_0, V_0)\right)_H=0,
\end{equation}
where $V_0$ is regarded as a subspace of $W$ by Lemma \ref{subsp}.

\noindent {\rm (IV)} 
The endomorphism $T$ gives an embedding of $Gr_p(\mathbf K^m)$ into $Gr_{p^{\prime}}(W)$, 
where $p^{\prime}=p+\text{\rm dim}\,\text{\rm Ker}\,T$
and also gives a bundle isomorphism $\phi:V\to f^{\ast}Q$. 

Then, $f:G/K_0 \to Gr_{p}(\mathbf K^m)$ 
can be expressed as 
\begin{equation}\label{DW3} 
f\left([g]\right)=\left(\iota^{\ast}T\iota \right)^{-1}\left(f_0\left([g]\right)\cap \text{\rm Ker}\,T^{\bot}\right), 
\end{equation}
where $\iota^{\ast}$ denotes the adjoint operator of $\iota$ under the induced scalar product on 
$\mathbf K^m$ from $(\cdot,\cdot)_W$ on $W$ and $f_0$ is the standard map by $W$. 

The pairs $(f_1, \phi_{1})$ and $(f_2, \phi_{2})$ are gauge equivalent if and only if 
$$
\iota_1^{\ast}T_1\iota_1=\iota_2^{\ast}T_2\iota_2, 
$$
where $(T_i,\iota_i)$ correspond to $f_i$ $(i=1,2)$ under the expression in \eqref{DW3}, respectively. 

Conversely, 
suppose that a vector space $\mathbf K^m$, an eigenspace $W \subset \Gamma(V)$ 
with eigenvalue $\mu$ and 
a semi-positive Hermitian endomorphism 
$T\in \text{\rm End}\,(W)$ 
satisfying 
conditions {\rm (I)}, {\rm (II)} 
and {\rm (III)} are given.  
Then there is a unique embedding of $Gr_p(\mathbf K^m)$ into $Gr_{p^{\prime}}(W)$ 
and 
the map $f:G/K_0 \to Gr_{p}(\mathbf K^m)$ defined in \eqref{DW3}
is a full harmonic map into $Gr_p(\mathbf K^m)$ 
satisfying conditions {\rm (i)} and {\rm (ii)} with bundle isomorphsm $V\cong f^{\ast}Q$. 
\end{thm}

\begin{proof}
We follow the notation in the proof of Thereom \ref{GGenDW} and 
pay attention only to the role of condition (III). 

When $f^{\ast}Q \to M$ is identified with the orthogonal complememt of $f^{\ast}S \to M$, 
it follows from \eqref{GnatidT} and \eqref{stnatid} that 
the natural identification 
$\phi:V \to f^{\ast}Q$ is expressed as 
\begin{equation}\label{natidT}
\phi\left([g, v]\right)=\left([g], Tgv\right), 
\end{equation}
where $g\in G$ and $v\in V_0\subset W$. 

Since the metric on $f^{\ast}Q\to M$ is induced from the scalar product on $W$, 
it follows from condition (i) and \eqref{natidT} that 
\begin{equation}\label{isomVV}
(Tgv,Tgv^{\prime})_W=(v,v^{\prime})_W,
\end{equation}
for arbitrary $v, v^{\prime} \in V_0$. 
From the definition of the scalar product on $\text{\rm H}(W)$, 
we have 
\begin{align}\label{WS}
\text{Re}(Tgv, Tgv^{\prime})_W=&(T^2, gH(v,v^{\prime}))_H \quad \text{and} \\
\text{Im}(Tgv, Tgv^{\prime})_W=&(T^2, gH(v,\sqrt{-1}v^{\prime}))_H. \notag
\end{align}
Together with  
$(Id, H(v,v^{\prime}))_H=\text{Re}(v,v^{\prime})_W$, 
\eqref{isomVV} and \eqref{WS} yield that 
\begin{equation}\label{III1}
(T^2-Id, gH(v,v^{\prime}))_H=0
\end{equation}
for an arbitrary $g\in G$ and arbitrary $v$, $v^{\prime}\in V_0$, 
which is equivalent to 
\begin{equation}\label{Bmet}
(T^2-Id, GH(V_0,V_0))_H=0.
\end{equation}
Since the equation \eqref{isomVV} is equivalent to 
$$
(g^{-1}T^2gv, v^{\prime})_W=(v,v^{\prime})_W, 
$$
we obtain 
\begin{eqnarray}\label{BV}
\pi_0(g^{-1}T^2g)i_0=Id_{V_0},
\end{eqnarray}
where $i_0:V_0 \to W$ is the natural inclusion and 
$\pi_0:W \to V_0$ is the orthogonal projection to $V_0$. 


Next, we compare the canonical connection $\nabla^0$ with the induced connection $\nabla$ 
on $V \to M$ 
by $f:M\to Gr_p(\mathbf K^m)$. 
To describe the orthogonal projection $\pi_V:\underline{W}\to V$, notice that 
$\pi_V$ is recognized as the adjoint of $\phi$. 
Thus $\pi_V:\underline{W}\to V$
is expressed as:
$$
\pi_V\left([g],w \right)
=\left[g,\pi_0(g^{-1}Tw)\right],
$$
If we use a section $t[g]=\left[g,\pi_0(g^{-1}w)\right]$ corresponding to $w\in W$,  
then the canonical connection is calculated as follows:
\begin{equation}\label{cancon}
\nabla^0_{d\pi L_g\xi} t=\left[g, -\pi_0\left(\xi g^{-1}w\right)\right].
\end{equation}
Next the induced connection is calculated as follows:
\begin{align*}
\nabla_{d\pi L_g\xi} t=&\pi_V d_{d\pi L_g\xi} \phi(t)
=\pi_V d_{d\pi L_g\xi}\left([g],{T}g\pi_0(g^{-1}w)\right) \\
=&\pi_V\left([g], {T}g\xi \pi_0(g^{-1}w)
-{T}g \pi_0(\xi g^{-1}w)\right) \\
=&\left[g, \pi_0\left(g^{-1}T^2g\xi \pi_0(g^{-1}w)\right)
-\pi_0\left(g^{-1}T^2g\pi_0 (\xi g^{-1}w)\right)\right]. 
\end{align*}
It follows from \eqref{BV} and \eqref{cancon} that 
$$
\nabla_{d\pi L_g\xi} t-\nabla^0_{d\pi L_g\xi} t
=\left[g, \pi_0\left(g^{-1}T^2g\xi \pi_0(g^{-1}w)\right)\right]. 
$$ 
Since $W$ globally generates $V\to M$, condition (i) yields that 
\begin{eqnarray}\label{GE}
\pi_0(g^{-1}T^2g\xi)i_0=0, 
\end{eqnarray}
for an arbitrary $g\in G$ and $\xi \in \mathfrak{m}$, 
which is equivalent to 
\begin{equation}\label{III2}
(T^2, GH(\mathfrak{m}V_0, V_0))_H=0. 
\end{equation}
It follows from \eqref{III1} and \eqref{III2} that condition (III) holds. 

\end{proof}

In the homogeneous case, we also examine the image equivalence, 
mainly, to fix the notation. 
Since $K_0$ is the holonomy group and the structure group of $V \to G/K_0$ is 
(a subgroup of) $K_0$, 
we can find $c$ an element of the center of the structure group such that 
$$
\phi_2^{-1}\tilde \psi \phi_1(v)=\left[g,\varrho_0(c)v\right], 
$$
where $\varrho_0$ denotes the representation of $K_0$ on $V_0$. 
Then \eqref{natidT} yields that 
\begin{equation}\label{imeq}
\tilde \psi T_1 gv = T_2 g \varrho_0(c)v=T_2 gcv. 
\end{equation}
Hence we must take the action of the center of $K_0$ 
into account. 


\section{Joyce inequalities}

Joyce obtains estimates of the dimension of two special eigenspaces of the Laplacian 
on any compact minimal Lagrangian submanifold of 
$\mathbf CP^{N}$ \cite[Proposition 3.5]{Joy}. 

We show that such estimates can also be obtained using Thereom \ref{harmGrirr}. 
Proposition 3.5 in \cite{Joy} has two statements. 
We divide it to two parts, because the latter estimate will be further generalized 
to the case of any minimal Lagrangian submanifold of compact irreducible 
Hermitian symmetric spaces. 

\begin{thm}\label{J1}\cite{Joy}
Let $f:\Sigma \to \mathbf Gr_{n}(\mathbf C^{n+1})$ be 
a full minimal Lagrangian submanifold of $Gr_{n}(\mathbf C^{n+1})$. 
If $\Sigma$ is simply connected, then 
$$
m_{\Sigma}(n) \geqq 2(n+1), 
$$
where $m_{\Sigma}(n)$ is the dimension of the eigenspace with eigenvalue $n$ on $\Sigma$. 
\end{thm}

\begin{proof}
Since $f:\Sigma \to \mathbf Gr_{n}(\mathbf C^{n+1})$ is a minimal immersion, 
Theorem \ref{genTak} yields that 
$$
\Delta t=-At, \quad \text{for any}\, t \in \mathbf C^{n+1} \subset \Gamma\left(f^{\ast}\mathcal O(1)\right). 
$$
Since $\mathcal O(1) \to \mathbf CP^{N-1}$ is of rank $1$, 
we can regard $A \in \Gamma\left( \text{End}\,\mathcal O(1)\right)$ as a function. 
It follows from $f$ being an isometric immersion that 
$$
-\text{trace}\,A=|df|^2=n, 
$$
and so, 
$$
\Delta t=nt, \quad \text{for any}\, t \in \mathbf C^{n+1} \subset \Gamma\left(f^{\ast}\mathcal O(1)\right). 
$$

Since $f$ is a Lagrangian immersion, the curvature of the pull-back connection on 
$f^{\ast}\mathcal O(1)\to \Sigma$ vanishes. 
Thus $f^{\ast}\mathcal O(1)\to \Sigma$ is a trivial bundle, 
because $\Sigma$ is simply connected. 
Hence, $t$ is recognized as a function on $\Sigma$. 
The result follows from $f$ being a full map. 
\end{proof}

Let $(G,K_0)$ be an irreducible Hermitian symmetric pair of compact type 
and $\mathfrak{g}=\mathfrak{k} \oplus \mathfrak {m}$ the standard decomposition of Lie algebra of $G$. 
We adopt $G$-invariant inner product $(\cdot, \cdot)$ on $\mathfrak{g}$ as the minus of the Killing form, 
and induce a $K_0$-invariant inner product on $\mathfrak{m}$ 
denoted by the same symbol $(\cdot, \cdot)$. 

Since $(G,K_0)$ be an irreducible Hermitian symmetric pair, 
the Lie algebra $\mathfrak{k}$ has a one-dimensional center $\mathfrak {u}(1)$. 
Notice that there exists a unique $Z \in \mathfrak {u}(1)$ corresponding to the complex structure $J$ on $\mathfrak {m}$. 
Since $G/K_0$ is K\"ahler, we obtain 

\begin{lemma}\label{metZ}
For any $\xi, \eta \in \mathfrak {m}$, 
$$
\left([Z,\xi], [Z,\eta] \right)=\left(\xi,\eta \right). 
$$
\end{lemma}

Let $Q:=G\times_{K_0} \mathfrak{k}$ be a homogeneous vector bundle over $G/K_0$. 
Let $i$ be a standard map by $(Q \to G/K_0, \mathfrak{g})$. 
If $n$ denotes the complex dimension of $G/K_0$, then $i$ turns out to be a 
{\it totally geodesic immersion} of $G/K_0$ into $Gr_{2n}(\mathfrak {g})$.  
(This theory will be expanded in the forthcoming paper.) 
It follows from the general theory that the second fundamental forms satisfy 
\begin{equation}\label{HKG}
K_{\xi}X=[\xi,X], \quad H_{\xi} \eta =[\xi,\eta], 
\end{equation}
for any $\xi,\eta \in \mathfrak{m}$ and $X \in \mathfrak{k}$. 
Here we consider the second fundamental forms at $[e]\in G/K_0$, where $e$ is the unit element of $G$. 

\begin{lemma}\label{JHK}
Let $J$ be the complex structure of $G/K_0$. 
For any $\xi,\eta \in \mathfrak{m}$ and $X \in \mathfrak{k}$, 
$J$ and the second fundamental forms $H$ and $K$ satisfy 
$$
K_{J\xi}X = JK_{\xi}X, \quad H_{J\xi} \eta =-H_{\xi}J\eta. 
$$
\end{lemma}
\begin{proof}
Since $\mathfrak{u}(1)$ is the center of $\mathfrak{k}$ and $(G,K_0)$ is a symmetric pair, 
the Jacobi identity and \eqref{HKG} yield that 
$$
K_{J\xi}X = \left[[Z,\xi],X \right]=-\left[[\xi,X],Z \right]-\left[[X,Z],\xi \right]=
\left[Z,[\xi,X]\right]=JK_{\xi}X,  
$$
and 
$$
H_{J\xi}\eta = \left[[Z,\xi],\eta \right]=-\left[[\xi,\eta],Z \right]-\left[[\eta ,Z],\xi \right]=
-\left[\xi, [Z,\eta]\right]=-H_{\xi}J\eta.  
$$
\end{proof}

\begin{thm}\label{J2}
Let $(G,K_0)$ be an irreducible Hermitian symmetric pair of compact type 
and $f:\Sigma \to G/K_0$ a minimal Lagrangian immersion. 
Let $H$ be a maximal subgroup of $G$ which preserves $f(\Sigma)$. 
Then $n$ is the eigenvalue of the Laplacian on $\Sigma$ and the dimension $m_{\Sigma}(n)$ 
of the corresponding eigenspace satisfies 
$$
m_{\Sigma}(n)\geqq \text{\rm dim}\,G-\text{\rm dim}\,H.
$$
\end{thm}
\begin{rem}
In Theorem \ref{J2}, we adopt a different metric on $\mathbf CP^n$ from the one in Theorem \ref{J1}. 
When we use the same metric as the one in Theorem \ref{J1}, 
the eigenvalue in Theorem \ref{J2} is different from $n$. 
\end{rem}

\begin{proof}
Let 
$
F:=i\circ f:\Sigma \to G/K_0 \to Gr_{2n}(\mathfrak{g})
$
be the composition of $f$ and $i$. 
Since $i$ is a totally geodesic immersion, 
$F$ is also a minimal immersion. 

Applying Theorem \ref{genTak}, we get 
\begin{equation}\label{mom}
\Delta t=-At, \quad \text{for any}\, t \in \mathfrak{g} \subset \Gamma\left(F^{\ast}Q\right). 
\end{equation}

Using \eqref{HKG}, the mean curvature operator $\tilde A$ of $i$ is expressed as 
$$
\tilde A=\sum_{i=1}^{N-1} H_{e_i}K_{e_i}+H_{Je_i}K_{Je_i}
=\sum_{i=1}^{N-1} \left[e_i,[e_i,\xi] \right] +\left[Je_i,[Je_i,\xi] \right], 
$$
where $e_1, Je_1, e_2, Je_2,\cdots, e_n, Je_{n}$ is an orthonormal basis of $Tx(G/K_0)$. 
Since $i$ is $G$-equivariant, Schur's lemma yields that 
$\mathfrak{u}(1)$ is an eigenspace of $\tilde A$. 
We denote by $c_0$ the corresponding eigenvalue. 
To compute the value of $c_0$,  we use $Z$ and Lemma \ref{metZ} to obtain 
\begin{align*}
\left( \tilde A Z, Z \right)  
=&\left( 
\sum_{i=1}^{n} \left[e_i,[e_i,Z] \right] +\left[Je_i,[Je_i,Z] \right], Z \right) \\
=&-
\sum_{i=1}^{n} \left( [e_i,Z], [e_i,Z] \right)
+\left( [Je_i,Z], [Je_i,Z] \right) 
=-2n, 
\end{align*}
and so, 
$c_0=-2n$. 

Let $e_1,\cdots, e_{n}$ be an orthonormal basis of $T_x\Sigma$. 
Then, the mean curvature operator $A$ of $F$ is expressed as 
$$
A=\sum_{i=1}^{n} H_{e_i}K_{e_i}
=\sum_{i=1}^{n} \left[e_i,[e_i,\xi] \right]. 
$$
Notice that $e_1, Je_1, e_2, Je_2,\cdots, e_n, Je_{n}$ is an orthogonal basis of $Tx(G/K_0)$, 
because $\Sigma$ is Lagrangian. 
Thus, Lemma \ref{JHK} yields that 
$$
\tilde A= 2\sum_{i=1}^{n} \left[e_i,[e_i,\xi] \right]
$$
and so, 
\begin{equation}\label{half}
A=\frac{1}{2}\tilde A. 
\end{equation}

Since $G\times_K \mathfrak{u}(1)$ is a subbundle of $G\times_K \mathfrak{k}$ 
and $\mathfrak{u}(1)$ is the eigenspace of $\tilde A$ with eigenvalue $-2n$,  
Theorem \ref{genTak} with \eqref{half} yields that 
\begin{equation}\label{mom2}
\Delta t=n t, \quad \text{for any}\,t \in \mathfrak{g} \subset 
\Gamma\left(F^{\ast}\left(G\times_K \mathfrak{u}(1)\right)\right).
\end{equation}

The vector bundle $G\times_K \mathfrak{u}(1) \to G/K_0$ is a trivial bundle, 
because $K_0$ acts trivially on $\mathfrak{u}(1)$.
Hence $t$ is ragarded as a function. 
Indeed, $t$ is a restriction of moment map on $G/K_0$ for $G$-action. 
(See \cite{D-K}.)
More precisely, if we denote by $t_X$ the function determined by $X \in \mathfrak g$, 
then we have 
$$
dt_X=\omega(X^M, \cdot), 
$$
where $X^M$ denotes the Killing vector field on $G/K_0$ induced by $X$ 
and $\omega$ is the K\"ahler form on $G/K_0$. 

Hence, for $\xi \in \text{Lie}\,H\subset \mathfrak{g}$, 
\begin{equation}\label{moment}
dt_{\xi}=\omega(\xi^M,\cdot)=g(J\xi^M, \cdot), 
\end{equation}
and so, 
$$
dt_{\xi}=0\quad \text{on}\,\Sigma. 
$$
Thus $t_{\xi}$ is (locally) constant. 
Since $t_{\xi}$ satisfies \eqref{mom2}, we have
$t_{\xi}=0$.

Since $F$ is not a full map in general, 
the linear map $\mathfrak{g} \to C^{\infty}(\Sigma)$ might have a kernel. 
If $X$ is in the kernel, then 
$$
dt_{X}=0\quad \text{on}\,\Sigma. 
$$
Therefore, it follows from \eqref{moment} 
that $JX^M$ is orthogonal to $\Sigma$. 
Then $X^M$ is regarded as a vector field on $\Sigma$, 
the maximality of $H$ implies that 
$X\in \mathfrak{h}$.  
\end{proof}

\section{Applications}

First of all, we shall give an algebraic result which is a slight modification of do Carmo-Wallach \cite{DoC-Wal}.

Throughout this section, 
$G/K_0$ denotes a compact reductive Riemannian homogeneous space with decomposition 
$\mathfrak{g}=\mathfrak{k}\oplus \mathfrak{m}$. 
Let $V_0$ be a $K_{0}$-representation 
with an invariant scalar product 
and 
$V=G\times_{K_0} V_0 \to G/K_0$ an  associated homogeneous vector bundle with 
the induced invariant metric. 
Let $W$ be a $G$-submodule of $\Gamma(V)$ 
endowed with a $G$-invariant scalar product induced by $L^2$-scalar product.  
If $W$ globally generates $V \to G/K_0$, then it follows from Lemma \ref{subsp} 
that $V_0$ can be considered as a subspace of $W$. 
We denote by $\pi_0:W \to V_0$ the orthogonal projection. 
Let $U_0$ be the orthogonal complement of $V_0$ in $W$ 
with the orthogonal projection $\pi_1:W \to U_0$. 


We would like to describe the decomposition of $W$ as $K_0$-module. 
We follow an idea of do-Carmo and Wallach \cite{DoC-Wal}. 



\begin{defn}
A linear map $B_1:\mathfrak{m}\otimes V_0 \to U_0$ 
is defined as: 
$$
B_1\left(\xi \otimes v)\right)=\pi_{1}\left(\xi v)\right), \quad \text{for}\quad \xi \in \mathfrak{m}, v\in V_0.
$$
Then we also define 
$$
N_2:=\left( V_0 \oplus \text{\rm Im}\,B_1\right)^{\bot} \subset W,
$$
For simplicity, $U_0$ is also denoted by $N_1$. 
\end{defn}

Since $\pi_1$ is $K$-equivariant, we have 
\begin{lemma}
A linear map 
$B_1:\mathfrak{m}\otimes V_0 \to U_0$
is $K_0$-equivariant. 
\end{lemma}
\begin{cor}
$\text{\rm Im}\,B_1$ is a $K_0$-module. 
\end{cor}
The $n$-th symmetric power of $\mathfrak{m}$ is denoted by $S^n\mathfrak{m}$. 
Let us denote by $S_n$ the permutation group of order $n$, and define an element of $S^n\mathfrak{m}$ as 
$$
\xi_1 \cdots \xi_n:=\frac{1}{n!}\sum_{\sigma \in S_n}\xi_{\sigma(1)}\otimes \cdots \otimes \xi_{\sigma(n)},
$$
for $\xi_1, \cdots, \xi_n \in \mathfrak{m}$. 
\begin{defn}
Inductively, the subspace $\left(V_0 \bigoplus \oplus_{p=1}^{n-1}\text{\rm Im}\,B_p\right)^{\bot}$ of $W$ 
is denoted by $N_{n}$ with the orthogonal projection $\pi_n:W \to N_n$. 
Then the linear map 
$B_n:S^n\mathfrak{m}\otimes V_0 \to N_{n}$ 
is defined as 
$$
B_n(\xi_1\cdots  \xi_n\otimes v)
=\pi_{n}\left((\xi_1 \cdots \xi_n) v\right), 
$$
where 
$(\xi_1 \cdots \xi_n) v=\frac{1}{n!}\sum_{\sigma \in S_n}\xi_{\sigma(1)} \cdots \xi_{\sigma(n)}v$. 
\end{defn}

Since $B_n:S^n\mathfrak{m}\otimes V_0 \to N_{n}$ 
is $K_0$-equivariant, $\text{\rm Im}\,B_n$ is a $K_0$-module. 

\begin{lemma}\label{symB}
If $(G,K_0)$ is a symmetric pair, then 
$$
B_n(\xi_1\cdots  \xi_n\otimes v)
=\pi_{n}\left(\xi_1\left(\xi_2\left( \cdots \left(\xi_n v\right)\right)\right)\right).
$$
\end{lemma}
\begin{proof}
For $\xi_1,\cdots,  \xi_n \in \mathfrak{m}$ and $v\in V_0\subset W$, 
we have 
$$
\xi_1 \xi_2 \xi_3 \cdots \xi_n v-\xi_2 \xi_1 \xi_3 \cdots \xi_n v=
[\xi_1,\xi_2]\xi_3 \cdots \xi_n v. 
$$
By definition, it follows that $\xi_3 \cdots \xi_n v \in V_0\oplus \text{Im}\,B_1 \oplus \cdots \text{Im}\,B_{n-2}$. 
From the hypothesis that $(G,K_0)$ is a symmetric pair, we get $[\xi_1, \xi_2]\in \mathfrak{k}$. 
Since $V_0\oplus \text{Im}\,B_1 \oplus \cdots \text{Im}\,B_{n-2}$ is a $K_0$-module, 
we have
$$
[\xi_1,\xi_2]\xi_3 \cdots \xi_n v \in V_0\oplus \text{Im}\,B_1 \oplus \cdots \text{Im}\,B_{n-2},
$$
and so, 
$$
\pi_n\left(\xi_1 \xi_2 \xi_3 \cdots \xi_n v\right)=\pi_n\left(\xi_2 \xi_1 \xi_3 \cdots \xi_n v\right).
$$
In a similar way, we obtain 
$$
\pi_n\left(\xi_1 \cdots \xi_i \xi_{i+1} \cdots \xi_n v\right)=\pi_n\left(\xi_1  \cdots \xi_{i+1} \xi_{i} \cdots \xi_n v\right), 
$$
thus the result follows. 
\end{proof}



\begin{defn}
Let $W$ be a $G$-module with an invariant scalar product and 
$V_0 \subset W$ a $K_0$-module. 
If we can decompose $W$ as 
\begin{equation}\label{nordec}
W=V_0 \oplus \text{\rm Im}\,B_1 \oplus \cdots \oplus \text{\rm Im}\,B_n,  
\end{equation}
then $(W,V_0)$ is said to have a {\it normal} decomposition (into $K_0$-modules). 
We sometimes denote $V_0$ by $\text{\rm Im}\,B_0$ in the normal decomposition. 
\end{defn}
\begin{prop}\label{Bdeco}
If $W$ is an {\it irreducible} $G$-module, 
then for any $K_0$-module $V_0 \subset W$ 
there exists a positive integer $n$ such that 
$$
W=V_0 \oplus \text{\rm Im}\,B_1 \oplus \cdots \oplus \text{\rm Im}\,B_n,  
$$
which is a normal decomposition of $(W,V_0)$. 
\end{prop}
\begin{proof}
Since $W$ is irreducible, 
the result follows. 
\end{proof}

\begin{prop}\label{TI}{\rm (see Lemma 4.2 in \cite{DoC-Wal})}
\, 
Let $W$ be a $G$-module and $V_0$ a $K_0$-module regarded as a subspace of $W$.  
Suppose that $(W,V_0)$ has a normal decomposition \eqref{nordec}. 
Assume that $\text{\rm Im}\,B_i$ and $\text{\rm Im}\,B_j$ has no common irreducible submodules 
in the $K_0$-irreducible decomposition,  if $i\not= j=0,\cdots,n$.   
Let $T$ be a non-negative Hermitian endomorphism on $W$ which satisfies 
$(Tgv_1, Tgv_2)=(v_1, v_2)$ for arbitrary $g \in G$ and $v_1$, $v_2 \in V_0$.

If $T$ is $K_0$-equivariant, then $T=Id_{W}$. 
\end{prop}
\begin{proof}
The assumption on $\text{\rm Im}\,B_i$ and Schur's lemma imply that   
$T\text{Im}\,B_i \subset \text{Im}\,B_i$, where $i=0,\cdots, n$. 

It follows from $(Tgv_1, Tgv_2)=(v_1, v_2)$ that $T|_{V_0}=Id_{V_0}$, 
because $T$ is a non-negative Hermitian operator preserving the induced scalar product on $V_0$. 

From the hypothesis, it follows that for an arbitrary $\xi \in \mathfrak m$ 
$$
(Tge^{t\xi}v_1, Tge^{t\xi}v_2)=(v_1, v_2),  
$$
where $t \in \mathbf R$ and $v_1$, $v_2 \in V_0$. 
Then we get
$$
0=\frac{d^{2p}}{dt^{2p}}\Big |_{t=0}(Te^{t\xi}v_1, Te^{t\xi}\xi v_2)
=\sum_{r=0}^{2p} \binom{2p}{r}(T\xi^r v_1, T\xi^{2p-r}v_2), 
$$
and so, 
\begin{align*}
&\binom{2p}{p}(T\xi^{p} v_1,T\xi^{p}v_2)\\
=&-\sum_{r=0}^{p-1} \binom{2p}{r}(T\xi^r v_1, T\xi^{2p-r}v_2)
-\sum_{r=p+1}^{2p} \binom{2p}{r}(T\xi^{r} v_1, T\xi^{2p-r}v_2) \\
=&-\sum_{r=0}^{p-1} \binom{2p}{r}(T\xi^r v_1, T\xi^{2p-r}v_2)
-\sum_{r=0}^{p-1} \binom{2p}{r}(T\xi^{2p-r} v_1, T\xi^{r}v_2).
\end{align*}\\

Suppose that $T$ is the identity on 
$V_0\oplus \text{\rm Im}\,B_1 \oplus \cdots \oplus \text{\rm Im}\,B_{p-1}$. 
From the inductive hypothesis and condition that $T$ is a Hermitian operator, 
if $r < p$, then 
\begin{equation*}
(T\xi^r v_1, T\xi^{2p-r}v_2)=(\xi^r v_1, T\xi^{2p-r}v_2)=
(T\xi^r v_1, \xi^{2p-r}v_2)
=(\xi^r v_1, \xi^{2p-r}v_2).
\end{equation*}
Consequently, it follows that 
\begin{align*}
&\binom{2p}{p}(T\xi^{p} v_1,T\xi^{p}v_2)\\
=&-\sum_{r=0}^{p-1} \binom{2p}{r}(\xi^r v_1, \xi^{2p-r}v_2)
-\sum_{r=p+1}^{2p} \binom{2p}{r}(\xi^{r} v_1, \xi^{2p-r}v_2)\\
=&\binom{2p}{p}(\xi^{p} v_1,\xi^{p}v_2), 
\end{align*}\\
and so, 
$(T\xi^{p} v_1,T\xi^{p}v_2)=(\xi^{p} v_1,\xi^{p}v_2)$.

Hence $T|_{\text{\rm Im}\,B_p}$ is a non-negative Hermitian operator preserving 
the scalar product. 
Thus $T|_{\text{\rm Im}\,B_p}=Id_{\text{\rm Im}\,B_p}$. 
\end{proof}



Let $\tilde W$ be a $G$-module and $\tilde V_0 \subset \tilde W$ a $K_0$-module. 
The $K_0$-module $H(\tilde V_0, \tilde V_0)$ is decomposed into $K_0$-irreducible modules:
$$
H(\tilde V_0, \tilde V_0)=\oplus_{i=1}^l H_i, 
$$
where $H_0$ denotes the one-dimensional trivial representation corresponding to 
the identity transformation of $\tilde V_0$. 
Let $GH_0$ be the subspace of $\text{H}(\tilde W)$ generated by $G$ and $H_0$. 

\begin{prop}\label{clon}{\rm (see Lemma 4.4 in \cite{DoC-Wal})}
\, 
Let $W$ be an irreducible $G$-module and $V_0 \subset W$ a $K_0$-module. 
Assume that $\text{\rm Im}\,B_i$ and $\text{\rm Im}\,B_j$ has no common $K_0$-irreducible submodules,  
if $i\not= j$ 
in the normal decomposition of $(W,V_0)$. 

Let $\tilde W$ be a direct sum of $N$-copies of $W$ and $\tilde V_0$ 
a direct sum of $N$-copies of $V_0$ which is regarded as a subspace of $\tilde W$ 
in a natural way. 

Then $GH_0$ consists of class one submodules of $(G,K_0)$. 
An arbitrary class one representations of $(G,K_0)$ in $\text{\rm H}(\tilde W)$ 
is a submodule of $GH(\tilde V_0,\tilde V_0)$.  
\end{prop}
\begin{proof}
First of all, we decompose $GH_0$ into $G$-irreducible modules: 
$$
GH_0=\oplus_p W_p.  
$$
The obvious orthogonal projection is denoted by 
$\pi_p:GH_0 \to W_p$, for each $p$, 
which is a $G$-equivariant map. 
Then $\pi_p(H_0)\not=\{0\}$ for an arbitrary $p$, by definition of $GH_0$ and so, 
$W_p$ is a class one representation because 
$\pi_p:GH_0 \to W_p$ is also $K_0$-equivariant. 

Next, 
suppose that $H$ is a class one subrepresentation of $(G,K_0)$ in $\text{\rm H}(\tilde W)$ such 
that $H\not\subset GH(\tilde V_0, \tilde V_0)$. 
Then, by standard arguments, we can assume that 
$H \bot GH(\tilde V_0, \tilde V_0)$ without loss of generality. 

Since $H$ is a class one representation, there exists a non-zero $C \in H$ such that 
$kCk^{-1}=C$ for any $k \in K$. 
It follows from $H \bot GH(\tilde V_0, \tilde V_0)$ that 
\begin{equation*}
0=(C, gH(v_1, v_2))_H=(C,H(gv_1, gv_2))_H
=\frac{1}{2}\left\{(Cgv_1, gv_2)+ (Cgv_2, gv_1)\right\},
\end{equation*}
for arbitrary $g \in G$ and $v_1$, $v_2 \in V_0$. 
Thus we get 
$$
0=(Cgv_1, gv_2), \quad g \in G, \,\, v_1, v_2 \in V_0. 
$$
If $C$ is sufficiently small, then $I+C>0$ and so, we can define 
a positive Hermitian operator $T$ satisfying $T^2=I+C$. 
Then we have 
$$
(Tgv_1, Tgv_2)=(v_1, v_2) \quad g \in G, \,\, v_1, v_2 \in V_0.
$$
Since $C$ is $K_0$-equivariant, $T$ is also $K_0$-equivariant.  
Since $(\tilde W, \tilde V_0)$ has the normal decomposition induced from one of $(W,V_0)$ 
also satisfying $T\text{Im}\,\tilde B_i \subset \text{Im}\,\tilde B_i$ in the obvious sense, 
Lemma \ref{TI} yields that $T=Id$ and so, $C=0$, 
which is a contradiction. 

\end{proof}

%

We introduce two theorems which are proved in independent ways. 
A unified proof can be given in the light of Theorem \ref{GenDW}. 


\begin{thm}\cite{Ban-Ohn}\label{BanOhn}
Let $f:\mathbf CP^1 \to \mathbf CP^n$ 
be a full harmonic map with constant energy density.  
Then $f$ is an $\text{\rm SU}(2)$-equivariant map, in other words, it is a standard map 
up to gauge equivalence.
\end{thm}

\begin{thm}\cite{Cal}\label{Calr}
Let $f:\mathbf CP^m \to \mathbf CP^n$ 
be a full holomorphic map with constant energy density.  
Then $f$ is an $\text{\rm SU}(m+1)$-equivariant map, in other words, it is a standard map 
up to gauge equivalence.
\end{thm}

\begin{rem}
Indeed, since we classify those maps up to {\it gauge} equivalence, 
our results are slightly stronger than the previous ones. 
However our claims are essentially the same as thiers by Proposition \ref{gaimC}. 
\end{rem}

Before giving a proof, we fix notation used throughout this section. 
First of all, we begin with standard representation theory of $\text{SU}(2)$. 
Let $S^k\mathbf C^2$ be the k-th symmetric power of the standard representation 
$\mathbf C^2$ of $\text{\rm SU}(2)$, 
which is an irreducible representation of $\text{SU}(2)$. 
Let 
\begin{equation}\label{decu1}
S^k\mathbf C^2=\mathbf C_k \oplus \mathbf C_{k-2} \oplus \cdots \oplus
\mathbf C_{-(k-2)} \oplus \mathbf C_{-k}
\end{equation}
be a weight decomposition with respect to a subgroup $\text{\rm U}(1)$, 
where $\mathbf C_l$ is an irreducible representation of 
$\text{\rm U}(1)$ with weight $l$. 
Consider a symmetric pair $\left(\text{SU}(2), \text{U}(1)\right)$ and 
the holomprphic line bundle $\mathcal O(k) \to \mathbf CP^1$, 
which is regarded as a homogeneous bundle $\text{SU}(2)\times_{\text{U}(1)} \mathbf C_{-k}$ 
with the canonical connection. 
Using the theory of spherical harmonics, 
we have a decomposition of $\Gamma(\mathcal O(k))$ in the $L^2$-sense:
\begin{equation}\label{declb1}
\Gamma(\mathcal O(k))=\sum_{l=0}^{\infty} S^{|k|+2l}\mathbf C^2.
\end{equation}
Moreover, $S^{|k|+2l}\mathbf C^2$ is an eigenspace of the Laplacian induced by the canonical connection. 

We have a similar theory for a symmetric pair $\left(\text{SU}(n+1), \text{U}(n)\right)$. 
Let $\mathcal H_m^{k,l}$ be an irreducible representation of $\text{SU}(m)$ 
which is the complex vector space of harmonic polynomials on $\mathbf C^{m}$ 
of bi-degree $(k,l)$. 
In particular, the space $\mathcal H_{n+1}^{k,0}$ of holomorphic polynomials has the following irreducible 
decomposition as $\text{U}(n)$-module:
\begin{equation}\label{decun}
\mathcal H_{n+1}^{k,0}=\oplus_{p=0}^{k} \mathbf C_{-k+p}\otimes \mathcal H_n^{p,0}. 
\end{equation}
Here $\mathbf C_l$ denotes an irreducible representation of 
the center of $\text{\rm U}(n)$ with weight $l$. 
Let $\mathcal O(k) \to \mathbf CP^n$ be a holomorphic line bundle of degree $k$, 
which is regarded as a homogeneous bundle $\text{SU}(n+1)\times_{\text{U}(n)} \mathbf C_{-k}$ 
with the canonical connection. 
In a similar way, 
we have an irreducible decomposition of $\Gamma(\mathcal O(k))$ in the $L^2$-sense:
\begin{equation}\label{declbn}
\Gamma(\mathcal O(k))=
\begin{cases}
\sum_{l=0}^{\infty} \mathcal H_{n+1}^{k+2l,2l}, \,\,k \geqq 0, \\
\sum_{l=0}^{\infty} \mathcal H_{n+1}^{2l,|k|+2l}, \,\,k \leqq 0.
\end{cases}
\end{equation}
Moreover, 
each representation space appeared in the decomposition is 
an eigenspace of the Laplacian induced by the canonical connection. 

\newtheorem{bocpr}{Proof of Theorems 7.10 and 7.11}
\renewcommand{\thebocpr}{}
\begin{bocpr}
We regard $\mathbf CP^n$ as a complex Grassmannian $Gr_{n}(\mathbf C^{n+1})$ 
in both cases. 
Let $f$ be a harmonic map from $\mathbf CP^1$ to $Gr_{n}(\mathbf C^{n+1})$ 
with constant energy density or a holomorphic map from $\mathbf CP^m$ to $Gr_{n}(\mathbf C^{n+1})$ 
with constant energy density. 
Then the pull-back bundle of the universal quotient bundle 
has a holomorphic vector bundle structure induced by the pull-back connection, 
because $\mathbf CP^1$ is a $1$-dimensional complex manifold in the former case and 
$f$ is a holomorphic map in the latter case. 
Since $\mathbf CP^m$ is a Fano manifold, the holomorphic line bundle structure is unique 
by Kodaira vanishing thereom. 
Thus 
the pull-back bundle with the pull-back connection is holomorphically isomorphic to a line bundle with 
the canonical connection. 
Since the pull-back bundle is of rank $1$, 
every gauge transfomation can be regarded as a non-vanishing function. 
Since the compatible connection 
with fibre metric and holomorphic structure of bundle is unique, 
we may focus our attention on the metrics. 
Then we can find a gauge transformation between a holomorphic line bundle and 
the pull-back bundle preserving metrics and connections. 
We use again the fact that the pull-back bundle is of rank $1$ 
to deduce that the mean curvature operator is the identity up to a constant multiple 
depending on the degree of $f$ 
(see the third Remark after the proof of Theorem \ref{GGenDW}). 
Therefore we can apply Theorem \ref{GenDW} 
to conclude that $f$ corresponds to a pair of an eigenspace $W$ of the Laplacian and 
a semi-positive Hermitian transformation $T$ of $W$. 
 
Suppose that the pull-back bundle is a holomorphic line bundle 
$\mathcal O(k) \to \mathbf CP^m$ with the canonical connection. 
It follows from \eqref{declb1} and \eqref{declbn} that 
$W=S^{|k|+2l}$ for some non-negative integer $l$ in the former case 
and $W=\mathcal H^{k,0}$ in the latter case. 
In each case, $\mathbf C_{-k}$ is regarded as a subspace of $W$ appeared in the decomposition 
\eqref{decu1} and \eqref{decun}. 
To apply Theorem \ref{GenDW}, we need to specify the subspace $GH(\mathbf C_{-k},\mathbf C_{-k})$ of 
$\text{H}(W)$, the set of 
Hermitian endomorphisms on $W$. 
By definition, $H(\mathbf C_{-k},\mathbf C_{-k})$ is a trivial real $\text{U}(1)$-representation of dimension $1$. 
The representation space $\text{H}(W)$ has an irreducible decomposition as follows:
$$
\text{\rm H}(W)=
\begin{cases}
\sum_{i=0}^{2|k|+4l} S^{2|k|+4l-2i} \\
\sum_{l=0}^{k} \mathcal H_{n+1}^{k-i,k-i}
\end{cases}.
$$
(Though we must take an invariant real vector space in the decomposition, we omit it.) 
Since all representations appeared in the decoposition of $\text{H}(W)$ are class one representations, 
Proposition \ref{clon} implies that $GH(\mathbf C_{-k},\mathbf C_{-k})=\text{H}(W)$. 
Consequently, $T^2=Id_W$, and so, the corresponding map is the standard map by 
$\left(\mathcal O(k) \to \mathbf CP^m, W\right)$ up to gauge equivalence. 

Consider a symmetric pair $\left(\text{SU}(n+1), \text{U}(n)\right)$ with decomposition 
denoted by $\mathfrak{g}=\mathfrak{k}\oplus \mathfrak{m}$.
The subspace $\mathfrak{m}\mathbf C_{-k}$ of $W$ is regarded as 
$\mathbf C_{-k-2} \oplus \mathbf C_{-k+2}$ in the former case 
and $\mathbf C_{-k+1}\otimes \mathcal H_{n}^{k-1,1}$ in the latter case. 
Then Lemma \ref{stharm} yields that the standard maps are the desired maps.  
\end{bocpr}
As a result, those maps are isometric embeddings up to a constant multiple of the metric. 

We would like to discuss a generalization of Theorem \ref{Calr}. 
To do so, we give a definition. 
\begin{defn}
A holomorphic map from a compact K\"ahler manifold $M$ into a complex Grassmannian or a complex 
quadric $Gr_n(\mathbf R^{n+2})$ 
with a Fubini-Study metric 
is called an {\it Einstein-Hermitian} immersion 
if the pull-back connection of the pull-back bundle $f^{\ast}Q \to M$ of the universal quotient bundle is an 
Einstein-Hermitian connection and the first Chern class $c_1(f^{\ast}Q)$ of the pull-back bundle is positive. 
\end{defn}
\begin{prop}\label{EHCP}
Let $f$ be a holomorphic map from a compact K\"ahler manifold $M$ into the complex projective space 
$\mathbf CP^n$. 
Suppose that the K\"ahler form $\omega$ on $M$ is in the cohomology class represented by 
the first Chern class of the pull-back bundle of $\mathcal O(1) \to \mathbf CP^n$. 
Then the following three conditions are equivalent. 
\begin{enumerate}
\item 
The energy density of $f$ is constant. 
\item $f$ is an Einstein-Hermitian immersion, where $\mathbf CP^n$ is regarded as 
$Gr_n(\mathbf C^{n+1})$. 
\item $f$ is an isometric immersion. 
\end{enumerate}
\end{prop}
\begin{proof}
We regard $\mathbf CP^n$ as 
$Gr_n(\mathbf C^{n+1})$ thoughout the proof. 

\noindent$(1)\Longrightarrow (2)$: Suppose that the energy density of $f$ is constant. 
Since the universal quotient bundle $\mathcal O(1) \to \mathbf CP^n$ 
is of rank $1$, it follows that the mean curvature operator $A$ is 
considered as a non-positive constant. 
Indeed, $A$ is a negative constant, because $c_1\left(f^{\ast}\mathcal O(1)\right)$ is positive 
and $\text{trace}\,A=-e(f)$ (Lemma \ref{ed}). 
Then, Proposition \ref{EH} implies that the pull-back connection on $f^{\ast}\mathcal O(1) \to M$ 
is an Einstein-Hermitian connection. 
The positivity of $c_1\left(f^{\ast}\mathcal O(1)\right)$ also yields that $f$ is an immersion. 

\noindent$(2)\Longrightarrow (3)$: Let $\omega_n$ be the  K\"ahler form on $\mathbf CP^n$ of Fubini-study type. 
Since $-2\sqrt{-1}\pi\omega_n$ is the curvature of the canonical connection on $\mathcal O(1) \to \mathbf CP^n$, 
the Einstein-Hermitian condition yields that $\wedge f^{\ast}\omega_n$ is a constant, 
where $\wedge$ denotes the contraction with $\omega$.  
Since $f^{\ast}\omega_n$ is also closed, 
the K\"ahler identities yields that $f^{\ast}\omega_n$ is a harmonic form. 
Then Hodge theory yields that $f^{\ast}\omega_n=\omega$. 

\noindent$(3)\Longrightarrow (1)$: This is trivial. 
\end{proof}

We recover Calabi's rigidity result \cite{Cal} in the case of holomorphic isometric immersions of 
a {\it compact} K\"ahler manifold into a complex projective space. 
Then we also prove that the Fubini-Study metric is induced from $L^2$-scalar product on the space 
of holomorphic sections of the Einstein-Hermitian line bundle. 

\begin{thm}{\rm (}cf. \cite{Cal}{\rm )}\label{Calabi}
Let $f$ be a full holomorphic isometric immersion of a compact K\"ahler manifold $M$ into a complex projective space 
$Gr_{n}(\mathbf C^{n+1})$. 
We denote by $(L,h_L)$ the pull-back bundle of the universal quotient bundle and the pull-back metric. 

Then 
$n$ is uniquely determined 
and $\mathbf C^{n+1}$ can be regarded as a subspace of $W=H^0(M;L)$ 
the space of holomorphic sections with $L^2$-scalar product $(\cdot,\cdot)_W$. 
The Fubini-Study metric on $Gr_{n}(\mathbf C^{n+1})$ is induced from $(\cdot,\cdot)_W$. 
The map $f$ is expressed as $\pi_n f_0$ up to image equivalence, 
where $\pi_n:W \to \mathbf C^{n+1}$ is the orthgonal projection and $f_0$ is the standard map by 
$(L \to M, H^0(M;L))$. 
\end{thm} 
\begin{proof}
Since $f$ is an isometric immersion, the K\"ahler form on $M$ 
is in the cohomology class of the first Chern Class of $L \to M$. 
By Proposition \ref{EHCP}, $f$ is an Einstein-Hermitian immersion. 
Hence Theorem \ref{rigid} yields the result. 
\end{proof}

\begin{prop}
Let $f$ be a holomorphic map from a compact K\"ahler manifold $M$ into a complex Grassman manifold 
$Gr_p(\mathbf C^n)$. 
Suppose that the K\"ahler form $\omega$ on $M$ is in the cohomology class represented by 
the first Chern class of the pull-back bundle of the universal quotient bundle. 
If $f$ is an Einstein-Hermitian immersion, then $f$ is an isometric immersion. 
\end{prop}
\begin{proof}
It follows from the hypothesis that 
the determinant line bundle $\text{det}\,f^{\ast}Q \to M$ also has an Einstein-Hermitian connection 
induced by the pull-back connection on $f^{\ast}Q \to M$. 
Since $c_1(f^{\ast}Q)=c_1(\text{det}\,f^{\ast}Q)$, a proof goes in the same way as 
$(2) \Longrightarrow (3)$ in the proof of Proposition \ref{EHCP}.

\end{proof}

\begin{prop}\label{EHQU}
Let $f$ be a holomorphic map from a compact K\"ahler manifold $M$ into a complex quadric $Gr_n(\mathbf R^{n+2})$. 
Suppose that the K\"ahler form $\omega$ on $M$ is in the cohomology class represented by 
the first Chern class of the pull-back bundle of the universal quotient bundle. 
Then the following three conditions are equivalent. 
\begin{enumerate}
\item 
The energy density of $f$ is constant. 
\item $f$ is an Einstein-Hermitian immersion. 
\item $f$ is an isometric immersion. 
\end{enumerate}
\end{prop}
\begin{proof}
Though the universal quotient bundle $Q\to Gr_n(\mathbf R^{n+2})$ 
has a holomorphic vector bundle structure 
induced by the canonical connection, 
$Q\to Gr_n(\mathbf R^{n+2})$ is considered as a real vector bundle of rank $2$ which is a subbundle 
of $\underline{\mathbf R^{n+2}}\to Gr_n(\mathbf R^{n+2})$. 
Hence, 
the mean curvature operator $A$ is a real symmetric operator on $f^{\ast}Q\to Gr_n(\mathbf R^{n+2})$.  

However, Proposition \ref{EH} yields that 
$A$ must be a complex endomorphism on $f^{\ast}Q \to M$. 
Then a proof proceeds in the same way as in the proof of Proposition \ref{EHCP}.
\end{proof}

\begin{prop}\label{sbone}
Let $M$ be a compact K\"ahler manifold with second Betti number equal to one. 
Then $f:M \to \mathbf CP^n$ or $Gr_n(\mathbf R^{n+2})$ is an Einstein-Hermitian immersion if and only if  
$f$ is a holomorphic isometric immersion up to a positive constant multiple of the metric.
\end{prop}
\begin{proof}
Under the same notation as in the proof of Propositions \ref{EHCP} and \ref{EHQU}, 
the hypothesis yields that 
there exists a positive integer $k$ 
such that 
$[f^{\ast}\omega_n]=[k\omega]$ as element of the cohomology group 
from the positivity of $c_1\left(f^{\ast}\mathcal O(1)\right)$ 
by definition. 
Then we obtain the desired result by K\"ahler identities 
and the Hodge theory.   
\end{proof}


As a generalization of Theorem \ref{Calr}, we have the rigidity of a special class of Einstein-Hermitian 
embeddings of a compact Hermitian symmetric space without irreducibility of the Einstein-Hermitian vector bundle. 
Notice that the canonical connection on any irreducible complex homogeneous vector bundle over 
a compact irreducible Hermitian symmetric space 
is an Einstein-Hermitian connection (see \cite[p.121 Proposition (6.2)]{Kob}). 

\begin{thm}\label{rigid2}
Let $(G,K_0)$ be an irreducible Hermitian symmetric pair of compact type and 
$V \to G/{K_0}$ an irreducible holomorphic homogeneous vector bundle with the canonical connection. 
We denote by $W$ the space of holomorphic sections of $V \to G/K_0$. 
Assume that any $G$-irreducible submodule of $\text{H}(W)$ is a class one representation of $(G,K_0)$, 
where $\text{H}(W)$ is the set of Hermitian endomorphisms on $W$. 
Suppose that $f$ is an Einstein-Hermitian embedding of $G/K_0$ into complex Grassmannian.  
If the pull-back bundle of the universal quotient bundle is holomorphically isomorphic to 
a direct sum of $r$-copies of $V\to G/K_0$, 
then $f$ is the standard map by $\oplus^r W$ up to gauge equivalence. 
\end{thm}
\begin{proof}
Let $\mathfrak{g}=\mathfrak{k} \oplus \mathfrak{m}$ be the standard decomposition. 
Since $(G,K_0)$ is an irreducible Hermitian symmetric pair, 
$\mathfrak{k}$ has the center $\mathfrak{u}(1)$ and so, 
we have a decomposition: $\mathfrak{k}=\mathfrak{u}(1)\oplus \mathfrak{k_1}$. 
Then $\mathfrak{m}^{\mathbf C}$, which is the complexification of 
$\mathfrak{m}$, 
is regarded as 
$\left(\mathbf C_{-d}\otimes T_1^{\ast}\right)\oplus \left(\mathbf C_{d}\otimes T_1\right)$. 
Here $d$ is a positive integer and $T_1$ is an appropriate irreducible representation space of 
$\mathfrak{k_1}$. 
If we denote by $\mathfrak{m}_{(0,1)}$ the set of tangent vectors of type $(0,1)$, 
then we have $\mathfrak{m}_{(0,1)}=\mathbf C_{d}\otimes T_1$ in the decomposition. 
Hence $S^n\mathfrak{m}_{(0,1)}=\mathbf C_{nd}\otimes S^nT_1$. 

Let $V_0$ be an irreducible complex representation of $K_0$ associated to $V \to G/K_{0}$. 
By Borel-Weil theorem, $W$ is an irreducible $G$-representation space and 
$V_0$ can be regarded as a subspace of $W$. 
Moreover $V_0$ is expressed as $\mathbf C_{-m}\otimes V_1$, 
where $m$ is a positive integer and $V_1$ is an appropriate irreducible representation space of 
$\mathfrak{k_1}$. 
Using again Borel-Weil theorem, we see that $-m$ is the smallest integer appeared in 
the $\mathfrak{u}(1)$-irreducible decompositon of $W$ and the other irreducible summands have 
weights greater than $-m$. 
In particular, we have that 
\begin{equation}\label{BWtri}
\mathfrak{m}_{(1,0)}V_0=\{0\}.
\end{equation}

Since $W$ is irreducible, 
we have a normal decomposition of $W$ (Proposition \ref{Bdeco}). 
From lemma \ref{symB} and \eqref{BWtri}, $B_n$ can be regarded as a map from 
$S^n\mathfrak{m}_{(0,1)}\otimes V_0$ to $N_n$. 
Then we see that $\text{Im}\,B_i=\mathbf C_{-m+id}\otimes U_i$, 
where $U_i$ is a representation of $\mathfrak{k}_1$ ($i\geqq 1$). 
In particular, since $\mathfrak{m}\otimes V_0 = \text{Im}\,B_1$ by definition, 
Lemma \ref{stharm} implies that the standard map by $W$ is an Einstein-Hermitian embedding 
with the pull-back bundle being holomorphically isomorphic to $V \to G/K_0$. 
Hence, Theorem \ref{GenDW} and Proposition \ref{clon} with the assumption on $\text{H}(W)$ 
yield the result. 
\end{proof}
\begin{thm}\label{linecl}
Let $G/K_0$ be a compact simply-connected homogeneous K\"ahler manifold and 
$L \to G/K_0$ a homogeneous holomorphic line bundle. 
We denote by $W$ the space of holomorphic sections of $L \to G/K_0$. 
If $L\to G/K_0$ is a positive line bundle, 
then $H(W)$ the set of Hermitian endomorphisms on $W$ consists of class one representations of $(G,K_0)$. 
\end{thm}
\begin{proof}
By Borel-Weil theorem, $W$ is a representation space of $G$. 
We denote by $N$ the dimension of $W$. 
The K\"ahler class may be chosen in the first Chern class of $L \to G/K_{0}$. 
Since $L\to G/K_0$ is a line bundle, 
$L\to G/K_0$ has a unique Einstein-Hermitian connection $\nabla$. 

Let $f_0:G/K_0 \to Gr_{N-1}(W)$ be the standard map. 
We fix a $G$-invariant Hermitian structure $h$ on $L \to G/K_0$ 
and so, $W$ has a $G$-invariant Hermitian inner product $h_W$. 
Since $f_0$ is $G$-equivariant holomorphic map and $L\to G/K_0$ is of rank one, 
the mean curvature operator $A$ can be regarded as a negative constant. 
Then Proposition \ref{EH} yields that the pull-back connection on $f^{\ast}\mathcal O(1) \to G/K_0$ 
is gauge equivalent to the Einstein-Hermitian connection $\nabla$ on $L\to G/K_0$. 
In particular, since the pull-back metric is a $G$-invarint metric on $L\to G/K_0$ by $G$-equivariance of $f_0$, 
we can assume that $h$ coincides to the pull-back metric.  
Hence $f_0$ is an Einstein-Hermitian embedding, because $L\to G/K_0$ is positive. 
It follows from Proposition \ref{EHCP} that $f_0$ can be regarded as a holomorphic isometric embedding. 

Let $V_0$ be a complex $1$-dimensional representation of $K_0$ of which the associated bundle is 
$L \to G/K_0$.  
By Lemma \ref{subsp}, we consider $V_0$ as a subspace of $W$. 
Then we take a unit vector $v_0 \in V_0$. 
If $C \in \text{H}_0(W)$, then we define a real valued function $f_C:G/K_0 \to \mathbf R$ in such a way that
\begin{equation}\label{Cala}
f_C([g]):=h_W(Cgv_0, gv_0). 
\end{equation}
It follows that the correspondence $C \mapsto f_C$ gives a $G$-equivariant homomorphism $F:\text{\rm H}_0(W) \to C^{\infty}(G/K_0)$. 

Suppose that $F(C)=0$.  
If $C$ is small enough, $Id+C$ is positive. 
Consequently, we can use $Id+C$ to define a full holomorphic map $f:G/K_0 \to Gr_{N-1}(W)$ as 
$\left(\sqrt{Id+C}\right)^{-1}f_0$. 
Then \eqref{Cala} gives 
$h_W\left((Id+C)gv_0, gv_0\right)=1$ and so, 
the pull-back metric coincides to $h$. 
Since the pull-back bundle by $f$ is holomorphically isomorphic to $L\to G/K_0$, 
the uniqueness of the compatible connection implies that the pull-back connection 
also coincides to $\nabla$. 
In a similar way, $f$ also turns out to be a holomorphic isometric embedding. 
From Calabi's rigidity theorem \cite{Cal} or Theorem \ref{Calabi}, $f$ must be image equivalent to $f_0$. 
Proposition \ref{gaimC} implies that $f$ and $f_0$ with their natural identifications are also gauge equivalent. 
It follows from Theorem \ref{HGenDW} that $C=0$. 
Thus $\text{H}_0(W)$ can be regarded as $G$-submodule of $C^{\infty}(G/K_0)$. 

Since every irreducible submodule of $C^{\infty}(G/K_0)$ is class one representation of $(G,K_0)$, 
we obtain the desired result. 
\end{proof}

To state a corollary, we recall the definition of {\it projectively flat immersion} \cite{Kog-Nag}. 
Let $f$ be an Einstein-Hermitian immersion of $M$. 
If the pull-back connection on the pull-back bundle of the quotient bundle by $f$ 
is projectively flat, 
then $f$ is called {\it projectively flat immersion}. 
If $f$ is a projectively flat immersion and the positive multiple of the K\"ahler form is 
in the first Chern class of the pull-back bundle, 
then the curvature form on the pull-back bundle is parallel, 
because $f$ is also an Einstein-Hermitian immersion by definition. 
In this case, note that the pull-back bundle is the orthogonal sum of copies of a positive line bundle $L \to M$ 
with an Einstein-Hermitian connection 
by the holonomy decomposition. 

\begin{cor}
If $f$ is a projectively flat embedding of an irreducible Hermitian symmetric space $M$ 
into a complex Grassmannian, 
then there exists a positive line bundle $L\to M$ such that 
$f$ is the standard map by $(\oplus^r L\to M, \oplus^r W)$ up to gauge equivalence. 
Here, $W$ is the space of holomorphic sections of the homogeneous line bundle $L\to M$. 
\end{cor}

\begin{proof}
Since every irreducible Hermitian symmetric space has second Betti number equal to one, 
the first Chern class of the pull-back bundle of the quotient bundle 
is cohomologous to the K\"ahler class up to positive multiple.  
Then the K\"ahler identity and the Hodge decomposition implies that 
the pull-back bundle is $\oplus^r L\to M$, 
where $L\to M$ is a line bundle with an Einstein-Hermitian connection $\nabla$. 
Since $M$ is a Hermitian symmetric space, $L\to M$ is a homogeneous line budle and 
$\nabla$ is the canonical connection. 
Then Theorems \ref{rigid2} and \ref{linecl} imply the result. 
\end{proof}

We also use Proposition \ref{clon} to obtain a rigidity result on harmonic maps. 

\begin{thm}\label{righarm}
Let $G/K_0$ be a compact reductive Riemannian homogeneous space and $V \to G/K_0$ 
a homogeneous vector bundle with an invariant metric and an invariant connection preserving the metric. 

Suppose that $f:G/K_0\to Gr_p(W)$ is a harmonic map satisfying the gauge condition for $V \to G/K_0$ and 
the Einstein-Hermitian condition. 
Then, by Theorem \ref{GGenDW}, $W$ can be regarded as an eigenspace of the Laplacian acting on $\Gamma(V)$ 
and so, a $G$-representation space with $G$-invariant scalar product by invariance of the metric and the connection. 

If any $G$-irreducible submodule of $\text{\rm H}(W)$ is a class one representation of $(G,K_0)$, 
then $f$ is the standard map by $W$ up to gauge equivalence. 
\end{thm}
\begin{proof}
From Theorems \ref{GGenDW} and \ref{GenDW}, 
we have a semi-positive transformation $T$ of $W$ satisfying 
$$
\left(T^2-Id_W, GH(V_0,V_0) \right)_H=0, 
$$
where $V_0$ is the $K_0$-representation associated with $V\to G/K_0$. 
Then Proposition \ref{clon} and the hypothesis imply that $\text{\rm H}(W)=GH(V_0,V_0)$. 
Thus $T^2=Id_W$. 
\end{proof}

Toth gives a conception of {\it polynomial minimal immersion} 
between complex projective spaces \cite{Tot}. 
In the definition of polynomial minimal immersions, 
Toth makes use of $\mathcal H_{n+1}^{k,l}$ to define {\it polynomial maps} between spheres 
and of the Hopf fibration to get a map between complex projective spaces. 
This enables us to apply $\text{U}(n+1)$-representation theory instead of 
$\text{SO}(2n+2)$-representation theory
in the original do Carmo-Wallch theory. 
In addition, Toth implicitly requires condition {\rm (i)} in Theorem \ref{GenDW} as {\it horizontality}. 
Theorem \ref{harmGrirr} implies that the former condition is not needed to develop the theory. 
We replace a polynomilal minimal immersion by polynomial harmonic map with constant energy density. 

\begin{lemma}
Let $f:\mathbf CP^m \to \mathbf CP^n$ $(m\geqq 2)$
be a full harmonic map with constant energy density. 
Then $f$ is a polynomial harmonic map in the sense of Toth 
if and only if 
the pull-back bundle of the universal quotient bundle with pull-back connection by $f$ 
is gauge equivalent to a complex line bundle with the canonical connection. 
\end{lemma}
\begin{proof}
Since the universal quotient bundle is of rank $1$ and $f$ has constant energy density, 
the mean curvature operator is proportional to the identity up to constant 
(see the third Remark after the proof of Theorem \ref{GGenDW}). 

The sufficient implication holds by definition of polynomial harmonic map 
(see condition (3) in \cite{Tot}). 

Suppose that 
the pull-back bundle of the quotient bundle with pull-back connection 
is gauge equivalent to a complex line bundle with the canonical connection. 
Then Theorem \ref{harmGrirr} and the decomposition \eqref{declbn} imply that 
$\mathbf C^{n+1}\subset \mathcal H_{m+1}^{k,l}$ for some non-negative integers $k$ and $l$. 
This yields that $f$ is a polynomial harmonic map. 
\end{proof}

Toth gives an estimate of the dimension of the moduli space by 
the image equivalence relation. 
By Proposition \ref{gaimC}, we can apply Toth's estimate to get an estimate of moduli space by 
gauge equivalence.

Finally, we consider Einstein-Hermitian embeddings of the projective line into complex quadrics 
$Gr_n(\mathbf R^{n+2})$. 
By Proposition \ref{sbone}, such embeddings are really holomorphic isometric embeddings up to 
constant multiples of the metrics. 
Though research on harmonic maps from the projective line into quadrics 
has been pursued before from various viewpoints 
(for example,  \cite{ChiZhe}, \cite{FJXX}, \cite{LiYu} and \cite{Wolf}), 
we would like to apply Theorem \ref{GenDW} to give a description of the moduli. 

Notice that the curvature form $R$ of the canonical connection on the universal quotient bundle 
is related to the fundamental $2$-form 
$\omega_Q$ on $Gr_n(\mathbf R^{n+2})$ in such a way that $R=-2\pi \sqrt{-1}\omega_Q$. 
Denote by $\omega_0$ the fundamental $2$-form on $\mathbf CP^1$ 
satisfying $R_{\mathcal O(1)}=-2\pi \sqrt{-1}\omega_0$, 
where $R_{\mathcal O(1)}$ is the curvature form of the canonical connection on 
the hyperplane bundle over $\mathbf CP^1$. 
\begin{defn}
Let $f:\mathbf CP^1 \to Gr_n(\mathbf R^{n+2})$ be a holomorphic embedding. 
Then $f$ is called an isometric embedding of degree $k$ if 
$f^{\ast}\omega_Q=k\omega_0$ (and so, $k$ must be a positive integer). 
\end{defn}



\begin{lemma}\label{EHi}
Let $f:\mathbf CP^1 \to Gr_n(\mathbf R^{n+2})$ be a holomorphic embedding. 
Then $f$ is an isometric embedding of degree $k$ if and only if 
the pull-back connection is gauge equivalent to the canoical connection 
on $\mathcal O(k)\to \mathbf CP^1$. 
Under these conditions, the mean curvature operator is proportional to 
the identity up to negative constant. 
\end{lemma}
\begin{proof}
Since the holomorphic bundle structure of any line bundle on $\mathbf CP^1$ is unique, 
there exists a non-negative integer $k$ such that $f^{\ast}Q \to \mathbf CP^1$ is holomorphically 
isomorphic to $\mathcal O(k) \to \mathbf CP^1$. 


Then Proposition \ref{sbone} yields the result, 
because the canonical connection is an Einstein-Hermitian connection. 

Then Proposition \ref{EH} and the definition of Einstein-Hermitian immersion yield that 
the mean curvature operator is proportional to the identity. 
\end{proof}

Since the holonomy group of the connection is irreducible, 
the pull-back metric is also the same as the invariant fibre metric up to a positive real constant. 
If we change the inner product on $\mathbf R^{n+2}$ or the invariant metric on the line bundle 
by a positive constant, 
then we can assume that the pull-back metric is also the same as the invariant metric 
from the beginning. 

From this obsevation with Lemma  \ref{EHi}, we can apply Theorem \ref{HGenDWI} to obtain the moduli space $\mathcal M_k$ 
of holomorphic isometric embeddings of degree $k$ 
by the gauge equivalence of maps. 
We also use the Einstein-Hermitian connection to obtain 
that any holomorphic section of $\mathcal O(k)\to \mathbf CP^1$ 
is an eigensection. 

Let $\left(\text{SU}(2), \text{U}(1) \right)$ be the corresponding symmetric pair to 
$\mathbf CP^1$ 
and $\mathfrak{g}=\mathfrak{k}\oplus \mathfrak{m}$ the standard decomposition of 
the Lie algebra $\mathfrak{g}=\mathfrak{su}(2)$, 
where $\mathfrak{k}=\mathfrak{u}(1)$. 
In addition, notice that $\mathfrak{m}^{\mathbf C}=\mathbf C_2\oplus \mathbf C_{-2}$. 

Using Lemma \ref{stharm} and weight decomposition \eqref{decu1} of $S^k\mathbf C^2$, we can show
\begin{lemma}\label{quadstan}
The standard map by $H^0\left(\mathbf CP^1, \mathcal O(k)\right)\cong S^k\mathbf C^2$ 
is a holomorphic isometric embedding of a complex projective line of degree $k$ 
into $Gr_{2k}(\mathbf R^{2k+2})$. 
\end{lemma}
First of all, we consder the case that $k=1$. 
\begin{thm}\label{mod1}
If $f$ is a holomorphic isometric embedding of a complex projective line of degree $1$ 
into a complex quadric, 
then $f$ is the standard map by $\left(\mathcal O(1)\to \mathbf CP^1, H^0\left(\mathbf CP^1, \mathcal O(1)\right)\right)$ 
up to gauge equivalence. 
\end{thm}
\begin{proof}
Let $S^1\mathbf C^2 \cong \mathbf C^2$ be the standard representation of $\text{SU}(2)$ 
and $\mathbf C^2=\mathbf C_1\oplus \mathbf C_{-1}$ the weight decomposition. 
The line bundle $\mathcal O(1)\to \mathbf CP^1$ 
is expressed as $\text{SU}(2)\times_{\text{U}(1)} \mathbf C_{-1}$ and 
we can identify the representation space $\mathbf C_{-1}$ of $\text{U}(1)$ with the subspace 
denoted by the same symbol 
of the weight decomposition of $\mathbf C^2$ by Lemma \ref{subsp}. 
Then we can see that $\mathfrak{m}\mathbf C_{-1}=\mathbf C_1$, 
where $\mathfrak{su}(2)=\mathfrak{u}(1)\oplus \mathfrak{m}$ is a standard decomposition of 
the symmetric pair $(\text{SU}(2), \text{U}(1))$. 
Let $f$ be a holomorphic isometric immersion into $Gr_n(\mathbf R^{n+2})$.  
Then, Proposition \ref{sbone} yields that $f$ is an Einstein-Hermitian embedding. 
Thus we can apply Theorem \ref{GenDW}, 
because the Einstein-Hermitian connection is the canonical connection on $\mathcal O(1) \to \mathbf CP^1$. 
Since $f$ is of degree one and $H^0\left(\mathbf CP^1, \mathcal O(1)\right)=\mathbf C^2$ 
by Borel-Weil, Theorem \ref{GenDW} implies that $n+2 \leqq 4$. 
Hence we consider an Einstein-Hermitian embedding $f:\mathbf CP^1 \to Gr_2(\mathbf R^4)$. 

We regard $\mathbf C^2$ as a real vector space $\mathbf R^4$ with the complex structure $J$ 
when applying Theorem \ref{GenDW}, because $Gr_{2}(\mathbf R^{4})$ is a real Grassmannian. 
Consequently, $\text{H}(\mathbf R^4)$ denotes the set of symmetric endomorphism on $\mathbf R^4$ 
in our convention. 
Moreover, we only need to consider the equation 
\begin{equation}\label{Cnab}
ev\circ C \circ \nabla ev^{\ast}=0
\end{equation}
by Corollary \ref{Holsuftf} and Lemma \ref{quadstan}, 
where $C$ is a trace-free symmetric endomorphism on $\mathbf R^4$. 
In this case, \eqref{Cnab} is equivalent to the condition that 
$$
\left(C,GH(\mathfrak m\mathbf C_{-1},\mathbf C_{-1}) \right)
=\left(C,GH(\mathbf C_{1},\mathbf C_{-1}) \right)=0, 
$$
where $(\cdot,\cdot)$ is $\text{SU}(2)$-invariant inner product on $\mathbf R^4$. 

We can see that  
\begin{equation}\label{1deco}
\text{H}(\mathbf R^4)=3 \mathbf R^3 \oplus \mathbf R, 
\end{equation}
which is the irreducible decomposition of $\text{SU}(2)$-module $\text{H}(\mathbf R^4)$. 

To apply Theorem \ref{GenDW}, we need to understand the decomposition \eqref{1deco} in detail. 
To do so, let $j$ be an invariant quaternion structure on $\mathbf C^2$. 
Then,
$$
\mathbf R^3=\sqrt{-1}\rho(\mathfrak{g}), \,
\mathbf R^3=j\sqrt{-1}\rho(\mathfrak{g}), \,
\mathbf R^3=Jj\sqrt{-1}\rho(\mathfrak{g}),
$$
where $\rho:\mathfrak {su}(2)\to \text{End}(\mathbf C^2)$ denotes the representaion. 
Notice that $\sqrt{-1}\rho(\mathfrak{g})$ is the set of Hermitian endomorphisms on $(\mathbf R^4,J)$. 

We take an orthonormal basis $\{v_1,Jv_1,v_{-1},Jv_{-1} \}$ of $\mathbf R^4$ as 
$v_1,Jv_1 \in \mathbf C_1$, $v_{-1},Jv_{-1} \in \mathbf C_{-1}$ and $jv_1=v_{-1}$. 
Then, 
$H(\mathbf C_{1},\mathbf C_{-1})$ is spanned by 
$$
H(v_1,v_{-1}), H(Jv_1,v_{-1}), H(v_1,Jv_{-1})\, \text{and}\, H(Jv_1,Jv_{-1}).
$$
In general, we have $H(Ju,Jv)=-JH(u,v)J$ for arbitrary $u,v \in \mathbf R^4$. 
Hence $H(u,v)+H(Ju,Jv)$ is a Hermitian endomorphism on $(\mathbf R^4, J)$, 
because it commutes with $J$.
In particular,  $H(v_{1},v_{-1})+H(Jv_{1},Jv_{-1})$ and 
$H(v_{1},Jv_{-1})-H(Jv_{1},v_{-1})$ are Hermitian endomorphisms on $(\mathbf R^4, J)$.
Thus 
$$
\sqrt{-1}\rho(\mathfrak{g}) \subset GH(\mathbf C_{1},\mathbf C_{-1}). 
$$

Next, we see that 
\begin{align}\label{jrho}
&2H(v_1,v_{-1})-2H(Jv_1, Jv_{-1}) \\
=&j\left\{H(v_1,v_{1})-H(v_{-1},v_{-1})+H(Jv_{1},Jv_{1})-H(Jv_{-1},Jv_{-1})\right\} \notag, 
\end{align}
and so, 
$$
j\sqrt{-1}\rho(\mathfrak{g}) \subset GH(\mathbf C_{1},\mathbf C_{-1}). 
$$
 
Finally, we get 
$$
H(v_{1},Jv_{-1})+H(Jv_{1},v_{-1})=J\left\{ H(v_{1},v_{-1})-H(Jv_{1},Jv_{-1})\right\}. 
$$
It follows from \eqref{jrho} that 
$$
Jj\sqrt{-1}\rho(\mathfrak{g}) \subset GH(\mathbf C_{1},\mathbf C_{-1}).
$$

Therefore, $\text{\rm H}_0(\mathbf R^4)=GH(\mathbf C_{1},\mathbf C_{-1})$. 
The Remark after Corollary \ref{traceIT} yields the results. 

\end{proof}

Next, we are interested in holomorphic isometric embeddings of degree $2$. 
Notice that when the degree is even, say $2l$, $H^0\left(\mathbf CP^1, \mathcal O(2l)\right)$ 
has an invariant real subspace denoted by $W^l_{\mathbf R}$ of real dimension $2l+1$. 
Since $W^l_{\mathbf R}$ also globally generates $\mathcal O(2l)\to \mathbf CP^1$, 
we have a standard map by $W^l_{\mathbf R}$ which turns out to be a holomorphic isometric 
embedding of degree $2l$ by Lemma \ref{stharm} and weight decomposition. 
We call the standard map by $\left(\mathcal O(2l) \to \mathbf CP^1, W^l_{\mathbf R}\right)$ 
{\it real standard map}.  


\begin{thm}\label{mod}
If $f:\mathbf CP^1 \to Gr_n(\mathbf R^{n+2})$ is a full holomorphic isometric embedding 
of degree $2$ into 
a complex quadric, then $n\leqq 4$. 

Let $\mathcal M_2$ be the moduli space of full holomorphic isometric embeddings 
of the complex projective line into $Gr_{4}(\mathbf R^{6})$ 
of degree $2$ 
by the gauge equivalence of maps. 
Then $\mathcal M_2$ can be regarded as an open unit disk in $\mathbf C$. 

If we take a compactification of $\mathcal M_2$ by the topology induced from 
$L^2$ scalar product on $\Gamma(\mathcal O(2))$, 
each boundary point of $\mathcal M_2$ corresponds to a real standard map whose image 
is included in a totally geodesic submanifold $Gr_1(\mathbf R^{3})$ of 
$Gr_{4}(\mathbf R^{6})$.  
Each totally geodesic submanifold $Gr_1(\mathbf R^{3})$ 
is specified as the common zero set of some sections of the universal quotient bundle $Q\to Gr_{4}(\mathbf R^{6})$. 
\end{thm}
\begin{proof} 
We use the same notation as in the proof of Theorem \ref{mod1} and  
begin with a representation theory of $\text{SU}(2)$ and $\text{U}(1)$. 
Let $S^2\mathbf C^2$ be the complexification of the Lie algebra of $\text{SU}(2)$ with 
a real structure $\sigma$. 
A weight decomposition of $S^2\mathbf C^2$ is $\mathbf C_2 \oplus \mathbf C_0 \oplus \mathbf C_{-2}$. 
The associated line bundle $\text{SU}(2)\times_{\text{U}(1)} \mathbf C_{-2}$ is 
a holomorphic line bundle $\mathcal O(2) \to \mathbf CP^1$ and 
$H^0\left( \mathbf CP^1, \mathcal O(2)\right)$ the set of holomorphic sections 
of $\mathcal O(2)\to \mathbf CP^1$ 
is identified with 
$S^2\mathbf C^2$ by Borel-Weil. 
Then the representation space $\mathbf C_{-2}$ of $\text{U}(1)$ is regarded as a subspace 
of $S^2\mathbf C^2$ denoted by the same symbol by Lemma \ref{subsp}. 
We can see that $\mathfrak{m}\mathbf C_{-2}=\mathbf C_0$. 

Let $f:\mathbf CP^1 \to Gr_n(\mathbf R^{n+2})$ be a holomorphic isometric embedding 
of degree $2$. 
Then Lemma \ref{EHi} implies that $f$ is an Einstein-Hermitian embedding and so, $n\leqq 4$ by Theorem \ref{GenDW}. 

To apply Theorem \ref{GenDW}, $S^2\mathbf C^2$ is regarded as a real vector space $\mathbf R^6$ with 
the complex structure $J$. 
Let $\text{H}(\mathbf R^6)$ be the set of symmetric endomorphisms on $\mathbf R^6$. 
Notice that $\text{H}(\mathbf R^6,J)$ the set of Hermitian endomorphisms on $(\mathbf R^6, J)$ 
is a real subspace of $\text{H}(\mathbf R^6)$. 
The Clebsh-Gordan formula yields that 
the complexification of $\text{H}(\mathbf R^6,J)$ is decomposed as 
$S^4\mathbf C^2 \oplus S^2\mathbf C^2 \oplus \mathbf C$.  
Since these three spaces have invariant real structure, 
we denote by $D_4, D_2$ and $\mathbf R$ the corresponding real subspaces of $\text{H}(\mathbf R^6,J)$, respectively. 
We claim that 
\begin{equation}\label{6deco}
\text{H}(\mathbf R^6)=\left(D_4 \oplus \sigma D_4 \oplus J\sigma D_4 \right) 
\oplus D_2 \oplus \left(\mathbf R \oplus \mathbf R\sigma \oplus \mathbf RJ\sigma \right). 
\end{equation}
We fix an orthonormal basis $\{v_2,v_0,v_{-2}, Jv_2, Jv_0, Jv_{-2}\}$ of $\mathbf R^6$ 
in such a way that 
$v_i \in \mathbf C_{i}$ and $\sigma(v_i)=v_{-i}$ ($i=2,0,-2$). 
Using matrix representation and the block decomposition according to 
$\mathbf R^6=\text{Span}(v_2,v_0,v_{-2})\oplus \text{Span}(Jv_2,Jv_0,Jv_{-2})$, 
we have 
$$
D_4=\left\{ 
\begin{pmatrix}
D & O \\
O & D 
\end{pmatrix}\Bigm| {}^tD=D
\right\}, 
$$
and so, 
$$
\sigma D_4=\left\{ 
\begin{pmatrix}
D & O \\
O & -D 
\end{pmatrix}\Bigm| {}^tD=D
\right\}, 
\,\text{and}\, 
J\sigma D_4=\left\{ 
\begin{pmatrix}
O & D \\
D & O
\end{pmatrix}\Bigm| {}^tD=D
\right\}. 
$$
Moreover, we have 
$$
D_2=\left\{ 
\begin{pmatrix}
O & -C \\
C & O
\end{pmatrix}\Bigm| {}^tC=-C
\right\}. 
$$

By the same reason (Corollary \ref{Holsuftf} and Lemma \ref{quadstan})
as in the proof of Theorem \ref{mod1}, 
$GH(\mathbf C_0, \mathbf C_{-2})$ is needed to be specified as a subspace of $\text{H}_0(\mathbf R^6)$. 
From the definition, 
$H(\mathbf C_0, \mathbf C_{-2})$ 
is spanned by 
$$
H(v_0,v_{-2}), H(Jv_0,v_{-2}), H(v_0,Jv_{-2})\, \text{and}\, H(Jv_0,Jv_{-2}).
$$
The characterization of the decomposition of $\text{H}_0(\mathbf R^6)$ yields that 
\begin{align*}
&H(v_0,v_{-2})+H(Jv_0,Jv_{-2})\in D_4, \quad H(Jv_0,v_{-2})-H(v_0,Jv_{-2})\in D_2, \\
&H(v_0,v_{-2})-H(Jv_0,Jv_{-2})\in \sigma D_4, \quad H(Jv_0,v_{-2})+H(v_0,Jv_{-2})\in J\sigma D_4.
\end{align*}
Thus $\mathbf R \sigma \oplus \mathbf RJ\sigma$ is the orthogonal complement of 
$GH(\mathbf C_0, \mathbf C_{-2})$ in $\text{H}_0(\mathbf R^6)$. 
Notice that the complex structure $J$ on $\mathbf R^6$ gives a complex structure on 
$\mathbf R \sigma \oplus \mathbf RJ\sigma$. 
From the Remark after Corollary \ref{traceIT}, the moduli space 
$\mathcal M_2$ can be regarded as an open convex body in $\mathbf R \sigma \oplus \mathbf RJ\sigma$. 
A symmetric transformation $Id+\left(a\sigma + bJ\sigma \right)$ is positive, where $a,b \in \mathbf R$ 
if and only if $a^2+b^2<1$.  
Thus $\mathcal M_2=\left\{z \in \mathbf C\,| \, |z|^2<1 \right\}$. 

Next we consider a natural compactification of $\mathcal M_2$. 
Suppose that $a^2+b^2=1$. 
Then $(a+bJ)\sigma$ is also an invariant real structure on $S^2\mathbf C^2$. 
Hence we may consider only the case that $a=1$ and $b=0$. 
Since the kernel of $Id+\sigma$ is $\text{Span}(Jv_2,Jv_0,Jv_{-2})$, 
Theorem \ref{GGenDW} implies that $Id+\sigma$ determines a totally geodesic submanifold 
$Gr_1(\mathbf R^3)$ of $Gr_4(\mathbf R^6)$ and a holomorphic isometric embedding into 
the submanifold $Gr_1(\mathbf R^3)$ represented by $2Id_3$. 
This map is nothing but a standard map by $\mathbf R^3=\left(S^2\mathbf C^2\right)_{\mathbf R}$ 
which is an invariant real subspace of $S^2\mathbf C^2$. 
\end{proof}

We have a geometric interpretation of the existence of the complex structure on the moduli space $\mathcal M_2$. 
Let $(f,\phi)$ be a full holomorphic isometric embedding of $\mathbf CP^1$ into $Gr_4(\mathbf R^6)$ of degree $2$ 
with a bundle isomorphism $\mathcal O(2) \cong f^{\ast}Q$. 
Then $(f,J\phi)$ is also such an embedding, 
where $J$ is a complex structue on $\mathcal O(2) \to \mathbf CP^1$. 
Notice that we must regard the bundle as a real vector bundle, when applying Theorem \ref{GGenDW}, 
because the quadric can be expressed as {\it real} Grassmanian. 
The complex strucutre $J$ on the bundle 
induces a complex structure on the space of the sections. 
In paticular, $H^0(\mathbf CP^1, \mathcal O(2))$ is a complex subspace, 
and so, the induced complex structure coincides with the complex structure on 
$H^0(\mathbf CP^1, \mathcal O(2))$ in the proof. 
Thus we have a complex structure on $\mathcal M_2$. 

On the other hand, since $I+C$ is invariant under the $\text{SU}(2)$-action, 
we can deduce that all holomorphic isometric embeddings of $\mathbf CP^1$ into $Gr_4(\mathbf R^6)$ of degree $2$ 
is $\text{SU}(2)$-equivariant. 
This is a result of \cite{FJXX}. 

Next, we consider the image equivalence of maps. 
The holonomy group of the canonical connection on $\mathcal O(2) \to \mathbf CP^1$ is 
the strucutre group $\text{U}(1)$ of the bundle. 
Therefore the centralizer of the holonomy group is also the structure group $S^1=\text{U}(1)$. 
Since $\mathcal O(2)=\text{SU}(2)\times_{\text{U}(1)} \mathbf C_{-2}$, 
the centralizer acts on the vector bundle with weight $-2$. 
From the equation \eqref{imeq}, 
We can conclude that the centralizer $S^1$ of the holonomy group acts on $\mathcal M_2$ 
with a scalar multiplication of weight $-2$. 

\begin{thm}{\rm (}\cite{ChiZhe} and \cite{LiYu}{\rm )}
Let $\mathbf M_2$ be the moduli space of holomorphic isometric embeddings 
of the complex projective line into $Gr_{4}(\mathbf R^{6})$ 
of degree $2$ 
by the image equivalence of maps. 
Then $\mathbf M_2=\mathcal M_2\slash S^1=[0,1]$. 
\end{thm}
\begin{proof}
Let $f_1$ and $f_2$ be those image equivalent maps. 
By definition of the image equivalence, we have an isometry $\psi$ of $Gr_4(\mathbf R^6)$ such that 
$f_2=\psi \circ f_1$. 
When we denote by $\phi_1$ and $\phi_2$ the natural identifications induced by $f_1$ and $f_2$, respectively, 
it follows from \eqref{imeq} and the polar decomposition that 
$\tilde \psi \in S^1$. 
Hence Theorem \ref{mod} implies the result.
\end{proof}

The moduli space of holomorphic isometric embeddings of the complex projective line into quadrics of higher degree 
can be described by the same method. 
To do so, more detailed analysis of $\text{SU}(2)$-representations is required. 
This subject will be the object of the forthcoming paper \cite{MNT}.

\subsection{Comparison with the ADHM-construction}

Let $M$ denote the 4-dimensional sphere $S^4=\mathbf HP^1$. 
We follow the notation of the Example after Lemma \ref{stharm}.  

Let $\mathbf H \to M$ be the tautological bundle. 
The Penrose transform implies that the space of twistor sections 
of $\mathbf H \to M$ is naturally identified with 
$H^0\left(\mathbf CP^3;\mathcal O(1)\right)$, 
where $\mathbf CP^3$ is the twistor space of $M$. 
The Borel-Weil theorem implies that 
$H^0\left(\mathbf CP^3;\mathcal O(1)\right)$ 
is regarded as the standard representation $\mathbf C^{4^{\ast}}\cong \mathbf C^4$ 
of $\text{Sp}(2)$. 
Since $\mathbf C^4$ globally generates $\mathbf H \to M$, 
we can consider the induced map $f_0:M \to Gr_2(\mathbf C^4)$. 
This is nothing but a standard map, 
because $\mathbf H \to M$ is a homogeneous bundle 
$\text{Sp}(2)\times_{\text{Sp}_+(1)\times \text{Sp}_-(1)}\mathbf H$. 

As $\text{Sp}_+(1)\times \text{Sp}_-(1)$-module, 
$$
\mathbf C^4=\mathbf H \oplus \mathbf E, 
$$
and 
$$
S^2\mathbf C^4=S^2\mathbf H \oplus \mathbf H \otimes \mathbf E \oplus S^2\mathbf E \,\text{and}\,
\wedge^2\mathbf C^4=\mathbf C \oplus \mathbf H \otimes \mathbf E \oplus \mathbf C.
$$

Then, it follows from Theorem \ref{righarm} that $f_0$ can not be deformed as a harmonic 
map satisfying the gauge and EH conditions. 


Next we consider the ADHM-construction of instantons \cite{A}. 
For simplicity, we focus our attention on $1$-instantons. 
Let $\alpha:\underline{\mathbf C^4} \to \mathbf H$ be a surjective 
bundle homomorphism satisfying the twistor equation \cite{H}:
$$
\mathcal D\alpha=0, 
$$
where $\alpha$ is regarded as a section of 
$\mathbf C^{4^{\ast}} \otimes \mathbf H \to S^4$. 
Suppose that $\mathbf C^4$ has an invariant Hermitian inner product and 
an invariant quaternion structure $j$ under the action of $\text{Sp}(2)$. 
Then we have the induced real structure of 
$\mathbf C^{4^{\ast}} \otimes \mathbf H \cong 
\mathbf C^{4} \otimes \mathbf H$. 

Using the twistor space and the Borel-Weil theorem, 
we know that 
$\alpha$ can be expressed as 
$$
\alpha_{[g]}(w)=\left[g, \pi(g^{-1}Tw)\right], \quad g\in \text{Sp}(2), 
$$
where $T$ is a positive Hermitian endomorphism of $\mathbf C^4$, 
and $\pi:\mathbf C^4 \to \mathbf H$ is the orthogonal projection.  
The ADHM-construction requires that $T$ should satisfy 
$$
T^2=Id+C, \quad C \in (\wedge^2_0\mathbf C^4)^{\mathbf R}. 
$$
Here $\wedge^2_0\mathbf C^4$ is the orthogonal complement to 
$\mathbf  C \omega$ in $\wedge^2 \mathbf C^4$, 
where $\omega$ is an invariant symplectic form on $\mathbf C^4$, 
which is an 
irreducible representation of $\text{Sp}(2)$. 
Since $\wedge^2_0\mathbf C^4$ has an invariant real structure induced by $j$, 
we can take a real representation 
$(\wedge^2_0\mathbf C^4)^{\mathbf R}$. 
If $C$ is small enough, 
then $Id+C$ is positive, and so, $\text{Ker}\,\alpha \subset \underline{\mathbf C^4}$ is an 
instanton with the induced metric and connection from 
$\mathbf C^4$. 

If we regard $\alpha$ as an evaluation homomorphism, then we obtain the 
induced map $f:M \to Gr_2(\mathbf C^4)$:
$$
f\left([g] \right)=Tg\mathbf E. 
$$
When $T$ is the identity or $C=O$, we recover the standard map $f_0$. 
In the case that $C\not= O$, the pull-back connection on the pull-back 
$f^{\ast}Q \to M$ 
is not gauge equivalent to the canonical connection on $\mathbf H \to M$ 
(but is still self-dual). 

In both cases of the generalization of the do Carmo-Wallach construction and 
the ADHM-construction, 
the emergence of linear equations 
($(\Delta + A)t=0$ and $\mathcal D\alpha=0$, respectively) 
makes it possible to describe moduli spaces in linear algebraic terms.


\begin{thebibliography}{34}
%
\bibitem{A} 
M.F.Atiyah,
``Geometry of Yang-Mills Fields"
Lezioni Fermiane, Scuola Normale Superiore, 
Pisa (1979)
%
%
\bibitem{Ban-Ohn} 
S.Bando and Y.Ohnita, 
{\it Minimal 2-spheres with constant curvature in $\mathbf P_n(\mathbf C)$}, 
J. Math. Soc. Japan {\bf 39} (1987), 477--487
%
\bibitem{Cal}
E.Calabi, 
{\it Isometric Imbedding of Complex Manifolds}, 
Ann. of Math. {\bf 58} (1953), 1--23
\bibitem{ChiZhe}
Q.S.Chi and Y.Zheng, 
{\it Rigidity of pseudo-holomorphic curves of constant curvature in Grassmann manifolds}, 
Trans. Amer.Math. Soc. {\bf 313} (1989), 393-406
\bibitem{DoC-Wal}
M.P.do Carmo and N.R.Wallach, 
{\it Minimal immersions of spheres into spheres},  
Ann.Math. {\bf 93} (1971), 43--62
%
\bibitem{Do-3}
S.K.Donaldson, 
{\it Scalar curvature and projective embeddings, I}, 
J.Diff.Geom. {\bf 59} (2001), 479--522
%
%
\bibitem{D-K}
S.~K.~Donaldson and P.~B.~Kronheimer,
``The Geometry of Four-Manifolds''
Clarendon Press, Oxford (1990)
%
%
\bibitem{Ee-Sam}
J.Eells and J.H.Sampson, 
{\it Harmonic mappings of Riemannian manifolds}, 
Amer. J. Math. {\bf 86} (1964), 109-160
%
\bibitem{FJXX}
J.Fei, X.Jiao, L.Xiao and X.Xu, 
{\it On the Classification of Homogeneous 2-Spheres in Comple Grassmannians}, 
Osaka J. Math {\bf 50} (2013), 135-152
%
%
\bibitem{G-L}
K.Galicki and Lawson, 
{\it Quaternionic reduction and quaternionic orbifolds}, 
Math. Ann. {\bf 282} (1988), 1-21
%
\bibitem{Gam}
A.Gambioli, 
{\it Latent Quaternionic Geometry}, 
Tokyo Journal of Mathematics {\bf 31} (2008), 203--223
%
%
\bibitem{H}
N.J.Hitchin, 
{\it Linear field equations on self-dual spaces}, 
Proc.R.Soc.A. {\bf 370} (1980), 173--191
%
\bibitem{Joy}
D.Joyce, 
{\it Special Lagrangian Submanifolds with Isolated Conical Singularities. II. Moduli Spaces}, 
Annals of Global Analysis and Geometry {\bf 25} (2004), 301--352
%
\bibitem{Kog-Nag}
I.Koga and Y.Nagatomo, 
A study of submanifolds of complex Grassmannian manifolds with parallel second fundamental form, 
a preprint
%
\bibitem{LiYu}
Z.Q.Li and Z.H.Yu, 
{\it Constant curved minimal 2-spheres in G(2,4)}, 
Manuscripta Math. {\bf 100} (1999), 305-316
\bibitem{Kob}
S.Kobayashi, 
``Differential Geometry of Complex Vector Bundles", 
Iwanami Shoten and Princeton University, Tokyo (1987)
%
\bibitem{MNT}
O.Macia, Y.Nagatomo and M.Takahashi, 
Holomorphic isometric embeddings of projective lines into quadrics, 
a preprint
%
%
%
\bibitem{Ohn}
Y.Ohnita, 
{\it Homogeneous Harmonic Maps into Complex Projective Spaces}, 
Tokyo Journal of Mathematics {\bf 13} (1990), 87--116 
%
%
%
%
\bibitem{Swan}
A.Swann, 
HyperK\"aler and quaternionic K\"ahler geometry, 
Math. Ann. {\bf 289} (1991), 421--450
%
\bibitem{Kob-Tak}
M.Takeuchi and S.Kobayashi, 
{\it Minimal imbeddings of R-spaces}, 
J. Differential Geometry. {\bf 2} (1968), 203--215
%
\bibitem{TTaka}
T.Takahashi, 
{\it Minimal immersions of Riemannian manifolds}, 
J. Math. Soc. Japan {\bf 18} (1966), 380-385
%
\bibitem{Tot}
G.Toth, 
{\it Moduli Spaces of Polynomial Minimal Immersions between Complex Projective Spaces}, 
Michigan Math.J. {\bf 37} (1990), 385--396
%
\bibitem{Wall}
N.R.Wallach, 
``Harmonic Analysis on Homogeneous Spaces", 
Pure and Applied Mathematics, 
Marcel Dekker, INC, New York (1973)
%
%
\bibitem{Wolf}
J,G.Wolfson 
{\it Harmonic maps of the two-sphere into the complex hyperquadric}, 
J. Diff. Geo. {\bf 24} (1986), 141-152
\end{thebibliography}
\end{document}